\documentclass[11pt, reqno]{amsart}

\usepackage{preamble}

\title{New Characterizations of First Order Sobolev Spaces}
\author{Przemysław Górka, Kacper Kurowski}
\date{}

\begin{document}
\begin{abstract}
We provide new characterizations of Sobolev spaces that are true under some mild conditions.
We study modified first order Sobolev spaces on metric measure spaces: $\testCurves$-Newtonian space, $\hat{\testCurves}$-Newtonian space, and Ambrosio-Gigli-Savaré-like space. We prove that if the measure is Borel regular and $\sigma$-finite, then the modified $\testCurves$-Newtonian space is equivalent to the Hajłasz--Sobolev space. Moreover, if additionally the measure is doubling then all modified spaces are equivalent to the Hajłasz--Sobolev space. 
\end{abstract}

\maketitle
 \tableofcontents

\section{Introduction}

In recent decades, several definitions of first order Sobolev spaces on metric measures spaces have been proposed. 
Hajłasz \cite{hajlasz_orig}, defined the so-called Hajłasz--Sobolev space $M^{1,p}(X)$ on a metric measure space $(X, \metricAlone, \measure)$ as the space of those $f \in L^p(\measure)$ for which there exists a nonnegative function $g \in L^p(\measure)$ such that the inequality 
\begin{equation*}
    \abs{ f(x) - f(y) }
    \le 
    \del{ g(x) + g(y) }
    \metric\del{x, y}
\end{equation*}
holds for almost every $x, y \in X$. Each function $g$ that satisfy the above inequality is called the Hajłasz gradients of $u$. 
Another proposition for a first order Sobolev space on a metric measure space is the Newtonian space $N^{1,p}(X)$ introduced by Shanmugalingam \cite{newtonian}. 
It is the space of all functions $f \in L^p(\measure)$ for which there exists a nonnegative Borel function $g \in L^p(\measure)$ such that the inequality 
\begin{equation*}
    \abs{ f(\gamma(a)) - f(\gamma(b)) }
    \le 
    \int_\gamma g
    \quad 
\end{equation*}
holds for  $p\text{-modulus almost every rectifiable } \gamma \colon [a,b] \to X$.  
Yet another definition of a first order Sobolev space on a metric measure space has been proposed by Ambrosio, Gigli and Savaré \cite{AGS, AGS1, gigli}. They defined this space as the space of all functions $f \in L^2(\measure)$ for which there exists a nonnegative function $g \in L^2(\measure)$ such that 
\begin{equation*}
    \int_{C\del{[0,1];X} }
        \abs{ f(\gamma(0)) - f(\gamma(1)) }
    \ \mathrm{d}\mu(\gamma)
    \le 
    \int_{C\del{[0,1];X} }
        \int_0^1 g(\gamma(t)) \abs{\dot{\gamma}}(t)
        \ \mathrm{d}t
    \ \mathrm{d}\mu(\gamma) 
\end{equation*}
for all test plans $\mu$, where $\abs{\dot{\gamma}}$ is the metric speed of $\gamma$. Other definitions of first order Sobolev spaces on metric measure space than the ones listed above have been introduced (see Cheeger \cite{cheeger} for instance). Nevertheless, in the paper we will focus only on those three spaces.

The advantage of $M^{1,p}$ spaces is that, unlike most other approaches, the theory is rich without assuming the measure is doubling or the space is connected \cite{hk, hajlasz_orig, HajCon}. It is well known that if the measure $\measure$ on the metric space is doubling and supports some Poincar\'{e} inequality then  Newtonian space $N^{1,p}$ and the Hajłasz--Sobolev space $M^{1,p}$ are equivalent \cite{newtonian, hajlaszGradientsAreUpper,Heinonen2015SobolevSO}.

 Let us make the following observations. The integral along the curve $\int_\gamma g$ is usually defined as a Lebesgue integral. However, we can equivalently treat it as a Lebesgue--Stieltjes integral\footnote{The proof of this statement is a part of Remark \ref{rem::continuous_gamma_then_symmetrized_integral_is_the_usual_integral}} 
$\int_\gamma g(\gamma(t)) \ \mathrm{d}\mu_\gamma(t)$.
The latter interpretation is more general as the latter integral is well-defined for all $\gamma \colon [a,b] \to X$ that are right-continuous and of bounded variation. 
Next, let us fix $x, y \in X$ and let $\gamma_x^y \colon [0,1] \to X$ be defined by $\gamma_x^y(t) = \indicator{[0,1/2)}(t) x + \indicator{[1/2,1]}(t) y$. 
Then $\int_{\gamma_x^y} g = g(y)\metric\del{x,y}$ for any Borel $g \colon X \to [0,\infty]$. 
The right hand side of this expression is similar to the right hand side of the definition of a Hajłasz gradient. 
We can make this resemblance even more apparent by ``symmetrizing'' the integral, that is, by taking the average of this integral and the integral along the curve $\gamma_y^x$: 
\begin{equation*}
    \frac{1}{2}\del{
        \int_{\gamma_x^y} g + \int_{\gamma_y^x} g
    }
    =
    \frac{1}{2} \del{g(x) + g(y)} \metric\del{x, y}.
\end{equation*}
In consequence, if we were to define a modification of  Newtonian spaces in which we use the symmetrized integral\footnote{We will make this notion rigorous in Definition \ref{def::symmetrized_integral}} instead of the usual integral along the curve, we might define function spaces that are highly comparable with the Hajłasz--Sobolev spaces.
Within this paper we explore this idea and apply similar modifications to the definition of the first order Sobolev spaces introduced by Ambrosio, Gigli and Savaré.

The main result of this paper, Theorem \ref{thm::final_comparison}, shows that, if the measure $\measure$ is Borel regular and $\sigma$-finite, the modified Newtonian space is equivalent to the Hajłasz--Sobolev space. Also, if, additionally, $\measure$ is doubling, the  ``Ambrosio-Gigli-Savaré-like'' space is equivalent to the Hajłasz--Sobolev space.
This theorem therefore provides new characterizations of the Hajłasz--Sobolev spaces that are true in rather general settings.

The reminder of the paper is structured as follows.
We devote Section \ref{sec::preliminaries} for the Preliminaries.
In Section \ref{sec::bounded_variation} we first recall some of the basic properties of functions of bounded variation from the interval $[a,b]$. Then we introduce the family of test curves $\testCurves\del{[a,b];X}$ and their reversal.
Finally, we introduce the symmetrized integral along curves from this family. 
In Section \ref{sec::topology_on_TC} we endow family $\testCurves\del{[a,b];X}$ with the topology of convergence in measure and discuss examples of functions that are Borel in this topology.
Section \ref{sec::function_spaces} is devoted to the introduction of three normed spaces that can be viewed as first order Sobolev spaces.
The ones introduced in Subsections \ref{subsec::TC-Newtonian} and \ref{subsec::gigli-like_space} are modifications of the Newtonian spaces and the Sobolev spaces introduced by Ambrosio, Gigli and Savaré, respectively. The one introduced in Subsection \ref{subsec::TC-hat_Newtonian_space} can be seen as a space whose definition is ``in-between'' the definition of the other two spaces.
Section \ref{sec::hajlasz_and_TC_Newtonian} is dedicated for the comparison between the Hajłasz--Sobolev spaces and the modified Newtonian spaces. Finally, in Section \ref{sec::final_comparison} we compare the Hajłasz--Sobolev spaces and the Ambrosio-Gigli-Savaré-like spaces. We also dedicate this section for Theorem \ref{thm::final_comparison} that summarizes all the comparisons between the previously metioned spaces.

\section{Preliminaries} \label{sec::preliminaries}
\noindent{\bf Notation}\\
We make the convention that $|\infty - \infty| = \infty$ and $|(-\infty) - (-\infty)| = \infty$.
Moreover, we let $[n]=\set{1,...,n}$ for $n \in \mathbb{N}$. 
We will use $\lambda$ to denote the Lebesgue measure. If $(X, \metricAlone)$ is a metric space, we will use $\mathcal{B}(X)$ to denote the family of Borel subsets of $X.$
\begin{proposition}
\label{mierzalnosc}
Let $(X_i, \metricAlone_i)$ be a metric space and $ \gamma_i \colon [a,b] \to X_i$ be Borel map such that $\image (\gamma_i)$ is separable, where $i\in [n]$. Then, $F_{\gamma_1,...,\gamma_n}: [a,b]^n \to \bigtimes_{i=1}^n X_i$  defined as $F_{\gamma_1,...,\gamma_n}(t_1,...,t_n)=(\gamma_1(t_1),...,\gamma_n(t_n))$ is a Borel map.
\end{proposition}
\begin{proof}
    Let
    $
        \tilde{X}_i 
        \coloneqq 
        \image\del{ \gamma_i},
    $
    then by assumption 
    we have $\tilde{X}_i$ is separable. Therefore, by the Lindel\"{o}f Theorem, for every open set $U \subset X^n$
    we have $F_{\gamma_1,...,\gamma_n}^{-1}[U]=F_{\gamma_1,...,\gamma_n}^{-1}[\bigtimes_{i=1}^n \tilde{X}_i \cap U]$ is a Borel subset of $[a,b]^n$. 
\end{proof}
Let us remark that assuming the continuum hypothesis $\bf{CH}$ we have separability of image of Borel maps.
\begin{remark}
\label{lem::image_of_Borel_from_interval_is_separable}
    Let $(X, \metricAlone)$ be a metric space and 
    $
        \gamma \colon [a,b] \to X
    $
    be a Borel map. Then, $\image (\gamma)$ is separable.
\end{remark}
\begin{proof}
      Suppose that the image of $\gamma$
    is not separable. 
    Then there exist $\eps > 0$ and an uncountable
    family 
    $
        \set{t_i}_{i \in I}
    $
    such that 
    $
        \set{ \gamma(t_i)}_{i \in  I}
    $
    is a $2\eps$-separated family of elements of $\image\del{\gamma}$.

    We have that 
    $
        \mathcal{B} 
        \coloneqq
        \set{ B\del{ \gamma(t_i), \eps }}_{i \in I}
    $ 
    is a family of pairwise disjoint balls.
    Each element of $\mathcal{B}$ is an open set,
    hence the union of any subfamily of $\mathcal{B}$
    is open. 
    Thus, every element of
    $
                \set{
            \bigcup_{j \in J}
            B\del{ \gamma(t_j), \eps }
            \given 
            J \in 2^I
        }
    $
    is open. 
    In consequence, every element of 
    $
        \mathcal{A}
        \coloneqq 
        \set{
            \gamma^{-1}\sbr{ 
                \bigcup_{j \in J}
            B\del{ \gamma(t_j), \eps }
            }
            \given 
            J \in 2^I
        }
    $
    is Borel in $[a,b]$. Therefore, 
    \[
        \# \mathcal{A} \leq \mathfrak{c}.
    \]
    Let us notice that for $J, J' \in 2^I$
    if $J \ne J'$, then
    $
        \set{ \gamma(t_j) }_{j \in J}
        \ne 
        \set{ \gamma(t_j) }_{j \in J'},
    $
    hence
    $$
        \gamma^{-1}\sbr{ 
                \bigcup_{j \in J}
            B\del{ \gamma(t_j), \eps }
            }
        \ne 
        \gamma^{-1}\sbr{ 
                \bigcup_{j \in J'}
            B\del{ \gamma(t_j), \eps }
            },
    $$
    as the balls are pairwise disjoint.
    This shows that $\# \mathcal{A}\geq \#2^I$.
    As $I$ is uncountable, assuming the continuum hypothesis
    we have $\# \mathcal{A} \ge 2^\mathfrak{c}$ and we get a contradiction.
\end{proof}

\section{Functions of bounded variation} \label{sec::bounded_variation}
\begin{definition}[Partition of interval]
\label{def:partition_of_interval}
        Let $[a,b] \subseteq \bR$. A tuple $\Delta = \del{t_i}_{i=0}^n$ shall be called a \emph{partition} of $[a,b]$ if $a = t_0 < \cdots < t_n = b$.   The family of partitions of $[a,b]$ shall be denoted by $\partition\del{ [a,b] }$. For partition $\Delta = \del{ t_i }_{i=0}^n$ we define its \emph{diameter} as:
        \begin{equation*}
       \abs{ \Delta } 
        \coloneqq 
        \max_{i \in [n]} 
            t_i - t_{i-1}.
    \end{equation*}
     If $(X, \metricAlone)$ is a metric space, $\gamma \colon [a,b] \to X$, and $\del{ t_i}_{i=0}^n = \Delta \in \partition \del{ [a,b] }$, then we define the \emph{$\Delta$-variation} of $\gamma$ by
    \begin{equation*}
        V^\Delta(\gamma)
        \coloneqq
            \sum_{i=1}^n
                    \metric\del{ \gamma(t_i), \gamma(t_{i-1}) }.
    \end{equation*}
    A sequence $\del{ \Delta_n }_n$ of partitions of $[a,b]$ will be called \emph{normal}, if 
    $
        \lim_{n \to \infty}
            \abs{\Delta_n}
        =
        0.
    $
\end{definition}

    While we defined partitions to be tuples of elements, we shall often work with them as if they are sets. For example, for $\Delta, \sigma \in \partition\del{ [a,b] }$ by $\Delta \cup \sigma$ we shall denote the unique element of $\partition\del{ [a,b] }$ which is a tuple consisting of all terms of $\Delta$ and $\sigma$ arranged in an increasing sequence. We will also use symbol $t \in \Delta$ to mean that $t$ is a term within tuple $\Delta$.
\begin{definition}[Variation of a function]
\label{def:variation_of_function}
    Let $(X, \metricAlone)$ be a metric space. For $\gamma \colon [a,b] \to X$ we define
    \begin{equation*}
        V(\gamma)
        =
        V_\gamma 
        \coloneqq 
        \sup_{
            \Delta \in \partition\del{ [a,b] }
        }
            V^\Delta(\gamma),
    \end{equation*}
    and the value of $V(\gamma)$ shall be called \emph{variation} of $\gamma$.
\end{definition}
\begin{definition}[Functions of bounded variation]
\label{def:functions_of_bounded_variation}
    Let $(X, \metricAlone)$ be a metric space. We shall say that $\gamma \colon [a,b] \to X$ is of \emph{bounded variation} if $V(\gamma) < \infty$. We shall denote by $\boundedVariation \! \del{[a,b]; X}$ the family of all functions $\gamma \colon [a,b] \to X$ of bounded variation. We also define
    \begin{equation*}
        \boundedVariationRC\del{ [a,b]; X }
        \coloneqq 
        \set{
            \gamma \in \boundedVariation\del{ [a,b]; X }
            \given
            \gamma \text{ is right-continuous on}
            \intco{a, b}
        }.
    \end{equation*}
\end{definition}
\begin{remark}
\label{rem:restricting_functions_of_bounded_variation}
    For all $r, t \in [a,b]$ with $r \le t$, if $\gamma \in \boundedVariation\del{ [a,b]; X}$ ($\gamma \in \boundedVariationRC\del{ [a,b]; X}$), then $\gamma \rvert_{[r,t]} \in \boundedVariation\del{ [r,t]; X}$ ($\gamma \rvert_{[r,t]} \in \boundedVariationRC\del{ [r,t]; X}$).
\end{remark}
\begin{proof}
    Let $\Delta \in \partition\del{ [r,t] }$, then 
    \begin{equation*}
        V^\Delta \del{
            \gamma  \rvert_{[r,t]}      
        }
        \le 
        V^{\Delta \cup \set{a,b} } \del{
            \gamma        
        }
        \le 
        V(\gamma).
    \end{equation*}
    Taking supremum over $\Delta \in \partition\del{ [r,t] }$ we see $V(\gamma \rvert_{[r,t]} ) \le V(\gamma) < \infty$, so $\gamma \rvert_{[r,t]} \in \boundedVariation\del{ [a,b]; X}$. The other claim follows from the fact that restricting functions preserves their right-continuity.
\end{proof}
\begin{lemma}
\label{lem:functions_of_bounded_variation_have_cauchy_limits_everywhere}
    Let $(X, \metricAlone)$ be a metric space and $\gamma \in \boundedVariation\del{ [a,b]; X}$.
    Then
    \begin{enumerate}[itemindent=2em,leftmargin=0em, label=\roman*), itemsep=1em]
        \item
    \begin{equation}
    \label{lem:functions_of_bounded_variation_have_cauchy_limits_everywhere::eq:almost_right_limits}
        \forall t \in \intco{a,b}
        \quad 
        \forall \eps > 0
        \quad 
        \exists \delta > 0
        \quad \forall r, s \in (t, t+\delta)
        \qquad 
        \metric \del{
            \gamma(r),\gamma(s)
        } \le \eps 
    \end{equation}
    and
    \begin{equation}
    \label{lem:functions_of_bounded_variation_have_cauchy_limits_everywhere::eq:almost_left_limits}
        \forall t \in \intoc{a,b}
         \quad 
        \forall \eps > 0
        \quad 
        \exists \delta > 0
        \quad \forall r, s \in (t-\delta, t)
        \qquad 
        \metric \del{
            \gamma(r),\gamma(s)
        } \le \eps.
    \end{equation}

    In consequence, for all $t \in [a,b]$, if $t_n \to t^+$ (and $t \in \intco{a,b})$ or $t_n \to t^-$ (and $t \in \intoc{a,b}$), then $\del{\gamma(t_n)}_n$ is a Cauchy sequence. 
    Moreover, if $t_n \to t^-$ (or $t^+$) and $s_n \to t^-$ (or $t^+$)
    (respectively), then if $(r_n)$ is a sequence which alternates between $(t_n)$ and $(s_n)$ then $\del{\gamma(r_n)}_n$ is a Cauchy sequence. 
    \item
    \begin{equation}
    \label{cor:bounded_variation_limits_of_distances_exist::eq:main:right_limit_of_metric}
        \forall r \in \intco{a,b}
        \quad 
            \forall t \in \intcc{a,b}
        \qquad 
        \lim_{s \to r^+}
            \metric\del{
                \gamma(s),
                \gamma(t)
            }
        \text{ exists}
    \end{equation}
    and
    \begin{equation}
    \label{cor:bounded_variation_limits_of_distances_exist::eq:main:left_limit_of_metric}
        \forall r \in \intoc{a,b}
        \quad 
            \forall t \in \intcc{a,b}
        \qquad 
        \lim_{s \to r^-}
            \metric\del{
                \gamma(s),
                \gamma(t)
            }
        \text{ exists}.
    \end{equation}
    \end{enumerate}
\end{lemma}
\begin{proof}    
    \vspace{-1em}
    \begin{enumerate}[itemindent=2em,leftmargin=0em, label=\roman*), itemsep=0.5em]
        \item Without loss of generality, it suffices to prove only one of \eqref{lem:functions_of_bounded_variation_have_cauchy_limits_everywhere::eq:almost_right_limits} and
        \eqref{lem:functions_of_bounded_variation_have_cauchy_limits_everywhere::eq:almost_left_limits}, as the other can
        be proved analogously.
        We will show \eqref{lem:functions_of_bounded_variation_have_cauchy_limits_everywhere::eq:almost_left_limits}. For this purpose we fix $t \in \intoc{a,b}$ and suppose the claim is false. Then, there exist $\eps > 0$ and sequences $(s_n)_n$ and $(r_n)_n$ such that
         \begin{equation*}
                \forall n \in \bN 
                \qquad 
                s_n, r_n \in \del{ a, t},
                \quad 
                s_n < r_n < s_{n+1},
                \quad \text{ and } \quad 
                \metric \del{
                    \gamma(s_n), \gamma(r_n)
                }
                \ge \eps.
         \end{equation*}

        Let us define partitions 
        $
            \Delta_k = \del{ a, s_{1}, r_{1}, s_{2}, \ldots, s_{k-1},
            r_{k-1}, s_{k}, b
            }
        $, where $k \in \mathbb{N}$.
        Then, for all $k \in \bN$ we have
        \begin{align*}
            V^{\Delta_k}
            &=
            \metric \del{
                    \gamma(a), \gamma(s_{1})
                }
            + \sum_{i=1}^{k-1}
                \del{
                    \metric \del{
                        \gamma(s_{i}), \gamma(r_{i})
                    }
                    +\metric \del{
                        \gamma(r_{i}), \gamma(s_{i+1})
                    }
                }
            +
            \metric \del{
                    \gamma(s_{k}), \gamma(b)
                }
            \\
            &\ge 
             \sum_{i=1}^{k-1}
                    \metric \del{
                        \gamma(s_{i}), \gamma(r_{i})
                    }
            \\
            &\ge (k-1)\eps,
        \end{align*}
        so $V^{\Delta_k} \to \infty$ as $k \to \infty$. 
        This contradicts $\gamma \in \boundedVariation\del{ [a,b]; X}$. Hence, the claim is proved.
         
        For the claim with Cauchy sequences, it is again sufficient to 
        prove it for one of the sides.
        Let $t \in \intoc{a,b}$ and $t_n \to t^-$.
        Let $\eps > 0$. There exists $\delta > 0$ as in \eqref{lem:functions_of_bounded_variation_have_cauchy_limits_everywhere::eq:almost_left_limits}.
        Then there exists $N \in \bN$ such that for $n, m \ge N$
        we have 
        $
            \metric\del{ \gamma( t_n), \gamma(t_m) } \le \eps 
        $
        and the claim is proved. The last  claim follows from the fact that a sequence alternating between the two still approaches from below (or above).
        \item 
        Without loss of generality, it suffices to prove only one of 
        \eqref{cor:bounded_variation_limits_of_distances_exist::eq:main:right_limit_of_metric}
        and 
        \eqref{cor:bounded_variation_limits_of_distances_exist::eq:main:left_limit_of_metric}
        as the other can be proved analogously. 
        We will show \eqref{cor:bounded_variation_limits_of_distances_exist::eq:main:right_limit_of_metric}. Let $r \in \intco{a,b}$, $t \in [a,b]$ and $\eps > 0$.
        By property \eqref{lem:functions_of_bounded_variation_have_cauchy_limits_everywhere::eq:almost_right_limits} from Lemma
        \ref{lem:functions_of_bounded_variation_have_cauchy_limits_everywhere}
        there exists $\delta > 0$ such that 
        if $s_1, s_2 \in (r, r+\delta)$, then
        $
        \metric\del{ \gamma(s_1), \gamma(s_2) }\le \eps. 
        $
        Therefore
        \begin{equation*}
            \abs{
                \metric\del{ \gamma(s_1), \gamma(t) }
                - \metric\del{ \gamma(t), \gamma(s_2) }
            }
            \le 
            \metric\del{ \gamma(s_1), \gamma(s_2) }
            \le 
            \eps
        \end{equation*}
        and since $\eps > 0$ is arbitrary, the claim is proved.
    \end{enumerate}
\end{proof}
\begin{corollary}[Image of bounded variation function is totally bounded]
\label{cor:functions_of_bounded_variation_image_totally_bounded}
    Let $(X, \metricAlone)$ be a metric space and $\gamma \in \boundedVariation\del{ [a,b]; X}$. Then $\image\del{\gamma}$ is totally bounded and $\diam\del{ \image\del{ \gamma}} \leq V(\gamma)$. 
\end{corollary}
\begin{proof}
    It is sufficient to show that every sequence in $\image\del{\gamma}$ has a Cauchy subsequence. Let $(\gamma(t_n))_n$ be a sequence in $\image\del{\gamma}$, where $t_n \in [a,b]$. By compactness of $[a,b]$, there exist $t\in [a,b]$ and a subsequence $(t_{n_k})_k$ which converges to $t$ in such a way that all of its terms are either:
    strictly smaller than $t$, strictly greater than $t$, or equal to $t$. 
    In either of these cases, sequence $\del{\gamma\del{ t_{n_{k}}}}_k$ is Cauchy by Lemma \ref{lem:functions_of_bounded_variation_have_cauchy_limits_everywhere} or by being a constant sequence.

    Now, for the second part. Let $x, y \in \image\del{ \gamma}$, then from the very definition of variation of $\gamma$ we have
    \begin{equation*}
                \metric\del{ x, y} \leq V(\gamma).
    \end{equation*}
    Therefore, taking supremum over $x, y \in \image\del{ \gamma}$, we have $ \diam\del{ \image\del{ \gamma}}\leq V(\gamma)$.    
\end{proof}

\begin{definition}[Left-jump, right-jump]
\label{def:left-jump_right-jump}
    Let $(X, \metricAlone)$ be a metric space and $\gamma \in \boundedVariation\del{ [a,b]; X}$.
    We define functions 
    $
        \phi_\gamma^L, \phi_\gamma^R 
        \colon 
        [a,b]
        \to 
        \intco{0. \infty}
    $
    of the \emph{left-jumps} and the \emph{right-jumps} of $\gamma$
    by the formulas
    \begin{equation*}
       \forall t \in [a,b]
        \qquad 
        \phi_\gamma^L(t)
        \coloneqq 
        \lim_{s \to t^-}
            \metric\del{
                \gamma(s), \gamma(t)
            }
        \quad \text{ and } \quad 
        \phi_\gamma^R(t)
        \coloneqq 
        \lim_{s \to t^+}
            \metric\del{
                \gamma(s), \gamma(t)
            },
    \end{equation*}
    where we put $\phi_\gamma^L(a) = 0$ and $\phi^R_\gamma(b) = 0$.
    
    Functions $\phi_\gamma^L$ and $\phi_\gamma^R$
    are well-defined by Lemma
    \ref{lem:functions_of_bounded_variation_have_cauchy_limits_everywhere}.
    In the case of 
    $\gamma \in \boundedVariationRC\del{ [a,b]; X}$
    we have $\phi^R_\gamma \equiv 0$;
    in this case we will also simplify our 
    notation by writing $\phi_\gamma$
    instead of $\phi^L_\gamma$.
\end{definition}
\begin{remark}
\label{rem:L_and_R_are_the_sets_of_left_and_right_discontinuities_of_gamma}
    Let $(X, \metricAlone)$ be a metric space and $\gamma \in \boundedVariation\del{ [a,b]; X}$.
    Then sets
    \begin{equation*}
        L_{\gamma}
        \coloneq 
        \set{
            t \in [a,b]
            \given 
            \phi_\gamma^L(t) > 0
        }
        \quad \text{ and } \quad 
        R_{\gamma}
        \coloneq 
        \set{
            t \in [a,b]
            \given 
            \phi_\gamma^R(t) > 0
        }
    \end{equation*}
    are precisely the sets of points of left-
    and right-discontinuity of $\gamma$.
\end{remark}
\begin{lemma}
\label{cor:variation_bounds_sum_of_jumps_from_above}
    Let $(X, \metricAlone)$ be a metric space and $\gamma \in \boundedVariation\del{ [a,b]; X}$.
    Then\footnote{If $g:[a,b]\rightarrow [0,\infty]$, then $\sum_{t \in [a,b]} g(t) := \mathop{\sup}_{\substack{K\subset [a,b],\\ \#K<\infty}} \sum_{t \in K} g(t)$.} 
    \begin{equation}
        \label{nier}
        \sum_{t \in [a,b]}
            \phi_\gamma^L(t)
            + \phi_\gamma^R(t)
        \le 
        V(\gamma).
    \end{equation}
    Moreover, the set of points at which $\gamma$ is not
    continuous is at most countable.  
\end{lemma}
\begin{proof}
    Let $A= \del{s_i}_{i=1}^k$ be a finite subset of $[a,b]$. We arrange $s_i$ in ascending order, that is, $s_1 < s_2 < \cdots < s_k$. For $\eps > 0$ there exist $\del{r_i}_{i=1}^k \subset [a,b]$ and $\del{t_i}_{i=1}^k \subset [a,b]$ such that
    \begin{itemize}
        \item 
        $r_1 \in (a,s_1)$ if $s_1 \neq a$ or $r_1 = a$ if $s_1 = a$,
        \item 
        $t_k \in (s_k,b)$ if $s_k \neq b$ or $t_k = b$ if $s_k = b$,
        \item 
        For all $i \in [k-1]$ we have
        $s_i < t_i < r_{i+1} < s_{i+1}$
        \item 
        For all $i \in [k]$ we have:
        $
            \metric\del{
                \gamma(s_i),
                \gamma(r_i)
            }
            \ge 
            \phi^L_\gamma(s_i) - \frac{\eps}{2k}
        $
        and 
        $
            \metric\del{
                \gamma(s_i),
                \gamma(t_i)
            }
            \ge 
            \phi^R_\gamma(s_i) - \frac{\eps}{2k}
        $.
    \end{itemize}
   From the definition of the variation of $\gamma$ we have
    \begin{align*}
        V(\gamma)
        &\ge 
        \sum_{i=1}^k
            \del{
                \metric\del{
                    \gamma(r_i),
                    \gamma(s_i)
                }
                +
                \metric\del{
                    \gamma(s_i),
                    \gamma(t_i)
                }
            }
        \\ 
        &\ge 
        \sum_{i=1}^k
            \del{
                \phi_\gamma^L(s_i)-
                \frac{\eps}{2k} +
                \phi_\gamma^R(s_i)-
                \frac{\eps}{2k}
            }
        \\
        &=
        \sum_{i=1}^k
            \del{
                \phi_\gamma^L(s_i)
            +\phi_\gamma^R(s_i)}
        -
        \eps.
    \end{align*}
  Therefore, since $\eps$ is arbitrary and since $A$ is an arbitrary finite subset of $[a,b]$, we have
    \begin{equation*}
         V(\gamma)
         \ge 
        \sum_{t \in [a,b] }
            \phi^L_\gamma(t)
            +
            \phi_\gamma^R(t)
    \end{equation*}
    as needed. 

    Next, we prove that the set of points at which $\gamma$ is not
    continuous is at most countable. For this purpose it is enough to prove that $L_{\gamma}$ and $R_{\gamma}$ are at most countable. We shall prove that $L_{\gamma}$ is at most countable. Since we have $L_{\gamma} = \bigcup_{n=1}^{\infty}L_n$, where 
    \begin{equation*}
        L_n
        \coloneqq 
        \set{
            t \in [a,b]
            \given
            \phi^L_\gamma(t) \ge 
                \frac{1}{n}
        },
    \end{equation*}
    it is enough to show that $L_n$ is a finite set for every $n$. Let us suppose there is $N \in \bN$
     such that $L_N$ is not a finite set.  Then, for each $k \in \bN$ we can select $k$ elements 
     $t_1^k, \ldots, t_k^k \in L_N$. Therefore, by (\ref{nier}) we have
     \begin{equation*}
         V(\gamma)
         \ge 
        \sum_{n=1}^k
            \phi^L_\gamma(t^k_n)\ge \sum_{n=1}^k \frac{n}{N} = \frac{k}{N}.
    \end{equation*}
      In this way we get $V(\gamma) =\infty$. 
     However, that contradicts $\gamma \in \boundedVariation\del{ [a,b]; X}$.
     Hence, the claim is proved.
\end{proof}

\begin{proposition}
\label{cor:variation_is_a_uniform_limit_of_Delta-variations}
    Let $(X, \metricAlone)$ be a metric space and $\gamma \in \boundedVariationRC\del{ [a,b]; X}$. Then 
    \begin{equation*}
        V(\gamma)
        =
        \lim_{ 
            \substack{ 
                \abs{ \Delta } \to 0  \\
                \Delta \in \partition\del{ [a,b] }
            }
        }
            V^\Delta(\gamma).
    \end{equation*}
\end{proposition}
\begin{proof}
    First of all we shall prove the lemma
\begin{lemma}
\label{lem:substitution_for_uniform_continuity_for_RC_bounded_variation}
    Let $(X, \metricAlone)$ be a metric space and $\gamma \in \boundedVariationRC\del{ [a,b]; X}$. Then for every $\eps > 0$ there exists $\delta > 0$ such that for all $r, t \in [a,b], r < t$
    with $\abs{ t - r } \le \delta$ we have
       \begin{equation*}
         \sup_{s \in [r,t]}\metric\del{ \gamma(r), \gamma(s) }
            + \metric\del{ \gamma(s), \gamma(t) }
            \le 
            \metric\del{ \gamma(r), \gamma(t) }
            + \eps.
        \end{equation*} 
\end{lemma}
\begin{proof}
    Suppose the thesis is false. Then there exists $\eps > 0$ such that for all $n \in \bN$ there are $r_n, t_n \in [a,b]$ such that  $r_n < t_n$, 
    \begin{equation}
        \label{lem:almost_uniform_continuity::eq:distance_from_rn_to_rn}
        \abs{t_n - r_n } \le 1/n,
    \end{equation}
    and there exists $s_n \in [r_n, t_n]$ such that
    \begin{equation}         
        \label{lem:almost_uniform_continuity::eq:false_inequalities_from_thesis}
        \metric\del{ \gamma(r_n), \gamma(s_n) }
        + \metric\del{ \gamma(s_n), \gamma(t_n) }
        > 
        \metric\del{ \gamma(r_n), \gamma(t_n) }
        + \eps.
    \end{equation}
    Taking into account \eqref{lem:almost_uniform_continuity::eq:distance_from_rn_to_rn} with compactness of $[a,b]$, each of $(r_n)_n, (t_n)_n, $ and $(s_n)_n$ has convergent subsequence $(r_{n_k})_k, (t_{n_k})_k$, and $(s_{n_k})_k$ converging to $\tau$.

    We have three possibilities:
    \begin{enumerate}[label=(\alph*), itemindent=*]
        \item 
        Sequence $(r_{n_k})_k$ has a further subsequence $(r_{n_{k_l}})_l$ such that for all $l \in \bN$ we have $\tau \le r_{n_{k_l}}$. 

        We then have $r_{n_{k_l}}, s_{n_{k_l}}, t_{n_{k_l}} \to \tau^+$ and from the right-continuity of $\gamma$ at $\tau$, we have
        \begin{equation*}
            \metric\del{ \gamma(r_{n_{k_l}}), \gamma(s_{n_{k_l}}) }
            + \metric\del{ \gamma(s_{n_{k_l}}), \gamma(t_{n_{k_l}}) }
            \xrightarrow{k \to \infty}
            0, \quad 
            \metric\del{ \gamma(r_{n_{k_l}}), \gamma(t_{n_{k_l}}) }
            \xrightarrow{k \to \infty}
            0.
        \end{equation*}
        However, this contradicts \eqref{lem:almost_uniform_continuity::eq:false_inequalities_from_thesis}.

        \item 
        Sequence $(t_{n_k})_k$ has a further subsequence $(t_{n_{k_l}})_l$ such that for all $l \in \bN$ we have $t_{n_{k_l}} < \tau$.

        We have $r_{n_{k_l}}, s_{n_{k_l}}, t_{n_{k_l}} \to \tau^-$. Let $(\tau_l)_l$ be a sequence whose terms alternate between the terms of the other three other sequences. By Lemma \ref{lem:functions_of_bounded_variation_have_cauchy_limits_everywhere} we know that  $\del{ \gamma\del{\tau_l}}_l$ is a Cauchy sequence. In consequence,
        \begin{equation*}
            \metric\del{ \gamma(r_{n_{k_l}}), \gamma(s_{n_{k_l}}) }
            + \metric\del{ \gamma(s_{n_{k_l}}), \gamma(t_{n_{k_l}}) }
            \xrightarrow{k \to \infty}
            0, \quad 
            \metric\del{ \gamma(r_{n_{k_l}}), \gamma(t_{n_{k_l}}) }
            \xrightarrow{k \to \infty}
            0,
        \end{equation*}
        and this contradicts \eqref{lem:almost_uniform_continuity::eq:false_inequalities_from_thesis}.
        \item 
        Neither of the mentioned cases is true. Then, for all large $k \in \bN$ we have $r_{n_k} < \tau \le t_{n_k}$. We then have two possibilities:
        \begin{itemize}
            \item 
                Sequence $(s_{n_k})_k$ has a subsequence $(s_{n_{k_l}})_l$ such that for all $l \in \bN$ we have $s_{n_{k_l}} < \tau$. 

                Let $(\tau_l)_l$ be a sequence which terms alternate between the terms of $\del[1]{r_{n_{k_l}}}_l$ and $\del[1]{ s_{n_{k_l}}}_l$. By Lemma \ref{lem:functions_of_bounded_variation_have_cauchy_limits_everywhere} sequence $\del{ \gamma \del{ \tau_l}}_l$ is Cauchy. In consequence,
                \begin{equation*}
                    \left|-\metric\del{ \gamma(r_{n_{k_l}}), \gamma(t_{n_{k_l}}) }
                    +\metric\del{ \gamma(s_{n_{k_l}}), \gamma(t_{n_{k_l}}) }\right|
                    \le 
                    \metric\del{ \gamma(r_{n_{k_l}}), \gamma(s_{n_{k_l}}) }
                    \xrightarrow{k \to \infty}
                    0,
                \end{equation*}
                and this contradicts \eqref{lem:almost_uniform_continuity::eq:false_inequalities_from_thesis}.

            \item 
                Sequence $(s_{n_k})_k$ has a subsequence $(s_{n_{k_l}})_l$ such that for all $l \in \bN$ we have $\tau \le s_{n_{k_l}}$. 

                Let $(\tau_l)_l$ be a sequence which terms alternate between the terms of $\del[1]{t_{n_{k_l}}}_l$ and $\del[1]{ s_{n_{k_l}}}_l$. By Lemma \ref{lem:functions_of_bounded_variation_have_cauchy_limits_everywhere} sequence $\del{ \gamma \del{ \tau_l}}_l$ is Cauchy. In consequence,
                \begin{equation*}
                    \left|\metric\del{ \gamma(r_{n_{k_l}}), \gamma(t_{n_{k_l}}) }
                    - \metric\del{ \gamma(r_{n_{k_l}}), \gamma(s_{n_{k_l}}) }\right|
                    \le 
                    \metric\del{ \gamma(s_{n_{k_l}}), \gamma(t_{n_{k_l}}) }
                    \xrightarrow{k \to \infty}
                    0,
                \end{equation*}
                and once again, this contradicts \eqref{lem:almost_uniform_continuity::eq:false_inequalities_from_thesis}.
        \end{itemize}
    \end{enumerate}
    We see that in all cases we have arrived at a contradiction. As such, the thesis of the lemma is true.
\end{proof}
     Now, we are in position to prove the proposition. It is sufficient to show that for any $M > 0$ such that $M < V(\gamma)$ there exists $\delta > 0$ such that if $\Delta \in \partition\del{ [a,b] }$ satisfies $\abs{\Delta} \le \delta $, then $M \le V^\Delta(\gamma)$.
    
    Since $\gamma \in \boundedVariationRC\del{ [a,b]; X}$, we have $V(\gamma) < \infty$. 
    Let $\eps \coloneqq \del{ V(\gamma) - M }/2$.
    There exists $\del{\tau_j}_{j=0}^m = \sigma \in \partition\del{ [a,b] }$ such that $V(\gamma) - V^\sigma(\gamma) \le \eps$. 
    Let $\delta_1 \coloneqq \min_{j \in [m] } \del{\tau_j - \tau_{j-1} }/2$. By Lemma \ref{lem:substitution_for_uniform_continuity_for_RC_bounded_variation} there exists $\delta \in \intoo{ 0, \delta_1 }$ such that for all $r, t \in [a,b], r < t$ with $\abs{ t - r } \le \delta$
    we have 
    \begin{equation*}
        \forall s \in [r,t]
        \qquad 
            \metric\del{ \gamma(r), \gamma(s) }
            + \metric\del{ \gamma(s), \gamma(t) }
            \le 
            \metric\del{ \gamma(r), \gamma(t) }
            + \eps/(m-1).
    \end{equation*} 

    Let $(t_i)_{i=0}^n = \Delta \in \partition\del{ [a,b] }$ be any such that $\abs{ \Delta } \le \delta$. For $j \in [m-1]$ let us denote
    \begin{align*}
        {}^< \tau_j
        \coloneqq 
        \max 
            \set{
                t \in \Delta 
                \given 
                t < \tau_j
            }, \quad         
        \tau_j^{\le}
        \coloneqq 
        \min
            \set{
                t \in \Delta 
                \given 
                \tau_j \le t
            }.
    \end{align*}
    By definition, ${}^< \tau_j$ and $\tau_j^{\le}$ are consecutive elements of $\Delta$ and therefore $\abs[0]{ \tau_j^{\le} - {}^< \tau_j } \le \delta$. Since $\tau_j \in [{}^< \tau_j,  \tau_j^\le]$, by definition of $\delta$ we have 
    \begin{equation*}
    \label{cor:variation_is_a_uniform_limit_of_Delta-variations::eq:consequence_of_substitute_for_uinform_conitnuity}
         \metric\del{ \gamma({}^< \tau_j), \gamma(\tau_j) }
            + \metric\del{ \gamma(\tau_j), \gamma(\tau_j^\le) }
        \le 
            \metric\del{ \gamma({}^< \tau_j), \gamma(\tau_j^\le) }
            + \eps/(m-1).
    \end{equation*}
    From $\abs[0]{ \tau_j^{\le} - {}^< \tau_j } \le \delta$ and the fact that $\delta < \delta_1 = \min_{j \in [m] } \del{\tau_j - \tau_{j-1} }/2$, we also have 
    \begin{equation*}
        {}^< \tau_j 
        \enskip < \enskip
            \tau_j
        \enskip \le \enskip
            \tau_j^\le 
        \enskip \le \enskip
            {}^< \tau_{j+1}
        \enskip < \enskip
            \tau_{j+1}
        \enskip \le \enskip
            \tau_{j+1}^{\le}
    \end{equation*}
    for all $j \in [m-2]$. Let us denote by $I$ the family of indices $i$ of $\set{0, 1, \ldots, n}$ such that there is $j \in [m-1]$ with $t_i = \tau_j^{\leq}$. Thus,
    \begin{align*}
        V^{ \sigma \cup \Delta }(\gamma)
        &=
        \sum_{ 
            \substack{ i = 1 \\ i \notin I 
            }
        }^n
            \metric\del{
                \gamma \del{ t_i },
                \gamma \del{ t_{i-1}  }
            }
        +
        \sum_{ j=1 }^{m-1}
            \del{
                \metric\del{
                    \gamma \del{ {}^< \tau_j},
                    \gamma \del{ \tau_j   }
                }
                + \metric\del{
                    \gamma \del{ \tau_j},
                    \gamma \del{ \tau_j^\le   }
                }
            }
        \\
                &\le 
        \sum_{ 
            \substack{ i = 1 \\ i \notin I 
            }
        }^n
            \metric\del{
                \gamma \del{ t_i },
                \gamma \del{ t_{i-1}  }
            }
        +
        \sum_{ j=1 }^{m-1}
            \del{
                \metric\del{ \gamma \del{ \tau_j^\le   }, 
                    \gamma \del{ {}^< \tau_j}                   
                }
                + 
                \eps/(m-1)
            }
        \\ 
        &=
        \sum_{i=1}^{n}
             \metric\del{
                \gamma \del{ t_i },
                \gamma \del{ t_{i-1}  }
            }
        +
        \eps
        \\
        &=
        V^\Delta(\gamma) 
        +
        \eps.
    \end{align*}
    Terefore, by the definition of $\sigma$, we have
    \begin{equation*}
        V(\gamma) 
        \le 
        V^\sigma(\gamma) + \eps 
        \le 
        V^{ \sigma \cup \Delta }(\gamma) + \eps 
        \le 
        V^\Delta(\gamma) + 2 \eps.
    \end{equation*}
    Finally, since $\eps = \del{ V(\gamma) - M }/2$, we conclude
    \begin{equation*}
        M 
        =
        V(\gamma)
        - 2 \eps 
        \le V^{\Delta}(\gamma)
    \end{equation*}
    and $M \le V^{\Delta}(\gamma)$ for any $\Delta \in \partition\del{ [a,b] }$ with $\abs{ \Delta} \le \delta$, as needed.
\end{proof}
As a corollary we have.
\begin{corollary}[Additivity of variation]
\label{cor:variation_is_additive}
    Let $(X, \metricAlone)$ be a metric space and $\gamma \in \boundedVariationRC\del{ [a,b]; X}$. Then for $t \in [a,b]$ we have
    \begin{equation*}
        V\del{ \gamma \rvert_{ [a,t] } }
        + V\del{ \gamma \rvert_{ [t,b] } }
        =
        V\del{ \gamma }.
    \end{equation*}
\end{corollary}
\begin{definition}[Variation function]
\label{def:variation_function}
    Let $(X, \metricAlone)$ be a metric space. For $\gamma \in \boundedVariationRC\del{ [a,b]; X }$ we define function $V_\gamma \colon [a,b] \to \intco{0, \infty}$ by the formula
    \begin{equation*}
        \forall t \in [a,b]
        \qquad 
        V_\gamma(t)
        \coloneqq 
        V\del{ \gamma \rvert_{[a, t]} }.
    \end{equation*}
\end{definition}
\begin{proposition}[Properties of variation function]
\label{rem:properties_of_variation_function}
    For $\gamma \in \boundedVariationRC\del{ [a,b]; X }$ function $V_\gamma \colon [a,b] \to \intco{0, \infty}$ is non-negative, non-decreasing and bounded by $V(\gamma)$. Furthermore, $V_\gamma \in \boundedVariationRC\del{ [a,b]; \intco{0, \infty} }$ and 
    \begin{equation*}
        \forall t \in [a,b]
        \qquad 
        V_{V_\gamma}\del{ t }
        =
        V_\gamma(t).
    \end{equation*}
    Finally, we have
    \begin{equation*}
        \forall t \in \intoc{a,b}
        \qquad 
            \lim_{s \to t^-} V_{\gamma}(t) - V_{\gamma}(s)
        =
            \phi_{\gamma}(t).
    \end{equation*}
   Hence, $V_\gamma$ is left-continuous at points of left-continuity of $\gamma$. 
\end{proposition}
\begin{proof}
    Suppose $s, t \in [a,b]$ are such that $s \le t$. Let $\eps > 0$ and $\Delta_s \in \partition\del{ [a,s] }$ be such that $V_\gamma(s) \le V^{\Delta_s}( \gamma  \rvert_{[a, s]} ) + \eps$. Then $\Delta_t \coloneqq \Delta_s \cup \set{t} \in \partition\del{ [a,t] }$ and $V^{\Delta_s}( \gamma  \rvert_{[a, s]} ) \le V^{\Delta_t}( \gamma  \rvert_{[a, t]} ) \le V_\gamma(t)$. Hence, $V_\gamma(s) \le V_\gamma(t) + \eps$ for all $\eps > 0$. Thus, $V_\gamma(s) \le V_\gamma(t)$ for $s \le t$ and $V_\gamma$ is non-decreasing. In consequence, $V_\gamma$ is bounded by $V_\gamma(b) = V(\gamma)$.

    Let us prove that $V_\gamma \in \boundedVariation\del{ [a,b]; \intco{0, \infty} }$. For $\del{t_i}_{i=0}^n = \Delta \in \partition\del{ [a,b] }$ we have
    \begin{equation*}
        V^\Delta \del{ {V_\gamma} }
        =
        \sum_{i=1}^n
                V_\gamma(t_i)
                - V_\gamma\del{ t_{i-1} }            
        =
        V_\gamma(t_n) - V_\gamma(t_0)
        =
        V(\gamma).
    \end{equation*}
    Thus, by the definition of the variation we have $V(V_\gamma) = V_\gamma$. 
    Furthermore, by Remark \ref{rem:restricting_functions_of_bounded_variation} we have $\gamma \rvert_{ [a,t] } \in \boundedVariationRC\del{ [a,t]; X}$ for all $t \in [a,b] $,  and in consequence, 
    \begin{equation*}
        V_{V_\gamma}\del{ t }
        =
        V\del{ 
            V_{ \gamma \rvert_{[a,t] } } }
        =
        V \del{ \gamma \rvert_{[a,t] } }
        =
        V_{\gamma}(t).
    \end{equation*}
    Next, we shall show that $V_\gamma$ is right-continuous. For this purpose we fix $s \in \intco{a,b}$ and suppose that $V_\gamma$ is not right-continuous at $s$. 
    Then there exists $\eps > 0$ such that for all $t \in \intoc{s, b}$ we have 
    \begin{equation} \label{nier1}
             2\eps \le V_{\gamma}(t) - V_{\gamma}(s) = V\del{ \gamma \rvert_{[ s, t ]} },
    \end{equation}
    where Corollary \ref{cor:variation_is_additive} was applied. For $n \in \bN$ let $\del{ t_i^n}_{i=0}^{m_n} = \Delta_n \in \partition([s, b])$ be such that $\abs{ t_1^n-t_0^n } \le 1/n$ and 
    $$
            V\del{ \gamma \rvert_{[ s, b ]}} - V^{\Delta_n}\del{ \gamma \rvert_{[ s, b ] } } \le \eps.
    $$ 
    Let $s_n \coloneqq t_1^n$, then  
    \begin{equation*}
        V\del{
            \gamma \rvert_{[ s, s_n ] }
        }
        +V\del{
            \gamma \rvert_{[ s_n, b ] }
        }
        =
        V\del{
            \gamma \rvert_{[ s, b ] }
        }
        \le 
        V^{\Delta_n}\del{ \gamma \rvert_{[ s, b ] } } +\eps
        \le 
        \metric\del{
            \gamma(s), \gamma(s_n)
        }
        +V\del{
            \gamma \rvert_{[ s_n, b ] }
        }
        + 
        \eps,
    \end{equation*}
    hence
    \begin{equation*}
        V\del{
            \gamma \rvert_{[ s, s_n ] }
        }
        \le 
        \metric\del{
            \gamma(s), \gamma(s_n)
        }
        + \eps.
    \end{equation*}
    Therefore, by (\ref{nier1}) we get
    $\eps \le \metric\del{
            \gamma(s), \gamma(s_n)
        }$. 
       However, since $s_n \to s^+$, this inequality contradicts the right-continuity of $\gamma$. Thus, $V_\gamma$ is right-continuous. 

    Finally, let us prove that
    \begin{equation*}
        \forall t \in \intoc{a,b}
        \qquad 
            \lim_{s \to t^-} V(t) - V(s)
        =
            \lim_{s \to t^-} 
                \metric\del{
                    \gamma (t),
                    \gamma(s)
                }.
    \end{equation*}
    Let $t \in \intoc{a,b}$, $\eps > 0$ and for $n \in \bN$ let $\del{ t_i^n}_{i=0}^{m_n} = \Delta_n \in \partition\del{ [a, t]}$ be such that 
    $\abs{ t_{m_n}^n-t_{m_n-1}^n } \le 1/n$ and     
    $$
        V\del{ \gamma \rvert_{[ a, t  ]}} - V^{\Delta_n}\del{ \gamma \rvert_{[ a, t ] } } \le \eps.
    $$ 
    Denote $s_n \coloneqq t_{m_n-1}^n$, then  
    \begin{equation*}
        V\del{
            \gamma \rvert_{[ a, s_n ] }
        }
        +V\del{
            \gamma \rvert_{[ s_n, t ] }
        }
        =
        V\del{
            \gamma \rvert_{[ a, t ] }
        }
        \le 
        V^{\Delta_n}\del{ \gamma \rvert_{[ a, t ] } } +\eps
        \le 
        V\del{
            \gamma \rvert_{[ a, s_n ] }
        }
        + \metric\del{
            \gamma(s_n), \gamma(t)
        }
        +
        \eps,
    \end{equation*}
    and thus
    \begin{equation*}
        V_\gamma(t)
        -
        V_\gamma(s_n)
        =
        V\del{
            \gamma \rvert_{[ s_n , t ] }
        }
        \le 
        \metric\del{
            \gamma(s_n), \gamma(t)
        }
        + \eps.
    \end{equation*}
    Therefore, since $s_n \to t^-$, we get
    \begin{equation*}
        \lim_{ s \to  t^-}
            V_\gamma(t)
            -
            V_\gamma(s)
        =
        \lim_{ n \to \infty}
            V_\gamma(t)
            -
            V_\gamma(s_n)
        \le 
         \lim_{ n \to \infty}
            \metric\del{
                \gamma(s_n), \gamma(t)
            }
            + \eps
        =
        \lim_{ s \to  t^-}
            \metric\del{
                \gamma(t), \gamma(s)
            }
            + \eps.
    \end{equation*}
    Thus, since $\eps > 0$ was arbitrary, we have
    \begin{equation}\label{zdolu}
        \lim_{ s \to  t^-}
            V_\gamma(t)
            -
            V_\gamma(s)
        \le 
        \lim_{ s \to  t^-}
            \metric\del{
                \gamma(t), \gamma(s)
            }.
    \end{equation}
    On the other hand, for all $s \in \intco{a,t}$ we have
    \begin{equation*}
        \metric\del{
                \gamma(t), \gamma(s)
            }
        \le 
        V\del{
            \gamma \rvert_{ [s, t] }
        }
        =
        V_\gamma(t)
        - V_\gamma(s),
    \end{equation*}
    hence
    \begin{equation*}
        \lim_{ s \to  t^-}
            V_\gamma(t)
            -
            V_\gamma(s)
        \ge 
        \lim_{ s \to  t^-}
            \metric\del{
                \gamma(t), \gamma(s)
            }.
    \end{equation*}
    Thus, gathering the above inequality with (\ref{zdolu}) the proof follows.
\end{proof}
\begin{corollary}
\label{cor:continuity_of_curve_implies_the_continuity_of_its_variation_function}
    Let $(X, \metricAlone)$ be a metric space and 
    $\gamma \in \boundedVariationRC\del{ [a,b]; X}$. If $\gamma$ is continuous, 
    then so is $V_\gamma$.
\end{corollary}

\begin{lemma}[Functions of bounded variation are Borel]
    \label{lem:functions_of_buonded_variation_are_borel}
    Let $(X, \metricAlone)$ be a metric space.
    Every element of 
    $
        \boundedVariation\del{  
            [a,b] ; X    
        }
    $ is a Borel map.
\end{lemma}
\begin{proof}
    Let $\gamma \in 
        \boundedVariation\del{  
            [a,b] ; X    
        }
    $, then by Lemma \ref{cor:variation_bounds_sum_of_jumps_from_above} there are at most countably many points of discontinuity of $\gamma$. Let us denote this set by $D$. Set $D$ is Borel as it is countable. This means that $[a,b] \setminus D$ is Borel.
    Function $\gamma \rvert_{ [a,b] \setminus D}$ is Borel, since it is continuous. 
    Moreover, the map $\gamma \rvert_D$ is Borel, since any subset of $D$ is Borel.
    Let $B \in \borel\del{X}$, then
    \begin{equation*}
        \gamma^{-1} \sbr{ B }
        =
        \del{ \gamma \rvert_{D} }^{-1} \sbr{B}
        \cup \del{
            \gamma \rvert_{ [a,b] \setminus D}
        }^{-1} \sbr{B}.
    \end{equation*}
    Hence, $\del{ \gamma \rvert_{D} }^{-1} \sbr{B}$ is Borel in $D$, and 
    $
         \del{
            \gamma \rvert_{ [a,b] \setminus D}
        }^{-1} \sbr{B}    
    $ 
    is Borel in $[a,b] \setminus D$. 
    However, as $D$ and $[a,b] \setminus D$ are Borel subsets of $[a,b]$, both of the mentioned preimages are Borel in $[a,b]$.
    We conclude that $\gamma^{-1} \sbr{ B }$ is Borel in $[a,b]$ and therefore function $\gamma$ is Borel.
\end{proof}
\begin{remark}
Since totally bounded sets are separable, thanks to the above lemma and Corollary \ref{cor:functions_of_bounded_variation_image_totally_bounded}
 we are able to use Proposition \ref{mierzalnosc} for maps from  $\boundedVariation\del{[a,b] ; X}$.
\end{remark}

\begin{definition}[Test curves]
    Let $(X, \metricAlone)$ be a metric space.
    We shall say that a function $\gamma \colon [a,b] \to X$ is a \emph{test curve}, if
    \begin{enumerate}[label=(\roman*)]
        \item 
        $ \gamma \in \boundedVariationRC\del{ [a,b]; X}$,
        \item 
        $\gamma$ is left-continuous at $b$,
        \item 
        The limit $\lim_{ s \to t^- } \gamma(s)$ exists for every $t \in \intoo{ a, b}$.
    \end{enumerate}
    We shall denote the family of all test curves $\gamma \colon [a,b] \to X$ by 
    $
        \testCurves\del{
            [a,b];
            X
        }
    $, and for $t\in (a,b]$ we define $\gamma(t^-)= \lim_{ s \to t^- } \gamma(s)$.
\end{definition}
\begin{example}
    \label{e1}
    Let $(X, \metricAlone)$ be a metric space and $x, y \in X$. Then, $\gamma \colon [a,b] \to X$ be defined by
    \begin{equation*}
        \forall t \in [a,b] 
        \qquad 
        \gamma(t)
        \coloneqq 
        \begin{cases}
            x, &\text{ for } t \in \intco{ a, (a + b)/2 }, \\
            y, &\text{ otherwise}
        \end{cases}
    \end{equation*} 
    belongs to $\gamma \in \testCurves\del{ [a,b] ; X}$.
\end{example}
\begin{remark}
\label{rem:restrciting_test_curve_to_continutity_point}
    Let $(X, \metricAlone)$
    be a metric space and 
    $
        \gamma \in \testCurves\del{ 
            [a,b]; 
            X
        }.
    $
    If $t \in [a,b]$ is a continuity point of $\gamma$, then 
    $
        \gamma \rvert_{[a,t]}
        \in 
         \testCurves\del{ 
            [a,t]; 
            X
        }.
    $
\end{remark}
\begin{remark}[Closure of the image of a test curve is compact]
\label{rem:closure_of_the_image_of_a_test_curve_is_compact}
    Let $(X, \metricAlone)$ be a metric space and $\gamma \in \testCurves\del{ [a,b]; X}$, then 
    \begin{equation*}
        \closure{ \image\del{\gamma}}
        =
        \image\del{ \gamma }
        \cup 
        \set{
            \gamma(t^-)
            \given 
            t \in (a,b]
        }
    \end{equation*}
    and $\closure{ \image\del{ \gamma } }$ is compact.
\end{remark}
\begin{proof}
    We have 
    $
        \image\del{ \gamma } 
        \cup 
        \set{
            \gamma(t^-)
            \given 
            t \in (a,b]
        }
        \subseteq 
        \closure{ \image\del{ \gamma }}        
    $.
    Now, suppose that 
    $   
        x \in 
        \closure{ \image\del{ \gamma }} 
            $. Then by compactness of $[a,b]$ there is a sequence $(t_n)_n$ of elements of $[a,b]$ and $t\in [a,b]$ such that $\gamma(t_n) \rightarrow x$ and $t_n \rightarrow t$ as $n \to \infty$.
    
    We have then two possibilities:
    \begin{itemize}
        \item
        Sequence $(t_{n})_n$ has a subsequence $\del{t_{n_{k}}}_k$
        such that $t_{n_{k}} < t$ for all $k \in \bN$. Then $ t_{n_{k}}   \to t^-$ and 
        $\gamma \del{
                t_{n_{k}}            
            }
            \xrightarrow{k \to \infty}
            \gamma \del{ t^-}$.
        Hence $x = \gamma(t^-)$ and thus $x\in \set{
            \gamma(t^-)
            \given 
            t \in (a,b]
        }$.
        \item 
        For all large $n \in \bN$ $t \le t_{n}$ . 
        Then $ t_{n}   \to t^+$ and therefore $
            \gamma \del{
                t_{n}            
            }
            \xrightarrow{n \to \infty}
            \gamma \del{ t}.$      %
        In this way we get $x = \gamma(t) \in \image\del{ \gamma }$.
    \end{itemize}
    Thus, we have proved 
    \begin{equation*}
        \closure{ \image\del{ \gamma }} 
                    \subseteq 
        \set{
            \gamma(t^-)
            \given 
            t \in (a,b]
        } \cup 
            \image\del{ \gamma },
    \end{equation*}
    and we conclude that 
    \begin{equation*}
        \closure{ \image\del{\gamma}}
        =
        \image\del{ \gamma }
        \cup 
        \set{
            \gamma(t^-)
            \given 
            t \in (a,b]
        }.
    \end{equation*}
    
    Next, we shall prove that 
    $
        \closure{ \image\del{ \gamma } }
    $
    is compact. Let $(x_n)$ be a sequence in $\closure{\image\del{\gamma}} $. 
    Then for every $n \in \bN$ we have $t_n \in [a,b]$ such that 
    $x_n = \gamma(t_n)$
    or
    $x_n = \gamma(t_n^-)$.
    As $[a,b]$ is compact, 
    we have a subsequence $(t_{n_k})_k$ and $t \in [a,b]$ such that $t_{n_k} \to t$ as $k \to \infty$.

    We have three possibilities:
    \begin{itemize}
        \item
        Sequence $(t_{n_k})_k$ has a subsequence $(t_{n_{k_l}})_l$ such that 
        $t_{n_{k_l}} < t$ for all $l \in \bN$.
        Then $t_{n_{k_l}} \to t^-$.
        Since $\gamma \in \testCurves\del{ [a,b]; X}$,
        limit $\gamma(t^-)$ exists, so for every $\eps > 0$ there exists $\delta > 0$ such that if $\abs{ t - s} \le 2\delta$ for $s \in [a, t)$, then $\metric \del{ \gamma(t^-), \gamma(s)} \le \eps$. 
        This means that for every $s \in [a,t)$ such that $\abs{ t - s} \le \delta$ we have
        $\metric \del{ \gamma(t^-), \gamma(s)} \le \eps$ and $\metric \del{ \gamma(t^-), \gamma(s^-)} \le \eps$. 
        Hence, 
        $
            \metric\del{
                \gamma(t^-),
                x_{n_{k_l}}
            }
            \le 
            \eps             
        $
        for all large $l \in \bN$. 
        As $\eps > 0$ is arbitrary, we have 
        $x_{n_{k_l}} \to \gamma(t^-)$. 

        \item 
        Sequence $(t_{n_k})_k$ has a subsequence $(t_{n_{k_l}})_l$ such that 
        $t_{n_{k_l}} > t$ for all $n \in \bN$.
        Then $t_{n_{k_l}} \to t^+$.
        Since $\gamma \in \testCurves\del{ [a,b]; X}$,
        limit $\gamma(t^+) = \gamma(t)$ exists, so for every $\eps > 0$ there exists $\delta > 0$ such that if $\abs{ t - s} \le 2\delta$ for $s \in (t,b]$, then $\metric \del{ \gamma(t), \gamma(s)} \le \eps$. 
        This means that for every $s \in (t,b]$ such that $\abs{ t - s} \le \delta$ we have
        $\metric \del{ \gamma(t), \gamma(s)} \le \eps$ and $\metric \del{ \gamma(t), \gamma(s^-)} \le \eps$. 
        Hence, 
        $
            \metric\del{
                \gamma(t),
                x_{n_{k_l}}
            }
            \le 
            \eps             
        $
        for all large $l \in \bN$. 
        As $\eps > 0$ is arbitrary, we have 
        $x_{n_{k_l}} \to \gamma(t)$. 

        \item 
        For all large $k \in \bN$ we have $t_{n_{k}} = t$. Then $x_{n_{k}} = \gamma(t)$ for infinitely many $k$, or $x_{n_{k}} = \gamma(t^-)$ for infinitely many $k$. In either case, $x_{n_{k}}$ has a further subsequence which converges to $\gamma(t)$ or $\gamma(t^-)$.
    \end{itemize}
\end{proof}
\begin{lemma}
\label{lem:right_continuity_and_agreeing_on_a_dense_subset}
    Let $(X, \metricAlone)$ be a metric space and $C \subseteq (a,b)$ be a dense set. Then
    \begin{enumerate}
        \item     
    If $\gamma, \gamma' \colon (a,b) \to X$ are right-continuous and such that $\gamma \rvert_C = \gamma' \rvert_C$, then $\gamma = \gamma'$.
     \item If $\wave{\gamma} \colon C \to X$ is such that $\wave{\gamma}(t^+):=\lim_{C \ni s \rightarrow t^+} \wave{\gamma}(s)$  exists for $t \in [a,b)$, then $t \mapsto \wave{\gamma}(t^+)$ is right-continuous.
     \end{enumerate}
\end{lemma}
\begin{proof}
    The statement $(1)$ is straightforward. Now, for the second part of the lemma we fix $\eps > 0$ and $t \in [a,b)$. There exists $\delta > 0$ such if 
    $s \in \del{ t, t+2\delta}\cap C$, then $\metric\del{\wave{\gamma}(t^+), \wave{\gamma}(s)}\le \eps$. Hence, if $s \in \del{ t, t+\delta}$, then $\metric\del{\wave{\gamma}(t^+), \wave{\gamma}(s^+)}  \le \eps $ and $t \mapsto \wave{\gamma}(t^+)$ is right-continuous as needed.
\end{proof}
\begin{definition}[Reversal of a curve]
    For $\gamma \in \testCurves\del{ [a,b]; X}$ we define $\overleftarrow{\gamma} \colon [a,b] \to X$ by
    \begin{equation*}
        \forall t \in [a,b]
        \qquad 
        \overleftarrow{\gamma}(t) \coloneqq \gamma \del{\del{ a + b - t }^-},
     \end{equation*}
    where we use $\gamma \del{ a^-} = \gamma(a)$.
\end{definition}
\begin{lemma}
\label{lem:reversing_test_curve}
    Let $(X, \metricAlone)$ be a metric space and for $\gamma \in BV([a,b];X)$ we denote by $C_{\gamma}$ the set of continuity points of $\gamma$. Then function $\overleftarrow{\cdot} $ has the following properties:
    \begin{enumerate}[label=(\alph*)]
        \item 
        $
            \overleftarrow{\cdot} \colon \testCurves\del{ [a,b]; X} \to \testCurves\del{ [a,b]; X}$,
        \item
        $\overleftarrow{\overleftarrow{\gamma}} = \gamma$,
        \item 
         $C_{\overleftarrow{\gamma}} = b+a-C_{\gamma}$,
        \item
        For all $t \in [a,b]$ we have 
        $
            V_{\overleftarrow{\gamma}}(t) + \overleftarrow{V_{\gamma}}(t) =
            V(\gamma)
        $
        and hence $V(\overleftarrow{\gamma}) = V(\gamma)$.
    \end{enumerate}
\end{lemma}
\begin{proof}
    For simplicity of notation, let us define function $w \colon [a,b] \to [a,b]$
    by the formula $w(t) = a+b-t$ for $t \in [a,b]$.
    \begin{enumerate}[label=(\alph*), itemsep=0.5em,leftmargin=0em, itemindent=2em,]
        \item 
    Let $t \in \intco{a, b}$, then, since $\gamma \in \testCurves\del{ [a,b]; X}$,
    the limit $\gamma 
            \del{
                \del{w(t)}^-            
            }$ exists.
    For $\eps > 0$, there exists $\delta > 0$ such that if $w(s) \in \intco{ a, w(t)}$ satisfies 
    $
        \abs{t-s} = \abs{ w(t) - w(s) } \le 2\delta
    $,
    then $\metric\del[2]{
            \gamma 
            \del[1]{
                \del{ w(t)}^-            
            },
            \gamma\del{ w(s) }
        } 
        \le 
        \eps.$
    Therefore, if $s \in (t,b]$ satisfies $|t-s| \leq \delta$, then 
    \begin{equation*}
        \metric\del{
            \overleftarrow{ \gamma}(t),
            \overleftarrow{ \gamma}(s)
        }
        =
        \metric\del[2]{
            \gamma 
            \del[1]{
                \del{ w(t)}^-            
            },
            \gamma\del{ 
                \del{ w(s) }^-     
            }
        } 
        \le 
        \eps.
    \end{equation*}
      As $\eps > 0$ is arbitrary, we conclude that $\overleftarrow{ \gamma}$ is right-continuous at $t$ for all $t \in \intco{a, b}$.

    Let $t \in \intoc{a,b}$, we will show
    if $\gamma$ is left-continuous at
    $w(t)$, then $\overleftarrow{\gamma}$
    is left-continuous at $t$. Let $\eps >0$, since $\gamma$ be left-continuous at
    $w(t)$, then
    $
        \overleftarrow{\gamma}(t)
        =
        \gamma\del{ (w(t))^- }
        =
        \gamma\del{ w(t)}
    $.
    Since $\gamma$ is right-continuous at $w(t)$, there exists $\delta > 0$ such that
    if 
    $
        w(s) \in \intcc{ w(t), b},
    $
    satisfies
    $
        \abs{t-s}
        =
        \abs{w(t) - w(s)}
        \le 
        2\delta
    $,
    then
    $
        \metric\del{
            \gamma(w(t)),
            \gamma(w(s))
        }
        \le 
        \eps
    $.
    Therefore, if $s \in \intcc{a,t}$,
    satisfies
    $
        \abs{t-s}
        \le 
        \delta
    $,
    then
    \begin{equation*}
        \metric\del{
            \overleftarrow{\gamma}(t),
            \overleftarrow{\gamma}(s)
        }
        =
        \metric\del{
            \gamma(w(t)),
            \gamma\del{(w(s))^-}
        }
        \le 
        \eps.
    \end{equation*}
    We conclude that $\overleftarrow{\gamma}$
    is left-continuous at $t$.
    
    Note that, in particular, since
    $\gamma$ is left-continuous at $a$, 
    we have that $\overleftarrow{\gamma}$ 
    is left-continuous at $b$.
    Moreover, we have shown that if
    $t\in [a,b]$ is such that $\gamma$
    is continuous at $a+b-t$, then
    $\overleftarrow{\gamma}$ is continuous at $t$.

    Let $\Delta=\del{t_i}_{i=0}^k$
    be a partition of $[a,b]$. Then we have
    \begin{align*}
        V^{\Delta}\del{ \overleftarrow{\gamma} }
        &=
        V^{w\sbr{ \Delta }}(\gamma)
        +
        \del{
             V^{\Delta}\del{ \overleftarrow{\gamma} }
            -V^{w\sbr{ \Delta }}(\gamma)
        }
        \\
        &=
        V^{w\sbr{ \Delta }}(\gamma)
        +
        \sum_{i=1}^k
            \del{
                \metric\del{
                    \overleftarrow{\gamma}(t_i),
                    \overleftarrow{\gamma}(t_{i-1})
                }
                -
                \metric\del{
                    \gamma\del{ 
                         w\del{ t_{i}  }
                    },
                    \gamma\del{ 
                        w\del{ t_{i-1} } 
                    }
                }
            }   
        \\
        &=
        V^{w\sbr{ \Delta }}(\gamma)
        +
        \sum_{i=1}^{k}
            \del{
                \metric\del{
                    \gamma\del{ 
                        \del{ w\del{t_{i}} }^- 
                    },
                    \gamma\del{ 
                        \del{ w\del{ t_{i-1}} }^-  
                    }
                }
                -\metric\del{
                    \gamma\del{ 
                         w\del{ t_{i}  }
                    },
                    \gamma\del{ 
                        w\del{ t_{i-1} } 
                    }
                }
            }           
        \\ 
        &\le 
        V^{w\sbr{ \Delta }}(\gamma)
        +
        \sum_{i=1}^{k}
        \del{
                    \metric\del{
                        \gamma\del{ 
                            \del{ w\del{ t_{i} }}^- 
                        },
                        \gamma\del{ 
                            w\del{ t_{i}} 
                        }
                    }
                    +
                    \metric\del{
                        \gamma\del{ 
                             \del{ w\del{ t_{i-1}  } }^-
                        },
                        \gamma\del{ 
                             w\del{ t_{i-1} }
                        }
                    }
                }
        \\
        &=
        V^{w\sbr{ \Delta }}(\gamma)
        +
        \sum_{i=1}^{k}
        \del{
                    \phi_\gamma\del{w\del{ t_{i} }  }
                    +
                    \phi_\gamma \del{w\del{ t_{i-1} } }
                }
        \\
        &\le 
        V^{w\sbr{ \Delta }}(\gamma)
        +
        2
        \sum_{t \in [a,b] }
            \phi_\gamma(t)
        \\
        &\le 
        3 V(\gamma),
    \end{align*}
    where in the last step we used 
    Corollary \ref{cor:variation_bounds_sum_of_jumps_from_above}.
    Therefore, 
    $V(\overleftarrow{\gamma}) \le 3V(\gamma)$
    and $\overleftarrow{\gamma} \in \boundedVariationRC\del{ [a,b]; X}$.

    By the definition of $\overleftarrow{\gamma}$
    we have 
    $
        \image\del{
            \overleftarrow{\gamma}
        }
        \subseteq
        \closure{
            \image\del{ \gamma}
        }
    $
    and $\closure{
            \image\del{ \gamma}
        }$
    is compact by Remark \ref{rem:closure_of_the_image_of_a_test_curve_is_compact}. Therefore, the existence of left-limits
    $\overleftarrow{\gamma}(t^-)$ 
    for $t \in \intoc{a,b}$  follows from Lemma \ref{lem:functions_of_bounded_variation_have_cauchy_limits_everywhere}. 
    We conclude that 
    $\overleftarrow{\gamma} \in \testCurves\del{ [a,b]; X}$.

    \item Next, we will prove that 
    $\overleftarrow{\overleftarrow{\gamma}} = \gamma$.
    Let $t \in C_{\gamma}$, then $\overleftarrow{\gamma}$ is continuous at $a+b-t$. 
    Hence, $\overleftarrow{\overleftarrow{\gamma}}$ is continuous at $t$. 
    Therefore, 
    \begin{equation*}
        \overleftarrow{\overleftarrow{\gamma}}(t)
        =
        \overleftarrow{\gamma}\del{ (a+b-t)^-}
        =
        \overleftarrow{\gamma}\del{a+b-t}
        =
        \gamma\del{
            (a+b - (a+b-t))^-
        }
        =
        \gamma(t^-)
        =
        \gamma(t),
    \end{equation*}
    and thus,  $\overleftarrow{\overleftarrow{\gamma}} \rvert_{C_{\gamma}} = \gamma \rvert_{C_{\gamma}}$. Since $a, b \in C_{\gamma}$, $C_{\gamma}$ is dense in $(a,b)$, and $\overleftarrow{\overleftarrow{\gamma}}$, $\gamma$ are right-continuous functions on $(a,b)$, 
    by Lemma \ref{lem:right_continuity_and_agreeing_on_a_dense_subset} we have $\overleftarrow{\overleftarrow{\gamma}} = \gamma$ on the entire $[a,b]$.

    \item We have previously shown that if $\gamma$
    is continuous at $t$, then 
    $\overleftarrow{\gamma}$ is continuous at $w(t)$. 
    Now, if $\overleftarrow{\gamma}$ is continuous at $t$, then $\overleftarrow{\overleftarrow{\gamma}} = \gamma$ is continuous at $w(t)$.
    This proves that $a+b - C_{\gamma}= C_{\overleftarrow{\gamma}}$.

    \item We know that the set $C_{\gamma}$ contains $a$ and $b$, and is dense in $[a,b]$. 
    The same is true for $C_{{\overleftarrow{\gamma}}}$, the set
    of continuity points of 
    $\overleftarrow{\gamma}$.
    Let $t \in C_{\overleftarrow{\gamma}}$. Since $C_{\overleftarrow{\gamma}}$ is dense in $[a,b]$,
    then $C_{\overleftarrow{\gamma}}\cap[a,t]$ is dense in $[a,t]$.
    Therefore, there exists a normal sequence
    $\del{\tau_i^n}_{i=0}^{m_n} = \Delta_n$
    of partitions of $[a,t]$ such that
    $\Delta_n \subseteq C_{\overleftarrow{\gamma}}$ for all $n \in \bN$.
    Then $\del{ a+b-\Delta_n}_n \subset C_{\gamma}$
    is a normal sequence of partitions of 
    $[a+b-t,b]$.  Therefore, for all $n \in \bN$ we have
    \begin{align*}
        V^{\Delta_n}\del{ 
            \overleftarrow{\gamma}\rvert_{[a,t]}
        }
        &=
        \sum_{i=0}^{m_n}
            \metric\del{
                \overleftarrow{\gamma}(\tau_i^n),
                \overleftarrow{\gamma}(\tau_{i-1}^n)
            }
        \\
        &=
        \sum_{i=0}^{m_n}
            \metric\del{
                \gamma\del{(w(\tau_i^n))^-},
                \gamma\del{(w(\tau_{i-1}^n))^-}
            }
        \\
        &=
        \sum_{i=0}^{m_n}
            \metric\del{
                \gamma\del{w(\tau_i^n)},
                \gamma\del{w(\tau_{i-1}^n)}
            }
        \\
        &=
        V^{a+b-\Delta_n}\del{ 
            \gamma\rvert_{[w(t),b]}
        }.
    \end{align*}
    Hence, by Corollary \ref{cor:variation_is_a_uniform_limit_of_Delta-variations} and Corollary \ref{cor:variation_is_additive} we have
    \begin{equation*}
        V_{\overleftarrow{\gamma}}
        (t)
        =
        V\del{ \overleftarrow{\gamma} \rvert_{[a,t]} }
        =
        V\del{ \gamma \rvert_{[w(t),b]} }
        =
        V_\gamma(b) - V_\gamma(w(t))
        =
        V(\gamma) - V_\gamma(w(t))= V(\gamma) - V_\gamma((w(t))^-),
    \end{equation*}
    where the last equality follows from Proposition \ref{rem:properties_of_variation_function} and the assumption $t\in C_{\overleftarrow{\gamma}}$. Let us point out that by Proposition \ref{rem:properties_of_variation_function} we have $V_{\gamma} \in \testCurves\del{ [a,b]; [0,\infty)}$ and by point $(a)$ we have $\overleftarrow{V_{\gamma}} \in \testCurves\del{ [a,b]; [0,\infty)}$. Therefore, for $t \in C_{\overleftarrow{\gamma}}$ we have
    $$
        V_{\overleftarrow{\gamma}}
        (t)
        =
        V(\gamma) - \overleftarrow{V_\gamma}(t).
    $$
    Finally, since $(a,b) \ni t \mapsto V(\gamma) - \overleftarrow{V_\gamma}(t)$ and $(a,b) \ni t \mapsto V_{\overleftarrow{\gamma}}(t)$ are right-continuous maps which coincide on a dense set $C_{\overleftarrow{\gamma}}$ and $a,b \in C_{\overleftarrow{\gamma}}$, Lemma \ref{lem:right_continuity_and_agreeing_on_a_dense_subset} finishes the proof.

    The final part of the lemma
    is a consequence of the fact that
    \begin{equation*}
        V \del{\overleftarrow{\gamma}}
        =
        V_{\overleftarrow{\gamma}}(b)
        =
        V(\gamma) - V_\gamma(a+b-b)
        =
        V(\gamma) - V_\gamma(a)
        =
        V(\gamma),
    \end{equation*}
    where we used the fact that $a \in C_{\gamma}$.
    \end{enumerate}
\end{proof}

\subsection{Integration along a curve} \label{subsec::integration_along_a_curve}
 Let $(X, \metricAlone)$ be a metric space and  $\gamma \in \boundedVariationRC\del{[a,b];X}$. Since $V_{\gamma} \in \boundedVariationRC\del{[a,b];[0, \infty)}$ and $V_{\gamma}$ is non-decreasing, by the Caratheodory Extension Theorem there exists a measure\footnote{Such a measure will be called a Lebesgue--Stieltjes measure induced by $V_{\gamma}$.} $\mu_{\gamma}$ defined on $\borel\del{ [a,b] }$ such that 
    \begin{equation*}
        \forall r, t \in [a,b],
        \enskip r < t
        \qquad 
            \mu_{\gamma}\del{
                \intoc{ r, t}        
            }
        =
            V_{\gamma}(t) - V_{\gamma}(r) \quad \text{and} \quad \mu_{\gamma}\del{ \set{a} } =  0.
    \end{equation*}
     Let  $f \colon X \to \bRExtended$ be Borel. We define the Lebesgue--Stieltjes integral of $f$ along curve $\gamma$ by 
    \begin{equation*}
        \int_\gamma f
        \coloneqq 
        \integral{ [a,b] }{ f \circ \gamma }{ \mu_\gamma }.
    \end{equation*}
    We note that $f \circ \gamma$ is Borel since $f$ is Borel and $\gamma$ is Borel by Lemma \ref{lem:functions_of_buonded_variation_are_borel}.
\begin{example}
\label{ex:assymetry_of_the_Lebesgue-Stieltjes_integral}
    Let $(X, \metricAlone)$ be a metric space and $x, y \in X$. Let $\gamma \colon [a,b] \to X$ be defined by
    \begin{equation*}
        \gamma(t)
        \coloneqq 
        \begin{cases}
            x, &\text{ for } t \in \intco{ a, (a + b)/2 }, \\
            y, &\text{ otherwise}.
        \end{cases}
    \end{equation*} 
    Then, for Borel map $f \colon X \to \bRExtended$ we have $\int_\gamma f = f(y) \metric\del{ x, y}$.
\end{example}
Having in mind the above example we define the symmetrized integral. 
\begin{definition}[Symmetrized Lebesgue--Stieltjes integral]\label{def::symmetrized_integral}
      Let $(X, \metricAlone)$ be a metric space and 
    $
        \gamma \in \testCurves\del{  
            [a,b] ; X    
        }
    $.
    Let  $f \colon X \to \bRExtended$ be Borel. We define the symmetrized Lebesgue--Stieltjes integral of $f$ along curve $\gamma$ by
    \begin{equation*}
        \sint{ \gamma }{ f }
        \coloneqq 
        \frac{1}{2}
        \del{
            \int_\gamma f
            + \int_{ \overleftarrow{\gamma} } f     
        }.
    \end{equation*}
\end{definition}
\begin{example}\label{prz325}
    Let $(X, \metricAlone)$ be a metric space, $x, y \in X$ and let $\gamma \colon [a,b] \to X$ be a curve from Example \ref{ex:assymetry_of_the_Lebesgue-Stieltjes_integral}, then 
    \begin{equation*}
         \overleftarrow{\gamma}(t)
        =
        \begin{cases}
            y, &\text{ for } t \in \intco{ a, (a + b)/2 }, \\
            x, &\text{ otherwise}.
        \end{cases}
    \end{equation*} 
    Thus, for  Borel map $f \colon X \to \bRExtended$ such that $ (f(x) + f(y))$ makes sense we have
    $\sint{ \gamma }{ f }
            =
            (f(x) + f(y))
            \metric\del{ x, y }/2.$
\end{example}
\begin{remark} \label{rem::continuous_gamma_then_symmetrized_integral_is_the_usual_integral}
     Let $(X, \metricAlone)$ be a metric space and $\gamma \in \testCurves\del{ [a,b]; X}$. If $\gamma$ is continuous, then for every Borel map  $f \colon X \to [0,\infty]$ we have
     \begin{equation*}
          \sint{ \gamma }{ f }
            =
            \int_{\gamma} f.
    \end{equation*}
    Moreover, the right hand side in the above equality coincides with  the integral of a Borel function along a rectifiable curve \cite{Bjorn}.
    \end{remark}
\begin{proof}
     Let $w \colon [a,b] \to [a,b]$ be defined by the formula $w(t) = a+b-t$, then since $\gamma$ is continuous we have
    $ \overleftarrow{\gamma} = \gamma \circ w$. Moreover, by Lemma \ref{lem:reversing_test_curve} we get
$
            V_{\overleftarrow{\gamma}}(t) =V(\gamma) -V_{\gamma}(w(t)). 
        $
    Let us observe that    
    \begin{equation*}
        \mu_{\overleftarrow{\gamma}}\del{
            \set{a}
        }
        =
        0
        =
        \lim_{s \to b^-}
            V_\gamma(b) - V_\gamma(s)
        =
        \lim_{s \to b^-}
            \mu_\gamma\del{ \intoc{s, b} }
        =
        \mu_\gamma\del{ \set{b} }
        =
        \mu_\gamma\del{
            w \sbr{ \set{a } }
        },
    \end{equation*}
    and for $r,t \in [a,b]$ we have
    \begin{align*}
        \mu_{\overleftarrow{\gamma}}\del{
            \intoc{r,t}
        }
        &=
        V_{\overleftarrow{\gamma}}(t)
        - V_{\overleftarrow{\gamma}}(r)
        \\
          &=
        V_\gamma\del{ w(r) }
        - V_\gamma\del{ w(t) }
        \\
         &=
        \mu_{\gamma}\del{ 
            \intoc{ w(t), w(r)}
        }
        \\
        &=
        \mu_{\gamma}\del{
            w \sbr{ \intoc{r,t} }
        }.
    \end{align*}
    In this way we have proved that for every $B$ such that $B=\{a\}$ or $B=(r,t]$ we have 
    $
        \mu_{\overleftarrow{\gamma}}\del{ B}
        =
        \mu_{\gamma}\del{
            w \sbr{ B }
        }
    $.
    Therefore, since sets of the mentioned form
    generate the entire $\sigma$-algebra
    of Borel sets on $[a,b]$, by a standard application of the Dynkin $\pi$-$\lambda$ Lemma, we have 
    $
        \mu_{\overleftarrow{\gamma}}\del{ B}
        =
        \mu_{\gamma}\del{
            w \sbr{ B }
        }
        =
        \del{ (w^{-1})_\# \mu_{\gamma} }(B)
    $ 
    for all Borel sets.

    Therefore, for all Borel $f \colon X \to [0,\infty]$ we have
    \begin{equation*}
        \int_{\overleftarrow{\gamma}}
            f
        =
        \integral{ [a,b] } { 
            f \circ \overleftarrow{\gamma}
         }{ 
            \mu_{\overleftarrow{\gamma}}
         }
         =
         \integral{ [a,b] } { 
            f \circ \gamma \circ w
         }{ 
            \del{ (w^{-1})_\# \mu_{\gamma} }
         }
         =
         \integral{ [a,b] } { 
            f \circ \gamma 
         }{ 
            \mu_{\gamma} 
         }
         =
        \int_\gamma f.
    \end{equation*}
 Now, we shall prove $ \int_{\gamma} f$ coincides with  the integral of a Borel function along a rectifiable curve.
 Let $\tilde{\gamma}: [0, V(\gamma)] \rightarrow X$ be the arc-length parametrization of $\gamma$, i.e, $\gamma = \tilde{\gamma} \circ V_{\gamma}$. 
 It is well known that $\tilde{\gamma}=\gamma \circ h$, where $h:[0,V(\gamma)]\rightarrow [a,b]$ is given by the formula $h(t)=\inf V_{\gamma}^{-1}[\{t\}]$ (see \cite{HajCon}). Therefore,
 \[
  \int_{[0,V(\gamma)]} f \circ \tilde{\gamma} \ \mathrm{d} \lambda = \int_{[0,V(\gamma)]} f \circ \gamma \circ h \ \mathrm{d} \lambda =
  \int_{[a,b]} f \circ \gamma \ \mathrm{d} h_\#\lambda.
 \]
Furthermore, for every $r, t \in [a,b]$ such that $r<t$ we have
$h^{-1}(\{a\})=\{0\}$ and
\[
h^{-1}((r,t])=(V_{\gamma}(r),V_{\gamma}(t)].
\]
Thus 
\[
\mu_{\gamma}(\{a\})=0=\lambda(\{0\})=h_\#\lambda(\{a\})
\]
and
\[
\mu_{\gamma}((r,t])=V_{\gamma}(t) - V_{\gamma}(r) = \lambda((V_{\gamma}(r), V_{\gamma}(t)]) = h_\#\lambda ((r,t]).
\]
Hence, $ h_\#\lambda=\mu_\gamma$ on Borel sets of $[a,b]$ and we get 
\[
\int_{[0,V(\gamma)]} f \circ \tilde{\gamma} \ \mathrm{d} \lambda = 
  \int_{[a,b]} f \circ \gamma \ \mathrm{d} \mu_{\gamma}.
\]
\end{proof}
\begin{lemma}
\label{lem:symmetrized_Lebesgue-Stieltjes_integral_as_a_pushforward}
    Let $(X, \metricAlone)$ be a metric space, $\gamma \in \testCurves\del{ [a,b]; X}$ and let us define a Borel measure $\mu^S_\gamma$ on $X$ as follows 
    $
        \mu_{\gamma}^S
        \coloneqq 
        \frac{1}{2}
            \del{
                \gamma_\# \mu_\gamma
                +
                \overleftarrow{\gamma}_\#
                \mu_{\overleftarrow{\gamma}}
            }
    $. Then for every Borel map $g \colon X \to \intcc{0, \infty}$ we have 
    \begin{equation*}
               \sint{\gamma}{g}
        =
        \integral{X}{g}{\mu^S_\gamma}.
    \end{equation*}
\end{lemma}
\begin{proof}
    Let $g \colon X \to \intcc{0, \infty}$ be Borel.
    Then
    \begin{align*}
        \sint{\gamma}{g}
        &=
        \frac{1}{2}
        \del{
            \int_\gamma g
            +
            \int_{\overleftarrow{\gamma}} g
        }
                \\
        &=
        \frac{1}{2}
        \del{
            \integral{[a,b]}{ g \circ \gamma}{ \mu_\gamma}
            +
             \integral{[a,b]}{ g \circ \overleftarrow{\gamma}}{
                \mu_{\overleftarrow{\gamma}}
             }        
        }
        \\ 
        &=
        \frac{1}{2}
        \del{
            \integral{X}{ g}{ \gamma_\# \mu_\gamma}
            +
             \integral{X}{ g }{
                \overleftarrow{\gamma}_\#
                \mu_{\overleftarrow{\gamma}}
             }        
        }
        \\
        &=
        \integral{X}{g}
        {\mu_{\gamma}^S
                   }.
    \end{align*}
\end{proof}
\begin{remark}
   \label{remM}
    Let $(X, \metricAlone)$ be a metric space
    and $\gamma \in \testCurves\del{ [a,b]; X}$, then 
    \begin{equation*}
        \mu^S_\gamma(\closure{ \image{ \gamma } })=\mu^S_\gamma(X)= V(\gamma).
    \end{equation*}
\end{remark}
\begin{proof}
    By Lemma \ref{lem:reversing_test_curve} we have
    \begin{equation*}
        \mu^S_\gamma(X) = \frac{1}{2}\del{
            \int_\gamma 1
            +
            \int_{\overleftarrow{\gamma}} 1
        } = \frac{1}{2}\del{ V(\gamma )
            +
            V(\overleftarrow{\gamma}) 
        }= V(\gamma).
    \end{equation*}
    Furthermore,
    \begin{align*}
        \mu^S_\gamma(\closure{ \image{ \gamma } }) 
        &\le 
        \mu^S_\gamma(X) 
        \\
        &= 
        \frac{1}{2}
            \del{
                \gamma_\# \mu_\gamma (X)
                +
                \overleftarrow{\gamma}_\#
                \mu_{\overleftarrow{\gamma}}(X)
            }\\
            &=\frac{1}{2}
            \del{
                \gamma_\# \mu_\gamma (\image{ \gamma } )
                +
                \overleftarrow{\gamma}_\#
                \mu_{\overleftarrow{\gamma}}(\image{ \overleftarrow{\gamma} )}
            }\\
            &\le \frac{1}{2}
            \del{
                \gamma_\# \mu_\gamma (\closure{\image{ \gamma } })
                +
                \overleftarrow{\gamma}_\#
                \mu_{\overleftarrow{\gamma}}(\closure{\image{ \gamma } })} \\
                &=\mu^S_\gamma(\closure{ \image{ \gamma } }).
    \end{align*}
\end{proof}
\begin{proposition}
\label{rem:test_Curves_changing_intervals}
    Let $\psi \colon [c,d] \to [a,b]$ 
    be defined by the formula
    \begin{equation*}
        \forall t \in [c,d]
        \qquad 
        \psi(t)
            \coloneqq 
            a
            +
            \frac{ b-a}{d-c}(t-c).
    \end{equation*}
    Then the push-forward $\psi_\#$ defined by the formula 
    \begin{equation*}
        \forall 
            \gamma \in 
            \testCurves\del{ 
                [a,b];
                X
            }
        \qquad 
            \psi_\#(\gamma)
            =
            \gamma \circ \psi
    \end{equation*}
    has the following properties:
    \begin{enumerate}[label=(\alph*)]
        \item 
        $
            \psi_\#
            \colon
            \testCurves\del{ 
                [a,b];
                X
            }
            \to 
            \testCurves\del{ 
                [c,d];
                X
            }
        $
        is a bijection,
        \item 
        $C_{\psi_\#(\gamma)} = \psi^{-1}(C_{\gamma})$, 
        \item 
        $
            \psi_\#\del{ 
                \overleftarrow{\gamma}
            }
            =
            \overleftarrow{
                \psi_\#(\gamma)
            }
        $,
        \item 
        $
            V_{\psi_\#(\gamma)}
            =
            V_\gamma \circ \psi
        $
        \item 
        $
            (\psi^{-1})_\#(\mu_\gamma)
            =
            \mu_{\psi_\#(\gamma)}
        $
        \item 
        For all Borel $f \colon X \to [0,\infty]$ we have
        $
            \int_\gamma f
            =
            \int_{\psi_\#(\gamma) } f
        $
        and 
        $
            \sint{\gamma}{ f }
            =
            \sint{\psi_\#(\gamma) }{ f}
        $.
    \end{enumerate}
\end{proposition}
\begin{proof}
    \vspace{-1em}
    \begin{enumerate}[itemindent=2em,leftmargin=0em, label=(\alph*), itemsep=0.5em]
        \item 
        The map $\psi$ is an increasing homeomorphism.
    Let 
    $
        \gamma 
        \in 
        \testCurves\del{
            [a,b]; X
        }
    $, then for all $\Delta \in \partition\del{ [c,d] }$
    we have 
    $
        \psi\sbr{ \Delta }\in \partition\del{ [a,b] }
    $
    and 
    $
        V^{\psi\sbr{ \Delta }}(\gamma)
        =
        V^{\Delta}(\psi_\#(\gamma))
    $.
    Therefore,  
    $
        V\del{ \psi_\#(\gamma) }
        \le 
        V( \gamma)
    $
    and we get  
    $
        \psi_\#(\gamma)
        \in 
        \boundedVariation\del{
            [a,b]; X
        }
    $.
    Next, let us note that since $\psi$ is an increasing homeomorphism, 
    for 
    $
        \gamma \in \testCurves\del{
            [c,d];
            X
        }
    $
    we have that if $t \in [a,b]$
    is a point of left-continuity
    (right-continuity) of $\gamma$,
    then $\psi^{-1}(t)$
    is a point of left-continuity
    (right-continuity) of $\psi_\#(\gamma)$.
    Moreover, for all $s \in \intoc{c,d}$
    we have
    $
        \psi_\#(\gamma)(s^-)
        =
        \gamma\del{ \del{ \psi(s) }^- }
    $.
    This means that 
    $
        \psi_\#(\gamma)
        \in 
        \testCurves\del{
            [c,d]; X
        }
    $.
        We have proved that 
    $
            \psi_\#
            \colon
            \testCurves\del{ 
                [a,b];
                X
            }
            \to 
            \testCurves\del{ 
                [c,d];
                X
            }
        $.
    Since $\psi^{-1}$ has the same properties as 
    $\psi$, we have 
    $
        (\psi^{-1})_\# 
        \colon 
        \testCurves\del{ 
                [c,d];
                X
            }
            \to 
            \testCurves\del{ 
                [a,b];
                X
            }
        $
    and $(\psi^{-1})_\# $ is an inverse 
    of $\psi_\#$.
    Hence, $\psi_\#$ is a bijection. 

    \item We have previously shown that $\psi^{-1}\sbr{C_{\gamma}}\subset C_{\psi_\#(\gamma)}$. Hence, we also get $\psi\sbr{C_{\psi_\#(\gamma)}}\subset C_{\gamma}$,
    which implies  
    $
        \psi^{-1}\sbr{C_{\gamma}}= C_{\psi_\#(\gamma)}.  
    $

    \item Let $s \in d+c -C_{\psi_\#(\gamma)}$, then $\psi(s) \in a+b -C_{\gamma}$. Hence,
    \begin{multline*}
        \overleftarrow{
            \psi_\#(\gamma)
        }(s)
        =
        \psi_\#(\gamma)\del{
            (d+c-s)^-
        }
        =
        \psi_\#(\gamma)\del{
            d+c-s
        }
        \\
        =
        \gamma\del{
            \psi\del{
                c+d-s
            }
        }
        =
        \gamma\del{
            a+b-\psi(s)
        }
        \\
        =
        \gamma\del{
            (a+b-\psi(s))^-
        }
        =
        \overleftarrow{\gamma}\del{ \psi(s) }
        =
        \psi_\#(\overleftarrow{\gamma})(s).
    \end{multline*} 
    Therefore, 
    $
        \overleftarrow{
            \psi_\#(\gamma)
        } \rvert_{d+c -C_{\psi_\#(\gamma)}}
        =
        \psi_\#(\overleftarrow{\gamma})
        \rvert_{d+c -C_{\psi_\#(\gamma)}}
    $.
    Since $d+c -C_{\psi_\#(\gamma)}$ is dense in $[c,d]$,
    the equality from third point 
    is a consequence of Lemma
    \ref{lem:right_continuity_and_agreeing_on_a_dense_subset}.

    \item We have previously shown that 
    $
        V\del{ \psi_\#(\gamma) }
        \le 
        V( \gamma)
    $.
    Since $\psi^{-1}$ has the same properties 
    as $\psi$ and $(\psi^{-1})_\# $ is an inverse of $\psi_\#$, we also have
    $
        V\del{ \psi_\#(\gamma) }
        \ge 
        V( \gamma)
    $.
    Hence, 
    $
        V\del{ \psi_\#(\gamma) }
        = 
        V( \gamma)
    $.
    For $s \in C_{\psi_\#(\gamma)}$ let us define function 
    $\psi^s \colon [c,s] \to [a,\psi(s)]$
    by the formula
    \begin{equation*}
        \forall r \in [c,s]
        \qquad 
            \psi^s(r)
            \coloneqq 
            a + 
            \frac{ \psi(s) - a }{ s -c }( r- c).
    \end{equation*}
    By a straightforward calculation we have $\psi^s = \psi \rvert_{[c,s]}$. 
    Then by Remark \ref{rem:restrciting_test_curve_to_continutity_point} we have
    $
        \psi^s_\#\del{ \gamma \rvert_{[a,\psi(s)]}}
        =
        \psi_\#(\gamma) \rvert_{[c,s]}
        \in 
        \testCurves\del{
            [c,s];
            X
        }
    $.
    Therefore, by already proved equality of variations we get
    \begin{equation*}
        V_\gamma( \psi(s))
        =
        V\del{
            \gamma \rvert_{[a, \psi(s)]}
        }
        =
        V\del{
            \psi^s_\#\del{
                \gamma \rvert_{[a,\psi(s)]}
            }   
        }
        =
        V\del{
            \psi_\#\del{
                \gamma
            }    \rvert_{[c, s]}
        }
        =
        V_{\psi_\#\del{
                \gamma
            } }
        (s).
    \end{equation*}
    Therefore, $V_\gamma \circ \psi = V_{\psi_\#\del{
                \gamma
            } }$ on $C_{\psi_\#(\gamma)}$.
        Set $C_{\psi_\#(\gamma)}$ is dense in $[c,d]$ and since $V_\gamma \circ \psi$ and 
    $V_{\psi_\#\del{
                \gamma
            } }$
    are both right-continuous, we have
    $V_\gamma \circ \psi = V_{\psi_\#\del{
                \gamma
            } }$
    by Lemma
    \ref{lem:right_continuity_and_agreeing_on_a_dense_subset}.

    \item We have
    \begin{equation*}
        \mu_{\psi_\#(\gamma)}(\set{c})
        =
        0
        =
        \mu_{\gamma}(\set{a})
        =
        \mu_{\gamma}( \psi\sbr{ \set{c} } )
        =
        (\psi^{-1})_\#(\mu_\gamma)(\set{c}),
    \end{equation*}
    and for $r, s \in [c,d]$ 
    \begin{multline*}
        \mu_{\psi_\#(\gamma)}(\intoc{r, t})
        =
        V_{\psi_\#(\gamma)}(t)
        - V_{\psi_\#(\gamma)}(r)
        =
        V_\gamma( \psi(t))
        - V_\gamma( \psi(r))
        \\
        =
        \mu_\gamma\del{
            \intoc{ \psi(r), \psi(t) }
        }
        =
        \mu_\gamma\del{
            \psi\sbr{ 
                \intoc{ r, t }
            }
        }
        =
        (\psi^{-1})_\#(\mu_\gamma)\del{
            \intoc{ r, t }
        }.
    \end{multline*}
    Thus, by the Dynkin Lemma we conclude that 
    $\mu_{\psi_\#(\gamma)} = (\psi^{-1})_\#(\mu_\gamma)$ on Borel subsets of $[c,d]$.
    
    \item Finally, for the last point, 
    let us note that
    \begin{equation*}
        \int_\gamma f
        =
        \integral{ [a,b] }{ f \circ \gamma }{
            \mu_\gamma 
        }
        =
        \integral{ [c,d]  }{ f \circ \gamma \circ \psi }{
            ((\psi^{-1})_\#\del{\mu_\gamma) }
        }
        =
        \integral{ [c,d]  }{ f \circ (\psi_\#(\gamma)) }{
            \del{\mu_{\psi_\#(\gamma)} }
        }
        =
        \int_{\psi_\#(\gamma)} f.
    \end{equation*}
    Therefore,
    \begin{equation*}
        \sint{\gamma}{f}
        =
        \frac{1}{2}
        \del{
            \int_\gamma f
            +
            \int_{\overleftarrow{\gamma}} f
        }
        =
        \frac{1}{2}
        \del{
            \int_{\psi_\#(\gamma)} f
            +
            \int_{
                \psi_\#\del{\overleftarrow{\gamma}
                }} f
        }
        =
        \frac{1}{2}
        \del{
            \int_{\psi_\#(\gamma)} f
            +
            \int_{
                \overleftarrow{\psi_\#(\gamma)
                }} f
        }
        =
        \sint{\psi_\#(\gamma)}{f}
    \end{equation*}
    which proves the last claim.
    \end{enumerate}
\end{proof}

\begin{definition}[Left-adjusted restriction]
\label{def:left-adjusted_restriction_of_test_cruve}
    Let $(X, \metricAlone)$ be a metric space.
    Let $r, t \in [a,b]$ satisfy $r <t$. 
    For $\gamma \in \testCurves\del{ [a,b]; X}$
    we define its left-adjusted restriction 
    $
        \gamma \rvert_{[r,t^-]}
        \colon 
        [r,t] 
        \to 
        X
    $ by
    \begin{equation*}
        \forall s \in [r,t]
        \qquad 
        \gamma \rvert_{[r,t^-]}(s)
        \coloneqq 
        \begin{cases}
            \gamma(s), &\text{ for }  s \in \intco{r,t}, \\ 
            \gamma(t^-), &\text{ for } s = t.
        \end{cases}
    \end{equation*}
    It is worth noting that 
    $
        \gamma \rvert_{[r,t^-]}
        \in 
        \testCurves\del{ 
            [r,t]; X
        }
    $ 
    and if $\gamma$ is continuous at $t$, then
     $
        \gamma \rvert_{[r,t^-]}
        =
        \gamma \rvert_{[r,t]}
    $.
\end{definition}
\begin{remark} \label{rem}
     Let $(X, \metricAlone)$ be a metric space, $t \in (a,b]$ and $\gamma \in \testCurves\del{ [a,b]; X}$. Then, 
     for all $s \in \intco{a,t}$
    we have $\gamma(s)  = \gamma \vert_{ \intcc{a,t^-}}(s)$
    and $V_\gamma(s) = V_{\gamma \rvert_{ \intcc{a,t^-}}}(s)$.
    Moreover, since $ \gamma \vert_{ \intcc[0]{a,t^-}}$
    is by definition left-continuous at $t$, then 
    by Proposition \ref{rem:properties_of_variation_function} so is $V_{ \gamma \vert_{ \intcc[0]{a,t^-}}}$. In particular $V_{ \gamma \vert_{ \intcc[0]{a,t^-}}}=V_{\gamma}(t^{-})$.
\end{remark}
\begin{lemma}
\label{lem:adding_symmetrized_integrals_over_intervals}
    Let $(X, \metricAlone)$ be a metric space
    and $\gamma \in \testCurves\del{ [a,b]; X}$.
    If $f \colon X \to [0,\infty]$ is Borel, then
    for all $t \in (a,b)$ we have
    \begin{equation*}
        \sint{\gamma}{f}
        =
        \sint{ \gamma \rvert_{[a,t^-]}}{ f}
        +
        \frac{
             f(\gamma(t)) + f(\gamma(t^-))
        }{2}
        \del{
            V_\gamma(t) - V_\gamma(t^-)
        }
        +
        \sint{ \gamma \rvert_{[t, b]}}{ f}.
    \end{equation*}
    Moreover, if $t$ is a point of continuity of $\gamma$, 
    then
    \begin{equation*}
        \sint{\gamma}{f}
        =
        \sint{ \gamma \rvert_{[a,t]}}{ f}
        +
        \sint{ \gamma \rvert_{[t, b]}}{ f}.
    \end{equation*}
\end{lemma}
\begin{proof}
    Let $t \in (a,b)$.
    We have
    \begin{align*}
        \int_\gamma f
        &=
        \integral{[a,b]}{f \circ \gamma}{ \mu_\gamma}
        \\ 
        &=
        \integral{ \intco{a, t} }{
            f \circ \gamma
        }{ \mu_\gamma }
        +  \integral{\set{t}}{
            f \circ \gamma
        }{ \mu_\gamma }
        + \integral{ (t,b] }{
            f \circ \gamma
        }{ \mu_\gamma }.
    \end{align*}
      Therefore by Remark \ref{rem}, we get
    \begin{equation*}
         \integral{ \intco{a, t} }{
            f \circ \gamma
        }{ \mu_\gamma }
        =
         \integral{ \intco{a, t} }{
            f \circ \del{\gamma \rvert_{ \intcc[0]{a,t^-}}}
        }{ \mu_{\gamma \rvert_{ \intcc[0]{a,t^-}}} }
        =
         \integral{ \intcc{a, t} }{
            f \circ \del{\gamma \rvert_{ \intcc[0]{a,t^-}}}
        }{ \mu_{\gamma \rvert_{ \intcc[0]{a,t^-}}} }
        =
        \int_{\gamma \rvert_{ \intcc[0]{a,t^-}}} f.
    \end{equation*}
    For similar reasons we have
    \begin{equation*}
        \integral{ (t,b] }{
            f \circ \gamma
        }{ \mu_\gamma }
        =
        \integral{ (t, b] }{
            f \circ \del{\gamma \rvert_{ \intcc[0]{t, b }}}
        }{ \mu_{\gamma \rvert_{ \intcc[0]{ t, b }}} }
        =
         \integral{ \intcc{t, b} }{
            f \circ \del{\gamma \rvert_{ \intcc[0]{t, b }}}
        }{ \mu_{\gamma \rvert_{ \intcc[0]{ t, b }}} }
        =
        \int_{ \gamma \rvert_{ \intcc[0]{ t, b }} } f.
    \end{equation*}
    Also
    \begin{equation*}
        \integral{\set{t}}{
            f \circ \gamma
        }{\mu_\gamma }
        =
        f \circ \gamma(t)
        \del{
            V_\gamma(t) - V_\gamma(t^-)
        },
    \end{equation*}
    where we used the fact that
    \begin{equation*}
        \mu_\gamma( \set{t} )
        =
        \lim_{n \to \infty} \mu_{\gamma}\del{ \intoc{ t-\frac{1}{n}, t }}
        =
        \lim_{n \to \infty} V_{\gamma}(t) - V_{\gamma}\del{ t -\frac{1}{n}}
        =
        V_{\gamma}(t) - V_{\gamma}(t^-).
    \end{equation*}
    Therefore, for all $t \in (a,b)$ we get
    \begin{equation*}
        \int_\gamma f
        =
        \int_{\gamma \rvert_{ \intcc[0]{a,t^-}}} f
        +
        f \circ \gamma(t)
        \del{
            V_\gamma(t) - V_\gamma(t^-)
        }
        +
        \int_{ \gamma \rvert_{ \intcc[0]{ t, b }} } f.
    \end{equation*}

    Using the previously found formula for 
    $\overleftarrow{\gamma}$ and point $a+b-t$, we have
    \begin{equation*}
        \int_{\overleftarrow{\gamma}} f
        =
        \int_{
            \overleftarrow{\gamma} \rvert_{ \intcc[0]{a,(a+b-t)^-}}
        } f
        +
        f \circ\overleftarrow{\gamma}(a+b-t)
        \del{
            V_{\overleftarrow{\gamma}}(a+b-t) 
            - V_{\overleftarrow{\gamma}}((a+b-t)^-)
        }
        +
        \int_{ 
            \overleftarrow{\gamma} \rvert_{ \intcc[0]{ a+b-t, b }} 
        } f.
    \end{equation*}
    By Lemma \ref{lem:reversing_test_curve} we get
    \begin{multline*}
        f \circ\overleftarrow{\gamma}(a+b-t)
        \del{
            V_{\overleftarrow{\gamma}}(a+b-t) 
            - V_{\overleftarrow{\gamma}}((a+b-t)^-)
        }
        \\= f \circ\overleftarrow{\gamma}(a+b-t)
        \del{
            V(\gamma) - V_{\gamma}(t^-) - (V({\overleftarrow{\gamma}}) - V_{\overleftarrow{\overleftarrow{\gamma}}}(t))} 
        \\= f\circ \gamma(t^-)
        \del{
            V_{\gamma}(t) 
            - V_{\gamma}(t^-)
        }.
    \end{multline*}
    Let $\Phi_t : [a,t] \rightarrow [a+b-t,b]$ be defined as follows $\Phi_t (s)= s+b-t$. Then for $s \in [a,t]$ we have 
    \begin{multline*}
        (\Phi_t)_{\#}  \del{
            \overleftarrow{ \gamma }\rvert_{ [a+b-t,b]  }
        }(s)
            =
         \overleftarrow{ \gamma} \rvert_{ [a+b-t,b]} 
         \del{s+b-t }
         \\
        =
         \overleftarrow{ \gamma}\del{ s + b - t }
         =
         \gamma\del{
            \del{a + b -\del{ s +b - t }}^-
         }
         \\
         =
         \gamma\del{
            \del{a+t-s}^-
         }
         =
         \gamma\rvert_{[a,t^-]}\del{
            \del{a+t-s}^-
         }
         =
         \overleftarrow{ \gamma \rvert_{[a,t^-]} }(s).
    \end{multline*}
    Hence,
    \[
       (\Phi_{t})_{\#}   ( \overleftarrow{\gamma} \rvert_{ [ a+b-t, b]}) = \overleftarrow{ \gamma \rvert_{ [a,t^-] }}
    \]
    and by Proposition \ref{rem:test_Curves_changing_intervals} we have 
    \begin{equation*}
        \int_{ \overleftarrow{\gamma} \rvert_{ [ a+b-t, b]}} f
        =
                \int_{ 
            \overleftarrow{ \gamma \rvert_{ [a,t^-] } }
        } f.
    \end{equation*}

    Next, let 
    $
        \Psi_t \colon [t,b] \to [a, a+b-t]
    $
    by defined as follows
    $
        \Psi_t(s) \coloneqq a + s-t.
    $
    For $s \in \intco{t, b}$ we have
    $
        a+s-t \in \intco{a, a+b-t},
    $
    hence
    \begin{multline*}
        (\Psi_t)_{\#}  \del{
            \overleftarrow{ \gamma }\rvert_{ [a,(a+b-t)^-]  }
        }(s)
            =
         \overleftarrow{ \gamma} \rvert_{ [a,(a+b-t)^-]} 
         \del{ \Psi_t(s) }
         \\
        =
         \overleftarrow{ \gamma} \rvert_{ [a,(a+b-t)^-]} 
         \del{ a+ s - t }
         =
         \overleftarrow{ \gamma}\del{ a+ s - t }
         =
         \gamma\del{
            \del{a + b -\del{ a+ s - t }}^-
         }
         \\
         =
         \gamma\del{
            \del{b+t-s}^-
         }
         =
         \gamma\rvert_{[t,b]}\del{
            \del{b+t-s}^-
         }
         =
         \overleftarrow{ \gamma \rvert_{[t,b]} }(s).
    \end{multline*}
    Also
    \begin{multline*}
        (\Psi_t)_{\#}  
            \del{
                \overleftarrow{ \gamma }\rvert_{ [a,(a+b-t)^-]  } }
            (b)
            =
         \overleftarrow{ \gamma} \rvert_{ [a,(a+b-t)^-]} 
         \del{ \Psi_t(b) }
        =
         \overleftarrow{ \gamma} \rvert_{ [a,(a+b-t)^-]} 
         \del{ a + b -t }
         \\
         =
         \overleftarrow{ \gamma}\del{ \del{ a + b- t }^- }
         =
         \overleftarrow{ \overleftarrow{ \gamma}}(t)
         =
         \gamma(t)
         =
         \gamma \rvert_{[t,b]}(t)
         =
         \gamma \rvert_{[t,b]}( (t+b-b)^- )
         =
         \overleftarrow{ \gamma\rvert_{[t,b]}}(b),
    \end{multline*}
    where we have used part (b) of Lemma \ref{lem:reversing_test_curve}.
    Therefore, we conclude that
    \[
       (\Psi_{t})_{\#}   
       \del{
        \overleftarrow{\gamma} \rvert_{ [a, (a+b-t)^-]}
        } 
        = 
        \overleftarrow{ \gamma\rvert_{[t,b]}}
    \]
    and by Proposition \ref{rem:test_Curves_changing_intervals} we have 
    \begin{equation*}
        \int_{ 
            \overleftarrow{\gamma} \rvert_{ [a, (a+b-t)^-]}
        } f
        =
                \int_{ 
            \overleftarrow{ \gamma\rvert_{[t,b]}}
        } f.
    \end{equation*}
    Therefore we can write
    \begin{equation*}
        \int_{\overleftarrow{\gamma}} f
        =
        \int_{
             \overleftarrow{ \gamma \rvert_{ [t,b] } }
        } f
        +
        f \circ \gamma(t^-)
        \del{
            V_{\gamma}(t) 
            - V_{\gamma}(t^-)
        }
        +
        \int_{ 
            \overleftarrow{ \gamma \rvert_{ [a,t^-] } }
        } f,
    \end{equation*}
    and finally, we conclude 
    \begin{align*}
        \sint{\gamma}{f}
        &=
        \frac{1}{2}\del{
            \int_\gamma f 
            +
            \int_{\overleftarrow{\gamma}} f
        }
        \\
        &=
        \frac{1}{2}\left(
            \int_{\gamma \rvert_{ \intcc[0]{a,t^-}}} f
            +
            f \circ \gamma(t)
            \del{
                V_\gamma(t) - V_\gamma(t^-)
            }
            +
            \int_{ \gamma \rvert_{ \intcc[0]{ t, b }} } f
            \right.
        \\
        &\qquad \qquad +
        \left. 
            \int_{
                     \overleftarrow{ \gamma \rvert_{ [t,b] } }
                } f
                +
                f \circ \gamma(t^-)
                \del{
                    V_{\gamma}(t) 
                    - V_{\gamma}(t^-)
                }
                +
                \int_{ 
                    \overleftarrow{ \gamma \rvert_{ [a,t^-] } }
                } f
        \right)
        \\
        &=
        \sint{ \gamma \rvert_{[a,t^-]}}{ f}
        +
        \frac{
             f(\gamma(t)) + f(\gamma(t^-))
        }{2}
        \del{
            V_\gamma(t) - V_\gamma(t^-)
        }
        +
        \sint{ \gamma \rvert_{[t, b]}}{ f}
    \end{align*}
    as claimed.

    Now, let $t \in (a,b)$ be a point of continuity of $\gamma$. 
    It is also a point of continuity of $V_\gamma$ by Proposition \ref{rem:properties_of_variation_function}. 
    Hence, $V_\gamma(t) - V_\gamma(t^-) = 0$.
    Moreover, from the definition
    of 
    $\gamma \rvert_{[a,t^-]}$
    we have
     $
        \gamma \rvert_{[r,t^-]}
        =
        \gamma \rvert_{[r,t]}
    $.
    Therefore
    \begin{equation*}
        \sint{\gamma}{f}
        =
        \sint{ \gamma \rvert_{[a,t^-]}}{ f}
        +
        \frac{
             f(\gamma(t)) + f(\gamma(t^-))
        }{2}
        \del{
            V_\gamma(t) - V_\gamma(t^-)
        }
        +
        \sint{ \gamma \rvert_{[t, b]}}{ f}
        =
        \sint{ \gamma \rvert_{[a,t]}}{ f}
        +
        \sint{ \gamma \rvert_{[t, b]}}{ f}
    \end{equation*}
    as needed.
\end{proof}
\begin{lemma} \label{lem::descending_family_of_subsets_to_0}
       Let $(X, \metricAlone)$ be a metric space and $\gamma \in \testCurves{([a,b]; X)}$.
    If $f \colon X \to [0,\infty]$ is Borel and such that 
    $\sint{ \gamma } f < \infty$, then:
    \begin{itemize}
        \item If $t \in \intoc{a,b}$ and $r_n \to t^-$, then 
        $
            \lim_{n \rightarrow \infty}\sint{\gamma \rvert_{ [r_n, t^-] }} f =0,
        $
        \item If $t \in \intco{a,b}$ and $r_n \to t^+$, then 
        $
           \lim_{n \rightarrow \infty} \sint{\gamma \rvert_{ [t, r_n^-] }} f 
            +\frac{ f(\gamma(r_n^-)) + f(\gamma(r_n))}{2} 
            \metric\del{ \gamma(r_n^-), \gamma(r_n)}
            =0.
        $
       \end{itemize}
\end{lemma}
\begin{proof}
    First of all let us observe that if $c, d \in [a,b]$ and $c<d$, then 
     \begin{equation} \label{zal}
        \int_{ \overleftarrow{\gamma} \rvert_{ [ a+b-d, a+b-c]}} f
        =
                \int_{ 
            \overleftarrow{ \gamma \rvert_{ [c,d^-] } }
        } f.
    \end{equation}
    Indeed, let $\Phi : [c,d] \rightarrow [a+b-d,a+b-c]$ be defined as follows $\Phi (s)= a+b -c -d +s$. Then 
    \[
       (\Phi)_{\#}   ( \overleftarrow{\gamma} \rvert_{ [ a+b-d, a+b-c]}) = \overleftarrow{ \gamma \rvert_{ [c,d^-] }}
    \]
    and by Proposition \ref{rem:test_Curves_changing_intervals} we get (\ref{zal}). 
   
    Without loss of generality we can assume that $r_n$ is strictly monotone. 
    
        First, let us assume that $t \in \intoc{a,b}$ and $r_n \to t^-$. Since $\lim_{n\rightarrow \infty}\mu_{\gamma}((r_n, t))=\mu_{\gamma}(\bigcap_{n \in \bN}  \intoo{r_n, t }) =0$, $\lim_{n\rightarrow \infty} \mu_{\overleftarrow{\gamma}}\del{(a+b-t, a+ b -r_n)}=0$,  and     $\sint{ \gamma } f < \infty$, having in mind (\ref{zal}) we get
    \begin{align*}
       2 \sint{\gamma \rvert_{ [r_n, t^-] }} f
        &=
        \int_{\gamma \rvert_{ [r_n, t^-] }} f
        + \int_{ \overleftarrow{\gamma \rvert_{ [r_n, t^-] }}} f
        \\ 
        &=
        \int_{\gamma \rvert_{ [r_n, t^-] }} f
        + \int_{ \overleftarrow{\gamma } \rvert_{ [a+b-t, a+b-r_n] }} f
        \\ 
        &=
        \int_{ \intoo{r_n, t }} f\circ \gamma \  \mathrm{d} \mu_{\gamma}
        + \int_{ \intoo{a+b-t, a+ b -r_n } } f \circ \overleftarrow{\gamma} \ \mathrm{d}\mu_{\overleftarrow{\gamma}}
        \xrightarrow{n \to \infty}
        0.
    \end{align*}

    Next, in the same manner, assuming that $t \in \intco{a,b}$ and $r_n \to t^+$ we have
    \begin{align*}
        2\sint{\gamma \rvert_{ [t, r_n^-] }} f
        &=
        \int_{\gamma \rvert_{ [t, r_n^-] }} f
        + \int_{ \overleftarrow{\gamma \rvert_{ [t, r_n^-] }}} f
        \\ 
        &=
        \int_{\gamma \rvert_{ [t, r_n^-] }} f
        + \int_{ \overleftarrow{\gamma } \rvert_{ [a+b -r_n,a+ b-t] }} f
        \\ 
        &=
        \int_{ \intoo{t, r_n }} f \circ \gamma \ \mathrm{d}\mu_{\gamma}
        + \int_{ \intoo{a+b-r_n, a+b-t } } f \circ \overleftarrow{\gamma} \ \mathrm{d}\mu_{\overleftarrow{\gamma}}
        \xrightarrow{n \to \infty}
        0.
    \end{align*}
       Finally, by Lemma \ref{lem:adding_symmetrized_integrals_over_intervals} we have
    \begin{equation*}
        \sum_{n=2}^\infty 
        \frac{ f(\gamma(r_n^-)) + f(\gamma(r_n))}{2} 
            \metric\del{ \gamma(r_n^-), \gamma(r_n)}
                \le 
        \sint{\gamma} f
        < 
        \infty,
    \end{equation*}
    hence 
    $
        \frac{ f(\gamma(r_n^-)) + f(\gamma(r_n))}{2} 
            \metric\del{ \gamma(r_n^-), \gamma(r_n)}
        \to 
        0
    $
    as $n \to \infty$, which ends the proof.
\end{proof}

\section{Topology on $TC$} \label{sec::topology_on_TC}
\begin{definition}
    Let $(X, \metricAlone)$ be a metric space.
    By $\measurablePre\del{ [a,b]; X}$ we will
    denote the space of 
    Borel functions\footnote{ If we assume {\bf CH} then by Remark \ref{lem::image_of_Borel_from_interval_is_separable} every Borel map     $\gamma \colon [a,b] \to X$ has separable image.}  
    $\gamma \colon [a,b] \to X$ with separable image.
    On the space $\measurablePre\del{ [a,b]; X}$
    we introduce an equivalence relation $\sim$
    by
    \begin{equation*}
        \forall 
            \gamma, \gamma' \in \measurablePre\del{ [a,b]; X}
        \qquad 
            \gamma \sim \gamma' 
            \iff 
            \gamma = \gamma '
            \text{almost everywhere}.
    \end{equation*}
    We define space 
    $\measurable\del{ [a,b]; X}$
    as a quotient
    $
        \measurable\del{ [a,b]; X}
        \coloneqq 
        \measurablePre\del{ [a,b]; X } / {\sim}
    $.
    For simplicity, we will also refer to elements of
    $\measurable\del{ [a,b]; X}$
    as Borel functions when we can refer
    to a representative of its equivalence class.
    On space $\measurable\del{ [a,b]; X}$
    we define a metric 
    $
        \metricAlone_{\measurable}
        \colon 
        \measurable\del{ [a,b]; X}
        \times 
        \measurable\del{ [a,b]; X}
        \to 
        \intco{ 0, \infty}
    $
    by the formula\footnote{By Proposition \ref{mierzalnosc} the map $t \mapsto \min \del{ 1, \metric \del{\gamma(t), \gamma'(t)}}$ is measurable.}
    \begin{equation*}
        \forall 
            \gamma, \gamma' 
            \in \measurable\del{ [a,b]; X}
        \qquad 
            \metric_{\measurable}\del{
                \gamma, \gamma'
            }
            \coloneqq 
            \integral{ [a,b] }{ 
                \min \del{
                    1, 
                    \metric \del{
                        \gamma(t), 
                        \gamma'(t)
                    }
                }
            }{t}.
    \end{equation*} 
\end{definition}

\begin{proposition}
    Let $(X, \metricAlone)$ be a metric space, then metric $\metricAlone_{\measurable}$ metrizes 
    convergence in the Lebesgue measure $\lambda$. Moreover:
    \begin{enumerate}[label=\roman*)]
        \item 
            Let $\gamma, \gamma_n \in \measurable\del{ [a,b]; X}$ for all $n \in \bN$.  Suppose that $\gamma_n \to \gamma$ in $\del{\measurable\del{ [a,b]; X}, \metricAlone_{\measurable}}$.
            Then there exists a subsequence $(\gamma_{n_k})_k$ such that 
            $\gamma_{n_k} \to \gamma$ almost everywhere.
        \item
            The space $
                \del{
                    \measurable\del{ [a,b]; X},
                    \metricAlone[\metricAlone_{\measurable}]  
                }
            $ is complete if and only if $(X, \metricAlone)$ is complete.
        \item
            The space     
            $
                \del{
                    \measurable\del{ [a,b]; X},
                    \metricAlone[\metricAlone_{\measurable}]  
                }
            $ is separable if and only if $(X, \metricAlone)$ is separable.
    \end{enumerate}
\end{proposition}
\begin{proof}
     Let $\gamma, \gamma_n \in \measurable\del{ [a,b]; X}$ for all $n \in \bN$. 
    Suppose that $\gamma_n \to \gamma$ in $\metricAlone[\metricAlone_{\measurable}]$.
    Then for all $\eps > 0$
    \begin{equation*}
        \lambda \left(
            \set{
                \metric\del{ \gamma_n(t), \gamma(t)} > \eps 
            }
        \right)
        \leq
        \integral{ \set{ \metric \, \del[0]{ \gamma_n(t), \gamma(t) }  > min(1,\eps) } }{
            1
        }{t}
        \leq 
        \frac{b-a}{ \min(1,\eps)} \metric_{\measurable}\del{
            \gamma_n,
            \gamma
        }
        \xrightarrow{ n \to \infty} 0,
    \end{equation*}
    so $\gamma_n \to \gamma$ in measure.    
    Now suppose $\gamma_n \to \gamma$ in measure.
    Fix $\eps > 0$. 
    Then 
    \begin{equation*}
        \metric_{\measurable}\del{
            \gamma_n,
            \gamma
        }
        \le 
        \eps (b-a)
        + 
        \lambda \left(
            \set{
                \metric\del{ \gamma_n(t), \gamma(t)} > \eps 
            }
        \right)
        \xrightarrow{ n \to \infty}
        \eps (b-a).
    \end{equation*}
    As $\eps > 0$ is arbitrary, we have that $\gamma_n \to \gamma$ in 
    $\metricAlone_{\measurable}$.

    \begin{enumerate}[itemindent=2em,leftmargin=0em, label=\roman*), itemsep=0.5em]
        \item If $\gamma_n \to \gamma$ in $\del{\measurable\del{ [a,b]; X}, \metricAlone_{\measurable}}$, then     $
        \metric\del{ \gamma_{n}, \gamma} \rightarrow 0
    $
    in Lebesgue measure. Therefore, by the classical Riesz Theorem, there exists a subsequence such that for almost every $t \in [a,b]$ we have $\metric \del{\gamma_{n_k}(t) , \gamma(t)} \rightarrow 0$. 

    \item Now, let assume that $X$ is complete. 
    Let $(\gamma_n)_n$ be a Cauchy sequence in $\measurable\del{ [a,b]; X}$. There exists strictly increasing sequence $n_k$ such that
    \[
        \lambda\left(\set{
                \metric\del{ \gamma_{n_k}(t), \gamma_{n_{k+1}}(t)} > 1/2^k 
            }\right) < 1/2^k.
    \]
    For $k \in \mathbb{N}$ we define
    \[
        A_k =\bigcup_{j=k}^{\infty}\set{
                \metric\del{ \gamma_{n_k}(t), \gamma_{n_{k+1}}(t)} > 1/2^k}, \quad A = \bigcap_{k=1}^{\infty}A_k.
    \]
    Then, $\lambda(A_k) \rightarrow 0$ and $\lambda(A)=0$. Let us observe that for $t\in [a,b] \setminus A$ we have that $\gamma_{n_j}(t)$ is Cauchy in $X$. Indeed, if $t \notin A_k$ and $i \geq j \geq k$, then we have
    \begin{eqnarray}\label{cc}
    \metric\del{ \gamma_{n_j}(t), \gamma_{n_{i}}(t)} \leq  1/2^{j-1}.
    \end{eqnarray}
    Since $X$ is complete, we can define the following map $\gamma: [a,b]\setminus A \rightarrow X$, $\gamma(t)=\lim_{j\rightarrow \infty}\gamma_{n_j}(t)$. 
    Let us observe that for every closed set $D \subset X$ we have\footnote{$(D)_{1/l}=\bigcup_{x\in D}B(x,1/l)$}
    \[
        \gamma^{-1}(D)=  ([a,b]\setminus A) \cap\bigcap_{l=1}^{\infty}\bigcup_{k=1}^{\infty}\bigcap_{j=k}^{\infty}\gamma_{n_j}^{-1}((D)_{1/l}).
    \]
    Therefore, the function $\tilde{\gamma}:[a,b] \rightarrow X$ 
    defined by
    \begin{equation*}
       \tilde{\gamma}(t)
        \coloneqq 
        \begin{cases}
            \gamma(t), &\text{ for } t \in [a,b] \setminus A, \\
            \tilde{x}, &\text{ for } t \in A,
        \end{cases}
    \end{equation*} 
    where $\tilde{x}$ is a fixed point from $X$, is a Borel map.
    Moreover, for a fixed $\eps>0$ and $k \in \mathbb{N}$, by (\ref{cc}) we have 
    \[
        ([a,b]\setminus A) \cap \set{
                \metric\del{ \gamma_{n_j}(t), \tilde{\gamma}(t)} > \eps } \subset A_k,
    \]
    for $j \geq k$ such that $\eps/2 > 1/2^j$. Therefore, since $\lambda(A_k) \rightarrow 0$, we have that $\gamma_{n_j} \to \tilde{\gamma}$ in $\metricAlone_{\measurable}$. 
    Hence, since $\gamma_n$ is a Cauchy sequence, $\gamma_n \to \tilde{\gamma}$ in $\metricAlone_{\measurable}$.
    
    Next, let us suppose that  the space $
        \del{
            \measurable\del{ [a,b]; X},
            \metricAlone[\metricAlone_{\measurable}]  
        }
    $ is complete. We shall prove that $(X, \metricAlone)$ is complete. For this purpose we fix a Cauchy sequence $y_n$ in $X$. Now, we define the sequence $f_n:[a,b] \rightarrow X$ as follows $f_n(t)=y_n$. It is obvious that $f_n$ is a Cauchy sequence in $
        \del{
            \measurable\del{ [a,b]; X},
            \metricAlone[\metricAlone_{\measurable}]  
        }$. Therefore, there exists $f\in \del{
            \measurable\del{ [a,b]; X},
            \metricAlone[\metricAlone_{\measurable}]  
        }$ such that $f_n \to f$ in $\metricAlone_{\measurable}$. Furthermore, by $i)$ there exists a subsequence $f_{n_j}$ and a set of full measure $D\subset [a,b]$ such that 
        $\metricAlone (f_{n_j}(t),f(t)) \rightarrow 0$ for $t\in D$. Let $t_0 \in D$, then in particular $\metricAlone (y_{n_j},f(t_0)) \rightarrow 0$. Therefore, since $y_n$ is a Cauchy sequence, we have $y_n \rightarrow f(t_0)$ in $(X, \metricAlone)$.

    \item  
    Let us assume that $X$ is separable and let $S$ be a countable dense subset. For $n \in \bN$ and $i \in [n]$ denote
    $
        A_{i,n} \coloneqq \intco{ \tfrac{i-1}{n}(b-a), \tfrac{i}{n}(b-a) }.
    $
    We will show that 
    \begin{equation*}
        D=\bigcup_{n=1}^{\infty}D_n \, \text{where}\
        D_n 
        \coloneqq 
        \set{
            t \mapsto 
            \sum_{i=1}^n x_i \indicator{ A_{i,n}}(t)
            \given
            x_i \in S
            \text{ for all } i \in [n] 
        }
    \end{equation*}
    is a countable dense subset in $\measurable\del{ [a,b]; X}$.
    It is clear that $D$ is countable. 
    
    Fix $\eps > 0$ and $\gamma \in \measurable\del{ [a,b]; X}$. 
    By Lusin's Theorem\footnote{Separabilty of $X$ allows us to use the Lusin Theorem, see \cite{federer}},
    there exists a compact $F \subseteq [a,b]$ 
    such that $\abs{ [a,b] \setminus F } \le \eps/3$ 
    and $\gamma \rvert_{F}$ is continuous. 
    As $F$ is compact, $\gamma \rvert_{F}$ is absolutely continuous and there exists $\delta > 0$ such that 
    if $\abs{s - t} \le \delta$, then $\metric\del{ \gamma(s), \gamma(t) } \le \eps/4$. 
    Let $n \in \bN$ be such that $\delta \ge (b-a)/n$.
    Then $\diam\del{ A_{i,n} } \le \delta$.

    For every $i \in [n]$, if $F \cap A_{i,n} \ne \emptyset$, let 
    $x_i \in S$ be such that 
    $
        \overline{B}\del{x_i, \frac{\eps}{3(b-a)}} \cap \gamma \sbr{
            F \cap A_{i,n}
        }\ne \emptyset;
    $
    otherwise, let $x_i$ be arbitrary. 
    Let $\gamma' \in \measurable\del{ [a,b]; X}$ be defined by the formula
    $\gamma'(t) \coloneqq \sum_{i=1}^n x_i \indicator{ A_{i,n}}(t)$. 
    Then $\gamma' \in D$. 
    Furthermore, if $t \in F\setminus \set{b}$, then $\metric\del{ \gamma'(t), \gamma(t) } \le \frac{2\eps}{3(b-a)}$. 
    Indeed, for such a $t$ there exists $i \in [n]$ such that $t \in A_{i,n} \cap F$. Thus, there exists
    $t_0 \in A_{i,n}\cap F$ such that 
    $
        \gamma(t_0) \in \overline{B}\del{x_i, \frac{\eps}{3(b-a)}},
    $
    and since $\diam\del{ A_{i,n} } \le \delta$ we have $\abs{t_0 - t} \le \delta$.
        Then
    $$
        \metric\del{ \gamma'(t), \gamma(t) }
        =
        \metric\del{ x_i, \gamma(t) }
        \le 
        \metric\del{ x_i, \gamma(t_0) } + \metric\del{ \gamma(t_0), \gamma(t) }
        \le 
        \frac{2\eps}{3(b-a)}.
    $$
    Finally,
   \begin{align*}
        \metric_{\measurable}\del{
                \gamma, \gamma'
            }
            &=
            \integral{ [a,b] }{ 
                \min \del{
                    1, 
                    \metric \del{
                        \gamma(t), 
                        \gamma'(t)
                    }
                }
            }{t}
            \\
            &=
            \integral{ [a,b] \setminus F }{ 
                \min \del{
                    1, 
                    \metric \del{
                        \gamma(t), 
                        \gamma'(t)
                    }
                }
            }{t}
            +
            \integral{ F }{ 
                \min \del{
                    1, 
                    \metric \del{
                        \gamma(t), 
                        \gamma'(t)
                    }
                }
            }{t}
        \\ 
        &\le 
        \integral{ [a,b] \setminus F }{ 
                1
            }{t}
            +
            \integral{ F }{ 
                \frac{2\eps}{3(b-a)}
            }{t}
        \\ 
        &\le 
         \eps,
    \end{align*}
    which proves that $D$ is dense in $\measurable\del{ [a,b]; X}$.

    Next, we shall show the converse implication. For this we suppose that $(X, \metricAlone )$ is not separable. Hence, there exists $\delta >0$ and uncountable $\delta$-separated set $X_{\delta} \subset X$. Then 
    \[
    Y_{\delta} :=\{\gamma : [a,b] \rightarrow X: \gamma \equiv x, \text{where}\, x \in X_{\delta}\}
    \]
    is uncountable $\min(1, \delta)$-separated set in $\measurable\del{ [a,b]; X}$. Therefore, $\measurable\del{ [a,b]; X}$ is not separable.
    \end{enumerate}
\end{proof}

 By Lemma \ref{lem:right_continuity_and_agreeing_on_a_dense_subset} the canonical immersion 
    $
        T
        \colon 
        \testCurves\del{
            [a,b]; 
            X
        }
        \to 
        \measurable\del{
            [a,b];
            X
        }
    $
    is injective. 
    Hence, the map $ \metricAlone_{TC}:  \testCurves\del{
            [a,b]; 
            X
        } \times \testCurves\del{
            [a,b]; 
            X
        } \rightarrow [0, \infty)
    $ defined by the formula
    \begin{equation*}
        \forall    
            \gamma, \gamma' 
            \in 
            \testCurves\del{
                [a,b];
                X
            }
        \qquad 
            \metricAlone_{TC}
            \del{ \gamma, \gamma'}
            \coloneqq 
            \metric[\metric_{\measurable}]
            \del{ 
                T( \gamma), 
                T( \gamma')
            }
    \end{equation*}
    is a metric on $\testCurves\del{
                [a,b];
                X
            }$.

\begin{definition}[Essential variation]
    Let $(X, \metricAlone)$ be a metric space.
    For $\gamma \in \measurable\del{ [a,b]; X}$
    we define its essential variation
    $
        \essentialVariation 
        \colon 
        \measurable\del{ [a,b]; X}
        \to 
        \intcc{ 0, \infty}
    $
    by the formula
    \begin{equation*}
        \forall 
            \gamma \in 
            \measurable\del{ [a,b]; X}
        \qquad 
            \essentialVariation (\gamma)
            \coloneqq 
            \inf 
            \set{
                V(\gamma')
                \given 
                \gamma'
                \sim \gamma 
                            }.
    \end{equation*}
\end{definition}
\begin{lemma}
\label{lem:TC_representatives_have_minimal_variation}
    Let $(X, d)$ be a metric space. 
    Let 
    $
        \wave{\gamma}
        \in \measurable\del{ 
            [a,b];
            X
        }
    $
    be such that there exists
    $
        \gamma \in 
        \testCurves\del{ 
            [a,b];
            X
        }
    $
    such that $\wave{\gamma} = \gamma$ almost everywhere.
    Then 
    $
        \essentialVariation\del{ \wave{\gamma}}
        =
        V( \gamma)
    $.
\end{lemma}
\begin{proof}
    Of course we have $\essentialVariation\del{ \wave{\gamma}} \leq V( \gamma)$. Next, let 
        $
        \gamma' 
        \in 
        \boundedVariation\del{
            [a,b];
            X
        }
    $
    be such that $\wave{\gamma} = \gamma'$ almost everywhere.
    We also have that $\gamma' = \gamma$ almost everywhere, so we have set $D$ of full measure in 
    $[a,b]$ (hence, dense)
    such that $\gamma' = \gamma$ in $D$. Since $D$ is dense in $[a,b]$, there exists a normal sequence  $\del{(t^n_i)_{i=0}^{m_n}}_n = (\Delta_n)_n$ of 
    partitions of $[a,b]$ such that for all $n \in \bN$ we have
    $
        \Delta_n
        \setminus \set{a,b}
        \subseteq 
        D
    $.
    We have
    $
        \abs{ \Delta_n } \to 0
    $,
    so 
    $
        t^n_{1} \to t^n_0 = a
    $
    and 
    $
        t^n_{m_n-1} \to t^n_{m_n} = b
    $.
    Hence, since $\gamma \in \testCurves\del{ [a,b]; X}$,
    we have
    \begin{equation*}
        \metric\del{
                \gamma\del{ t^n_{1} },
                \gamma(a)
            }
        \xrightarrow{ n \to \infty }
        0 
        \quad \text{ and } \quad 
            \metric\del{
                \gamma\del{b },
                \gamma( t^n_{m_n-1})
            }
        \xrightarrow{ n \to \infty }
        0. 
    \end{equation*}
     By Proposition \ref{cor:variation_is_a_uniform_limit_of_Delta-variations}
    we have that $V^{\Delta_n}(\gamma) \to V(\gamma)$ as $n \to \infty $, and therefore
    \begin{align*}
        V(\gamma')
        &\ge 
        \limsup_{n \to \infty}
        V^{\Delta_n}(\gamma')
        \\
        &=
        \limsup_{n \to \infty}
            \metric\del{
                \gamma'\del{ t^n_{1} },
                \gamma'(a)
            }
            +
            \sum_{i=2}^{m_n-1}
                \metric\del{
                \gamma( t^n_{i} ),
                \gamma\del{ t^n_{i-1} }
            }
            +
            \metric\del{
                \gamma'\del{b },
                \gamma'( t^n_{m_n-1})
            }
        \\
        &=
        \limsup_{n \to \infty}
            \metric\del{
                \gamma'\del{ t^n_{1} },
                \gamma'(a)
            }
            -
            \metric\del{
                \gamma\del{ t^n_{1} },
                \gamma(a)
            }
            +
            V^{\Delta_n}(\gamma)
            -
            \metric\del{
                \gamma\del{b },
                \gamma( t^n_{m_n-1})
            }
            +
            \metric\del{
                \gamma'\del{b },
                \gamma'( t^n_{m_n-1})
            }
        \\
        &=
        V(\gamma)
        +
        \limsup_{n \to \infty}
            \metric\del{
                \gamma'\del{ t^n_{1} },
                \gamma'(a)
            }
            +
            \metric\del{
                \gamma'\del{b },
                \gamma'( t^n_{m_n-1})
            }
        \\ 
        &\ge 
        V(\gamma).
    \end{align*}
    Hence,  $V(\gamma') \ge  V(\gamma)$ and in this way we have $
        \essentialVariation\del{ \wave{\gamma}}
        \geq
        V( \gamma)
    $.
\end{proof}
\begin{lemma}
\label{lem:in_complete_metric_spaces_finite_essential_variation_has_TC_representative}
    Let $(X, \metricAlone)$ be a complete metric space. Let 
   $\wave{\gamma} \in  \measurable\del{[a,b];X}$ be such that $  \essentialVariation( \wave{\gamma})
        < \infty  $, then there exists $\gamma \in 
                \testCurves\del{
                    [a,b];
                    X
                }$ such that $\wave{\gamma} = \gamma$ almost everywhere.
\end{lemma}
\begin{proof}
   There exists 
    $
        \gamma' \in \boundedVariation\del{
            [a,b]; 
            X
        }   
    $
    such that 
    $
        \essentialVariation( \wave{\gamma})
        \le 
        V( \gamma' )
    $
    and $\wave{\gamma} = \gamma'$ almost everywhere.
    By Lemma \ref{cor:variation_bounds_sum_of_jumps_from_above}  the set $[a,b] \setminus C_{\gamma'}$ of points of discontinuity of $\gamma'$ is at most countable. Since $(X, \metricAlone)$ is complete, by Lemma \ref{lem:functions_of_bounded_variation_have_cauchy_limits_everywhere::eq:almost_left_limits} the quantities $\gamma' \rvert_{C_{\gamma'}} (t^+)$ and $\gamma' \rvert_{C_{\gamma'}} (b^-)$ are well defined.
    We define function 
    $
        \gamma 
        \colon [a,b] \to X    
    $
    by the formula
    \begin{equation*}
        \forall t \in [a,b]
        \qquad 
        \gamma(t)
        \coloneqq 
        \begin{cases}
            \gamma' \rvert_{C_{\gamma'}} (t^+)
            & \text{ if } 
            t \in \intco{a,b},  \\
            \gamma' \rvert_{C_{\gamma'}} (b^-)
            &\text{ if } t= b.
        \end{cases} 
    \end{equation*} 
    Then, by Lemma \ref{lem:right_continuity_and_agreeing_on_a_dense_subset}
and by simple considerations we have  $
        \gamma 
        \in \testCurves\del{
            [a,b];
            X
        }.
    $
    Moreover, since $\gamma \rvert_{C_{\gamma'}} = \gamma' \rvert_{C_{\gamma'}}$
    and $C_{\gamma'}$ is of full measure in $[a,b]$,
    then $\gamma = \wave{\gamma}$ almost everywhere.
\end{proof}
\begin{lemma}
\label{lem:almost_lower_semicontinuity_of_essential_variation}
    Let $(X, d)$ be a metric space and $M \ge 0$. 
    Let $(\wave{\gamma}_n)_n$ be a sequence of elements 
    of $\measurable\del{ [a,b]; X}$ such that
    $\wave{\gamma}_n \to \wave{\gamma}$ in 
    $\metricAlone[\metricAlone_{\measurable}]$
    and $\essentialVariation( \wave{\gamma}_n) \le M$
    for all $n \in \bN$. 
    If
    $(X, \metricAlone)$ is complete, or
    $\wave{\gamma}$ has a right-continuous representative,
    then $\essentialVariation(\wave{\gamma}) \le M$.
\end{lemma}
\begin{proof}
    Let $\eps > 0$, then there exists a
    sequence $(\gamma_n)$ of elements of 
    $
        \boundedVariation\del{
            [a,b];
            X
        }
    $
    such that for all $n \in \bN$
    $
        V(\gamma_n) \le M + \eps
    $
    and $\gamma_n = \wave{\gamma}_n$
    almost everywhere. If $\wave{\gamma}$ has a right-continuous representative, then let us denote it by $\gamma$.
    Otherwise, let $\gamma$ be any representative of 
    of $\wave{\gamma}$.
    There exists a subsequence $(\gamma_{n_k})_k$
    such that $\gamma_{n_k} \to \gamma$
    almost everywhere.
    Let $D$ be a set of full measure in $[a,b]$
    such that $\gamma_{n_k} \to \gamma$
    everywhere in $D$. 
    Let us also require that $b \notin D$.
    
    Let $v_k \coloneqq V_{\gamma_{n_k}}$.
    Then $(v_k)_k$ is a sequence of 
    non-decreasing functions bounded by 
    $ M + \eps$. By Helly's Selection Theorem
    there exists a non-decreasing function $v \colon [a,b] \to [0, M+\eps]$ 
    and a subsequence $(v_{k_m})_m$
    such that $v_{k_m} \to v$ everywhere.

    For all $m \in \bN$ and $s,t  \in D$ such that
    $s \le t$ we have
    \begin{equation*}
        \metric\del{
            \gamma_{n_{k_m}}(s),
             \gamma_{n_{k_m}}(t)
        }
        \le 
        v_{k_m}(t) - v_{k_m}(s).
    \end{equation*}
    Hence, by passing with $m \to \infty$, for $s,t  \in D$ such that
    $s \le t$ we have
    \begin{equation*}
        \metric\del{
            \gamma(s),
            \gamma(t)
        }
        \le 
        v(t) - v(s).
    \end{equation*}
    Let $t \in \intoc{a,b} \setminus D$ and $(r_n)_n \subset D$ 
    such that $r_n \to t^+$, then since $ v \in 
        \boundedVariation\del{
            [a,b];
            [0,\infty)
        }
    $ by the above inequality we have $(\gamma(r_n))_n$ is a Cauchy sequence which is convergent either because $\gamma$ is right-continuous, or $(X, \metricAlone)$ is complete. Hence, $\gamma\rvert_D(t^+)$ exists for $t \in \intoc{a,b} \setminus D$. Therefore, we can define 
    $
        \hat{\gamma}
        \colon 
        \intco{a,b} 
        \to 
        X
    $
    by the formula
    \begin{equation*}
        \forall 
        t \in \intco{a, b}
        \qquad 
        \hat{\gamma}(t)
        \coloneqq 
        \begin{cases}
            \gamma(t),
            &\text{ if } t \in D, \\
            \gamma\rvert_D(t^+),
            &\text{ if } t \in [a,b) \setminus D,
        \end{cases}
    \end{equation*}
    and let us define $\hat{v} \colon [a,b] \to \bR$
    by the formula
    \begin{equation*}
        \forall  t \in [a,b]
        \qquad 
        \hat{v}(t)
        \coloneqq 
        \begin{cases}
            v(t), &\text{ if } t \in D, \\
            v \rvert_D(t^+), &\text{ if } 
            t \in \intco{ a, b} \setminus D, \\
            M+2\eps, &\text{ if } t = b.
        \end{cases}
    \end{equation*}
    Let us note that  $\hat{v}(b)\ge \hat{v}(b^-) + \eps$ as $v$ was bounded by $M + \eps$ from above.
        Then, since $D$ is dense, for all $t, s \in \intco{a, b}$ such that
    $s < t$ we have
    \begin{equation*}
        \metric\del{
            \hat{\gamma}(s),
            \hat{\gamma}(t)
        }
        \le 
        \hat{v}(t) - \hat{v}(s).
    \end{equation*}
    Since $\hat{v}$ is non-decreasing and bounded, the limit $\hat{v}(b^-)$ exists. Hence, there exists $\delta > 0$ such that,
    if $r, s \in (b-\delta, b)$ and $s < r$, then
    \begin{equation*}
        \metric\del{
            \hat{\gamma}(r),
             \hat{\gamma}(s)
        }
        \le 
        \hat{v}(r) - \hat{v}(s)
        \le 
        \eps.
    \end{equation*}
    Hence, if we define $\hat{\gamma}(b) = \hat{\gamma}(t)$, where $t \in (b- \delta, b)$
    is arbitrary, then for all $r, s \in (b-\delta, b]$ such that $s < r$ we have
    \begin{equation*}
        \metric\del{
            \hat{\gamma}(r),
             \hat{\gamma}(s)
        }
        \le 
        \hat{v}(r) - \hat{v}(s).
    \end{equation*}
    Note that we used the fact that $\hat{v}(b)\ge \hat{v}(b^-) + \eps$.
    Finally, for $
         r, s \in [a,b]
        $ such that $ 
        s < r
        $
    we have
    \begin{equation*}
                \metric\del{
            \hat{\gamma}(r),
             \hat{\gamma}(s)
        }
        \le 
        \hat{v}(r) - \hat{v}(s).
    \end{equation*}
    Since $\hat{\gamma} = \gamma$ on $D$ (which is a set of full measure) and $\gamma$ ia a representative of $\wave{\gamma}$, we have that $\hat{\gamma}$ also is a representative of $\wave{\gamma}$.

    Next, let $\Delta = (t_i)_i^{m}$ be a partition of $[a,b]$, then
    \begin{equation*}
        V^\Delta(\hat{\gamma})
        =
        \sum_{i=1}^m
            \metric\del{
            \hat{\gamma}(t_i),
             \hat{\gamma}(t_{i-1})
        }
        \le 
        \sum_{i=1}^m
            \hat{v}(t_i)
            - \hat{v}(t_{i-1})
        =
        \hat{v}(t_m) - \hat{v}(t_0)
        =
        \hat{v}(b) - \hat{v}(a)
        \le 
        M + 2 \eps.
    \end{equation*}
    This shows that $V(\hat{\gamma}) \le M + 2\eps$ and thus $\essentialVariation\del{ \wave{\gamma}} \le M + 2\eps$.
    As $\eps > 0$ was arbitrary, 
    we have $\essentialVariation\del{ \wave{\gamma}} \le M$ as needed.
\end{proof}
\begin{corollary}
\label{cor:lower_semicontinuity_of_variation_function_on_test_curves}
    Let $(X, \metricAlone)$ be a metric space.
    Then 
    $
        V 
        \colon 
        \testCurves\del{ [a,b]; X}
        \to 
        \intco{ 0, \infty}
    $
    is lower semicontinuous. 
\end{corollary}
\begin{proof}
    In the first step we assume that $(X, d)$ is complete.
    Suppose that $\gamma_n \to \gamma $
    in $\del{ \testCurves\del{ [a,b]; X}, \metricAlone[\metricAlone_{TC}]}$.
    Let 
    $
        T
        \colon 
        \testCurves\del{ [a,b]; X}
        \to 
        \measurable\del{ [a,b]; X}
    $
    be the canonical immersion.
    Then $T(\gamma_n) \to T(\gamma)$ in 
    $\del{ \measurable\del{ [a,b]; X}, \metricAlone[\metricAlone_{\measurable}]}$.
    By Lemmata 
    \ref{lem:TC_representatives_have_minimal_variation}
    and \ref{lem:almost_lower_semicontinuity_of_essential_variation}
    we have  
    \begin{equation*}
        \liminf_{n\to \infty }
        V(\gamma_n)
        =
        \liminf_{n\to \infty }
        \essentialVariation\del{ T(\gamma_n)}
        \ge 
        \essentialVariation\del{ T(\gamma)}
        =
        V(\gamma)
    \end{equation*}
    which, as $(\gamma_n)_n$ was arbitrary, is equivalent
    to lower semicontinuity of $V$.

    Now, in the second step we we consider an arbitrary metric space $(X, d)$. Let $(\hat{X}, \hat{d})$ be a completion of $(X, d)$ and let $i : X \rightarrow \hat{X}$ be an isometric embedding. Then, $I: \testCurves\del{ [a,b]; X} \rightarrow \testCurves\del{ [a,b]; \hat{X}}$ defined as $I(\gamma)(t) =i(\gamma(t))$ is a continuous map, and by the first step 
    $V_{\hat{X}}\colon \testCurves\del{ [a,b]; \hat{X}}       \to  \intco{ 0, \infty}$ is lower semicontinuous. Therefore, $V= V_{\hat{X}} \circ I$ is lower semicontinuous.
\end{proof}
\begin{theorem}
\label{lem::hyper_envaluations_are_Borel}
    Let $(X, \metricAlone)$ be a metric space, then:
     \begin{enumerate}[label=(\alph*)]
        \item for $(s_1, \ldots, s_n) \in [a,b]^n$ and a Borel map $f \colon X^n \to \bRExtended$ the function
    \begin{equation*}
        \testCurves\del{ [a,b]; X}
        \ni 
        \gamma
        \mapsto 
        f\del{
            \gamma(s_1), \ldots, \gamma(s_n)
        }
    \end{equation*}
    is Borel,
    \item for every Borel and bounded from below or from above map 
     $f \colon X \to \bRExtended$ the function
    \begin{equation*}
                \testCurves\del{ [a,b]; X}
              \ni \gamma
        \mapsto 
        \int_\gamma f
    \end{equation*}
    is Borel,
    \item for every Borel and bounded from below or from above map 
     $f \colon X \to \bRExtended$ the function
    \begin{equation*}
              \testCurves\del{ [a,b]; X}
        \ni 
        \gamma
        \mapsto 
        \sint{\gamma}{ f}
    \end{equation*}
    is Borel.

    \end{enumerate}
    
\end{theorem}
\begin{proof}
First of all we shall prove the following lemma.

\begin{lemma}
\label{lem:subspaces_of_Borel_functions_closed_upon_pointwise_limits_with_continuous_are_families_of_Borel}
    Let $(X, \metricAlone)$ be a metric space and let $C$ be a linear subspace of Borel functions on $X$ with values in $\mathbb{R}$ such that
    \begin{itemize}
        \item 
        $C$ contains all (bounded) Lipschitz functions on $X$,
        \item 
        $C$ is closed upon
        pointwise limits of
        (equibounded)
        functions.
    \end{itemize}
    Then $C$ contains all (bounded and) Borel functions on $X$.
\end{lemma}
\begin{proof}
    The strategy of the proof is based on the proof of \cite[Theorem 11.6]{descriptive-set-theory}.
    
    First of all we shall show that $C$ contains characteristic functions of all open sets. It is clear $\indicator{ \emptyset}, \indicator{X} \in C$. Therefore, let us fix nonempty open set $U \subseteq X$ such $X \setminus U \neq \emptyset$. For such $U$ we define the following sets
    \begin{equation*}
        F_n
        \coloneqq 
        \set{
            x \in U
            \given 
            \dist\del{ x, X \setminus U} 
            \ge 1/n
        }.
    \end{equation*}
    Then $(F_n)_n$ is an increasing sequence of closed sets
    such that $U = \bigcup_{n=1}^\infty F_n$.
    Let $f_n \colon X \to \bR$
    be a sequence of Lipschitz and bounded functions given by 
    the formula
    \begin{equation*}
        \forall x \in X
        \qquad 
        f_n(x)
        \coloneqq 
        \min\del{
            1, 
            n\dist\del{ x, X \setminus U} 
        }.
    \end{equation*}
    Since $f_n \equiv 1$ on $F_n$, $f_n \equiv 0$ on $X \setminus U$
    and $f_n \colon X \to [0,1]$, then $f_n \to \indicator{ U }$
    pointwise. In this way we have proved that $\indicator{ U } \in C$.

    Next, let 
    $$
    \mathcal{A} \coloneqq \set{ B\in \borel\del{X} \given \indicator{ B} \in C}.
    $$
    We have shown that family $\mathcal{A}$ contains a
    $\pi$-system of open sets. 
    Let us notice that if $B \in \mathcal{A}$, 
    then 
    $
        \indicator{ X \setminus B}
        =
        1
        - \indicator{ B}
        \in C
    $
    as $1$ is Lipschitz (and bounded).
    Hence, $X \setminus B \in \mathcal{A}$.
    Moreover, if $(B_n)$ is a sequence of disjoint elements of $\mathcal{A}$, then
    $
        \indicator{ \bigcup_{n=1}^\infty B_n}
        =
        \lim_{n \to \infty} 
            \sum_{k=1}^n
            \indicator{ B_k }
    $,
    so 
    $\indicator{ \bigcup_{n=1}^\infty B_n}$ 
    is a pointwise limit of (equibounded by $1$)  sequence of elements of $C$. 
    Thus, $\indicator{ \bigcup_{n=1}^\infty B_n} \in C$ and $\bigcup_{n=1}^\infty B_n \in \mathcal{A}$.
    We have shown that $\mathcal{A}$
    is a $\lambda$-system containing a $\pi$-system
    of open sets.
    Hence, by the Dynkin Lemma $\mathcal{A}$ contains all Borel sets.

    Now, let $f \colon X \to \bR$ be bounded Borel map. 
    There exists $M \in \bN$ such that $\abs{f} \le M$ everywhere.
    For $n \in \bN$ and $i \in \set{0, \ldots, n}$, let
    \begin{equation*}
        A^n_i
        \coloneqq 
        f^{-1} \sbr{ 
            \intco{
                M\del{ -1 + 2i/n },
                M\del{ -1 + 2(i+1)/n }
            } 
        }
    \end{equation*}
    and 
    $
        f_n 
        \coloneqq 
        \sum_{i=0}^{n}
            M\del{ -1 + 2i/n }
            \indicator{ A^n_i}
    $.
    Then $(f_n)_n$ is a sequence
    of functions (equibounded by $M$),
    which are linear combinations of indicators of Borel sets, so $f_n \in C$ for all $n \in\bN$.
    Moreover, we have $\abs{ f_n - f }\le 2/n$ everywhere.
    Therefore, $f_n \to f$ everywhere
    and     we have $f \in C$. 
    Thus, we have proved that $C$ contains 
    all bounded Borel functions $f \colon X \to \bR$.

    In the case when $C$ is closed
    upon pointwise limits of not necessarily equibounded sequences of elements of $C$,
    then, if $f \colon X \to \bR$ is Borel,
    we have that 
    $
        f_n 
        \coloneqq 
        \max\del{ -n, \min\del{ n, f}}
    $
    is a sequence of elements of $C$ such that $f_n \to f$ pointwise.
    Hence, $C$ contains all Borel functions $f \colon X \to \bR$. 
 \end{proof}
   $(a)$  Since the map $
        \testCurves\del{ [a,b]; X}
        \ni 
        \gamma
        \mapsto 
        f\del{
            \gamma(s_1), \ldots, \gamma(s_n)
        }$ 
        is a pointwise limit of maps $
        \testCurves\del{ [a,b]; X}
        \ni 
        \gamma
        \mapsto 
        f_k\del{
            \gamma(s_1), \ldots, \gamma(s_n)
        }$, where $f_k=\max(-k, \min(k,f))$, we can assume that 
        $f \colon X^n \to \bR$. Let 
    \begin{equation*}
        \mathsf{Test}^n\del{X}
        \coloneqq 
        \set{
            f \colon [a,b]^n \times X^n \to \bR 
            \given 
            \abs{f} \le 1
            \text{ and }
            \mathrm{Lip}(f) \le 1
        }.
    \end{equation*}
    Then, for all 
    $f \in \mathsf{Test}^n(X)$, the map
    \begin{equation*}
        \mathcal{M}\del{ [a,b]; X }
        \ni 
        \gamma 
        \mapsto 
        \integral{ [a,b]^n }
        {
            f\del{ 
                s_1, \ldots, s_n, \gamma(s_1), \ldots, \gamma(s_n)
            }
        }{s}
    \end{equation*}
    is Lipschitz, hence Borel.
   Indeed, for all 
    $\gamma, \gamma' \in \mathcal{M}\del{ [a,b]; X }$
     we have\footnote{The measurability of the map $( s_1, \ldots, s_n)\mapsto f\del{ s_1, \ldots, s_n, \gamma(s_1), \ldots, \gamma(s_n) 
            }$ follows from Proposition \ref{mierzalnosc}.}
    \begin{align*}
        &
        \abs{
            \integral{ [a,b]^n }
            {
                f\del{ s_1, \ldots, s_n, \gamma(s_1), \ldots, \gamma(s_n) }
            }{s}
            -
            \integral{ [a,b]^n }
            {
                f\del{ s_1, \ldots, s_n, \gamma'(s_1), \ldots, \gamma'(s_n) }
            }{s}
        }
        \\
        &\qquad\le 
        \integral{ [a,b]^n }{
            \abs{
                f\del{ s_1, \ldots, s_n, \gamma(s_1), \ldots, \gamma(s_n) }
                - f\del{ s_1, \ldots, s_n, \gamma'(s), \ldots, \gamma'(s_n) }
            }
        }{s}
        \\
        &\qquad=
         \integral{ [a,b]^n }{
            \min\del{
                2,
                \abs{
                    f\del{ s_1, \ldots, s_n, \gamma(s_1), \ldots, \gamma(s_n)}
                    - f\del{ s_1, \ldots, s_n, \gamma'(s_1), \ldots, \gamma'(s_n) }
                }
            }
        }{s}
        \\
        &\qquad\le 
         \integral{ [a,b]^n }{
            \min\del{
                2,
                \sum_{i=1}^n
                \metric\del{
                     \gamma(s_i), 
                     \gamma'(s_i) 
                }
            }
        }{s}
        \\
        &\qquad\le 
        \sum_{i=1}^n
        \integral{ [a,b] }
        {
        \min \del{2, \metric\del{
                     \gamma(s_i), 
                     \gamma'(s_i) 
                }
                }
        }{s_i}
        \\
        &\qquad\le 
        2n \metric_\mathcal{M} \del{
            \gamma , \gamma'
        }.
    \end{align*}

    Let $l_1, \ldots, l_n, r_1, \ldots, r_n \in [a,b]$ satisfy $l_i < r_i$ for all $i \in [n]$. 
    Fix $f \in \mathsf{Test}^n(X)$
    and define 
    $
        f_m \colon [a,b]^n \times X^n \to \bR 
    $
    by $f_m(s,x) \coloneqq f(s,x) g_m(s)$,
    where $g_m \colon [a,b]^n \to [0,1]$ is a smooth function such that  
        $
            g_m \equiv 1
        $
        on
        $
            \bigtimes_{i=1}^n
            [l_i, r_i]
        $ 
        and  
        $
            g_m \equiv 0
        $
        on $[a,b]^n \setminus \bigtimes_{i=1}^n
            [l_i -1/m, r_i +1/m] $
    Since each of $f_m$ is Lipschitz and bounded by 1, 
 there is $c_m > 0$ such that
    $c_m f_m\in \mathsf{Test}^n(X)$.
    In consequence,
    \begin{equation*}
        \mathcal{M}\del{ [a,b]; X }
        \ni 
        \gamma 
        \mapsto 
        \integral{ [a,b]^n }
        {
            f_m\del{ s_1, \ldots, s_n, \gamma(s_1), \ldots, \gamma(s_n) }
        }{s}
    \end{equation*}
    are continuous, hence Borel.
    
    Since $f_m$ are bounded by $1$ and the pointwise limit of $f_m$ is a function 
    $
        (s, x)
        \mapsto
        f(s, x)
        \indicator{ 
            \bigtimes_{i=1}^n [l_i, r_i]
        }(s),
    $
    by the Lebesgue dominated convergence theorem we get that 
    \begin{multline*}
        \mathcal{M}\del{ [a,b]; X }
        \ni 
        \gamma 
        \mapsto 
        \integral{ \bigtimes_{i=1}^n [l_i, r_i] }
        {
            f\del{ s, \gamma(s) }
        }{s}
        =
        \integral{ [a,b]^n }
        {
            f\del{ s_1, \ldots, s_n, 
                \gamma(s_1), \ldots, \gamma(s_n) }
            \indicator{ \bigtimes_{i=1}^n [l_i, r_i] }(s)
        }{s}
        \\
        =
        \lim_{n \to \infty}
        \integral{ [a,b]^n }
        {
            f_m\del{ s_1, \ldots, s_n,
                \gamma(s_1), \ldots, \gamma(s_n) }
        }{s}
    \end{multline*}
    is Borel.
    Furthermore, since embedding $\testCurves\del{ [a,b]; X} \hookrightarrow
    \mathcal{M}\del{ [a,b]; X}$ is continuous, we get that  
    \begin{equation*}
        \mathrm{TC}\del{ [a,b]; X }
        \ni 
        \gamma 
        \mapsto 
        \integral{ \bigtimes_{i=1}^n [l_i, r_i] }
        {
            f\del{ s, \gamma(s) }
        }{s}
    \end{equation*}
    is Borel. 
    Fix $l_1, \ldots l_n \in \intco{a,b}$ 
    and let $r_{i,m} \in \intoc{l_i, 1}$ be such that $r_{i, m} \to l_i^+$. 
    Then for a fixed 
    $\gamma \in \mathsf{TC}\del{ [a,b]; X }$
    we have
    \begin{align*}
        &\abs{
            f\del{l_1, \ldots, l_n, 
                \gamma(l_1), \ldots, \gamma(l_n) }
            -
            \fint_{ \bigtimes_{i=1}^n [l_i, r_{i,m}] }
                f\del{ s_1, \ldots s_n, 
                    \gamma(s_1), \ldots, \gamma(s_n) }
            \ \mathrm{d}s
        }
        \\
        &\qquad \le 
        \fint_{ \bigtimes_{i=1}^n [l_i, r_{i,m}]  }
            \abs{
                f\del{l_1, \ldots, l_n, 
                \gamma(l_1), \ldots, \gamma(l_n) }
                -
                 f\del{s_1, \ldots, s_n, 
                \gamma(s_1), \ldots, \gamma(s_n) }
            }
        \ \mathrm{d}s
        \\
        &\qquad \le 
        \fint_{\bigtimes_{i=1}^n [l_i, r_{i,m}]  }
            \sum_{i=1}^n 
                (s_i - l_i)
                + \metric\del{ 
                    \gamma(s_i), \gamma(l_i)
                }
        \ \mathrm{d}s \xrightarrow{m \to \infty} 0,
    \end{align*}
    where we have used the right continuity of $\gamma$. 
    We conclude that
    \begin{equation*}
        \testCurves\del{ [a,b]; X}
        \ni 
        \gamma 
        \mapsto 
        f\del{ r_1, \ldots, r_n , \gamma(r_1), \ldots, \gamma(r_n) },
    \end{equation*}
    where $r_1, \ldots, r_n \in \intco{a,b}$,
    is Borel as a pointwise limit of Borel functions.
    Since elements of 
    $
        \testCurves\del{ [a,b]; X}
    $
    are left-continuous at $b$, for all tuples
    $(r_i)_{i=1}^n$ such that $r_i \in [a,b]$ 
    we 
    can select $l_{i, m} \in \intco{a,b}$
    such that $l_{i, m} \to r_i^-$
    if $r_i = b$ and $l_{i, m} = r_i$ otherwise.
    Then
    \begin{equation*}
        f(l_{1, m}, \ldots, l_{n, m}, 
    \gamma(l_{1, m}), \ldots, \gamma(l_{n, m}))
    \to 
        f( r_1 , \ldots,  r_n, 
    \gamma( r_1 ), \ldots, \gamma( r_n))
    \end{equation*}
    for all 
    $\gamma
     \in \mathsf{TC}\del{ [a,b]; X}
     $
     and
    \begin{equation*}
        \testCurves\del{ [a,b]; X}
        \ni 
        \gamma 
        \mapsto 
        f\del{ r_1 , \ldots,  r_n, 
    \gamma( r_1 ), \ldots, \gamma( r_n) }
    \end{equation*}
    is Borel.
    This shows that
    \begin{equation*}
        \testCurves\del{ [a,b]; X}
        \ni 
        \gamma 
        \mapsto 
        f\del{ r_1 , \ldots,  r_n, 
    \gamma( r_1 ), \ldots, \gamma( r_n) }
    \end{equation*}
    are Borel for 
    all 
    $f \in \mathsf{Test}^n(X)$
    and 
    $r_1, \ldots, r_n \in [a,b]$.
    By scaling $f$ by a constant we get the same 
    result for any Lipschitz and bounded $f$.
    Since 
    $ c \mapsto \max( -m, \min( m, c))$
    are $1$-Lipschitz, by composing $f$ with the above functions,
    we can obtain the same result for 
    any Lipschitz function.

    Finally, for each $s = (s_1, \ldots, s_n) \in [a,b]^n$ we define
    \begin{equation*}
        C_s
        \coloneqq 
        \set{
            f \colon X^n \to \bR 
            \text{ Borel}
            \given 
            \testCurves\del{ [a,b]; X}
            \ni 
            \gamma 
            \mapsto  
            f\del{
                    \gamma(s_1),
                    \ldots 
                    \gamma(s_n)
                }
            \text{ is Borel}
        }.
    \end{equation*}
Since every Lipschitz 
    $
        f \colon X^n \to \bR 
    $
    can be treated as a Lipschitz 
    function
    $
        \wave{f} \colon [a,b]^n \times X^n.
    $       
    by 
    $(s,x) \mapsto \wave{f}(s, x) = f(x)$, we get that $C_s$ contains all Lipschitz functions $f \colon X^n \to \bR$. Therefore, since $X^n$ is a metric space and $C_s$ is a subspace of Borel functions, by Lemma     \ref{lem:subspaces_of_Borel_functions_closed_upon_pointwise_limits_with_continuous_are_families_of_Borel}
    we have that $C_s$ contains all Borel functions.

$(b)$
    First of all we shall prove the following lemma.
    \begin{lemma}
\label{lem:Lebesgue-Stieltjes_integral_as_a_limit_of_Riemann_sums}
    Let $(X, \metricAlone)$ be a metric space.
    Let\footnote{By $\continuousAndBounded(X)$ we denote the set of bounded and continuous functions on $X$.} $f \in \continuousAndBounded(X)$.
    Let 
    $
        \del{t^n_i}_{i=1}^{m_n} = \del{ \Delta_n }_n
    $ 
    be a normal sequence of partitions of $[a,b]$.
    Then for every 
    $
        \gamma \in \boundedVariationRC\del{
            [a,b];
            X
        }
    $
    we have 
    \begin{equation*}
        \int_\gamma f
        =
        \lim_{n \to \infty}
        \sum_{i=1}^{m_n}
            \metric\del{
                \gamma( t^n_i),
                \gamma( t^n_{i-1} )
            }
            f \circ \gamma (t^n_i).
    \end{equation*}
\end{lemma}
\begin{proof}
    For $n \in \bN$ let
    \begin{equation*}
        A_i 
        \coloneqq 
        \intoc{
            t^n_{i-1}, t^n_i
        },
        \text{ for }
        i \in [m_n],
        \text{ and }
        A_0 \coloneqq \set{a},
    \end{equation*}
    then for every $n \in \bN$ and every $i \in [m_n]$ we have
    \begin{equation*}
        \mu_{\gamma}\del{A_i}
        =
        \mu_{\gamma}\del{
            \intoc{ t^n_{i-1}, t^n_i }
        }
        =
        V_\gamma( t^n_i)
        - V_\gamma( t^n_{i-1} )
        <
        \infty.
    \end{equation*}
    Next, we define the sequence of simple functions 
    $
        g_n \colon [a, b] \to \bR
    $ 
    as follows
    \begin{equation*}
        g_n
        \coloneqq 
        \sum_{i=0}^{m_n}
            f \circ \gamma \del{ t^n_i}
            \indicator{ A_i }.
    \end{equation*}
    Since $f \in \continuousAndBounded(X)$, 
    there exists $M > 0$ such that $\abs{ f} \le M$
    everywhere in $X$. Therefore, $\abs{ g_n } \le M$ everywhere.

    Furthermore,  $g_n \to f \circ \gamma $ pointwise in $[a,b]$. Indeed, let us note that for all $n \in \bN$ and all $i \in \set{0, \ldots, m_n}$ we have that 
    $
        g_n( t^n_i ) = f \circ \gamma( t^n_i )
    $. In particular,  $g_n(a) = f \circ \gamma(a)$ and $g_n(b) = f \circ \gamma(b)$.
    
    Let $t \in (a,b)$, then $t$ is a point of right-continuity of $\gamma$ and, since $f$ is continuous, it is also a point of right-continuity of $f \circ \gamma$. 
    For each $n \in \bN$, 
    let 
    $
        t_n^\le 
        \coloneqq 
        \min 
        \set{
            s \in \Delta_n 
            \given 
            t \le s
        }.
    $ 
    Since $\del{ \Delta_n }_n$ is a normal sequence of partitions,   $t_n^\le \to t^+$ (where we allow that $t^\le_n = t$).
    Thus,
    \begin{equation*}
        g_n(t)
        =
        g_n( t_n^\le)
        =
        f \circ \gamma( t_n^\le)
        \xrightarrow{n \to \infty}
        f \circ \gamma(t).
    \end{equation*}
    Now, since $M \in L^1( \mu_{\gamma})$, by the Lebesgue Dominated Convergence Theorem we have
    \begin{equation*}
        \sum_{i=1}^{m_n}
                \mu_{\gamma}\del{
                    \intoc{ t^n_{i-1},
                        t^n_i
                    }
                }
                f \circ \gamma (t^n_i)
        =
        \integral{ [a,b] }{ g_n }{ \mu_{\gamma} }
        \xrightarrow{ n \to \infty }
        \integral{ [a,b] }{ f \circ \gamma }{ \mu_{\gamma}}
        =
        \int_\gamma f.
    \end{equation*}

    Finally, let us notice
    \begin{align*}
        \abs{
            \integral{ [a,b] }{ g_n }{ \mu_{\gamma} }
            - \sum_{i=1}^{m_n}
                \metric\del{
                    \gamma( t^n_i),
                    \gamma( t^n_{i-1} )
                }
                f \circ \gamma (t^n_i)
        }
        &=
        \abs{
            \sum_{i=1}^{m_n}
                \del{
                    \mu_{\gamma}\del{
                        \intoc{ t^n_{i-1},
                            t^n_i
                        }
                    }
                    -
                     \metric\del{
                        \gamma( t^n_i),
                        \gamma( t^n_{i-1} )
                    }
                }
                f \circ \gamma (t^n_i)
        }
        \\
        &\le 
        M
        \sum_{i=1}^{m_n}
            \abs{
                    \mu_{\gamma}\del{
                        \intoc{ t^n_{i-1},
                            t^n_i
                        }
                    }
                    -
                     \metric\del{
                        \gamma( t^n_i),
                        \gamma( t^n_{i-1} )
                    }
                }
        \\ 
        &=
        M
        \sum_{i=1}^{m_n}
            V_\gamma( t^n_i )
            - V_\gamma( t^n_{i-1} )
            - \metric\del{
                        \gamma( t^n_i),
                        \gamma( t^n_{i-1} )
                    }
        \\
        &=
        M \del{ 
            V_\gamma( t^n_{m_n}) - V_\gamma( t^n_{0})
            - V^{\Delta_n}(\gamma)
        }
        \\
        &=
        M\del{ V(\gamma) - V^{\Delta_n}(\gamma)
        }
        \xrightarrow{ n \to \infty}
        0.
    \end{align*}
     We conclude that
    \begin{equation*}
        \int_\gamma f
        =
        \lim_{n \to \infty}
            \integral{ [a,b] }{ g_n }{ \mu_{\gamma} }
        =
        \lim_{n \to \infty}
            \sum_{i=1}^{m_n}
                \metric\del{
                    \gamma( t^n_i),
                    \gamma( t^n_{i-1} )
                }
                f \circ \gamma (t^n_i)
    \end{equation*}
    as needed.
\end{proof}
 Let us define the following family
    \begin{equation*}
        C
        \coloneqq 
        \set{
            f \colon X \to \bR 
            \given
            f \text{ is Borel and bounded and }
            \quad 
            \testCurves\del{ [a,b]; X}
            \ni \gamma 
            \mapsto 
            \int_\gamma f 
            \quad 
            \text{ is Borel}
        }.
    \end{equation*}
    Let us observe that 
    $
        \continuousAndBounded\del{[a,b];\bR}
        \subseteq 
        C
    $.
    Indeed, let 
    $
        f \in \continuousAndBounded\del{ X; \bR}
    $ and $\del{t^n_i}_{i=0}^{m_n} = (\Delta_n)_n$ be a normal sequence of partitions. Then, from $(a)$ 
    we have that functions
    \begin{equation*}
            \testCurves\del{ [a,b]; X}
        \ni 
        \gamma
        \mapsto 
        \sum_{i=1}^{m_n}
            f\del{ \gamma \del{ t^n_i}}
            \metric\del{
                \gamma \del{ t^n_i},
                \gamma \del{ t^n_{i-1}}
            }
    \end{equation*}
    are Borel. 
    By Lemma \ref{lem:Lebesgue-Stieltjes_integral_as_a_limit_of_Riemann_sums}
    for all $\gamma \in \testCurves\del{ [a,b]; X}$ we have
    \begin{equation*}
        \int_\gamma f 
        =
        \lim_{n \to \infty} 
        \sum_{i=1}^{m_n}
            f\del{ \gamma \del{ t^n_i}}
            \metric\del{
                \gamma \del{ t^n_i},
                \gamma \del{ t^n_{i-1}}
            }.
    \end{equation*}
        Hence, $\gamma \mapsto \int_\gamma f$ is Borel as a pointwise limit
    of a sequence of Borel functions.

    Now, we will show that $C$ is closed upon pointwise limits
    of equibounded sequences.
    Let $(f_n)_n$ be a sequence of elements of $C$ such that
    there exists $M \ge 0$ such that $f_n \to f$ pointwise and
    for all $n \in \bN$ we have
    $\abs{ f_n } \le M$.
    Let $\gamma \in \testCurves\del{ [a,b]; X}$, then by the Lebesgue Dominated Convergence Theorem, we have
    \begin{equation*}
        \int_\gamma f_n
        =
        \integral{ [a,b] }{ f_n \circ \gamma }{ \mu_\gamma }
        \xrightarrow{ n \to \infty  }
        \integral{ [a,b] }{ f \circ \gamma }{ \mu_\gamma }
        =
        \int_\gamma f.
    \end{equation*}
    Therefore, the map 
    \begin{equation*}
        \testCurves\del{ [a,b]; X}
        \ni 
            \gamma 
            \mapsto 
            \int_\gamma f
    \end{equation*}
     is a pointwise limit of a sequence of Borel functions, hence it is Borel.
     Thus, $f \in C$.

     We have shown that $\continuousAndBounded\del{ X; \bR} \subseteq C$
     and that $C$ is bounded upon pointwise limits of 
     equibounded sequences. 
     Hence, by Lemma \ref{lem:subspaces_of_Borel_functions_closed_upon_pointwise_limits_with_continuous_are_families_of_Borel}
     family $C$ contains all Borel and bounded functions.

     Finally, if $f \colon X \to \bRExtended$ is Borel and bounded from
     above or from below, then sequence $(f_n)_n$
     defined by $f_n \coloneqq \max\del{ -n, \min\del{ n, f}}$
     is a sequence of bounded Borel functions such that $f_n \to f$
     pointwise.
     Moreover, since $f$ is bounded from above or from below, 
     there exists $N \in \bN$ such that $f \le N$ or $-N \le f$
     and subsequence $\del{ f_n }_{n \ge N}$ is monotonic.
     Therefore, by the Lebesgue Monotone Convergence Theorem,
     we have 
     \begin{equation*}
        \int_\gamma f_n
        =
        \integral{ [a,b] }{ f_n \circ \gamma }{ \mu_\gamma }
        \xrightarrow{ n \to \infty  }
        \integral{ [a,b] }{ f \circ \gamma }{ \mu_\gamma }
        =
        \int_\gamma f.
    \end{equation*}
    Thus, function
    \begin{equation*}
        \testCurves\del{ [a,b]; X}
        \ni 
            \gamma 
            \mapsto 
            \int_\gamma f
    \end{equation*}
     is a pointwise limit of a sequence of Borel functions, hence it is Borel. This proves the claim.

$(c)$
    Since $\sint{\gamma}{f} = \frac{1}{2} \del{ \int_\gamma f + \int_{\overleftarrow{\gamma}} f }$ and having in mind $(b)$, it is sufficient to show
    $
        \overleftarrow{\cdot} 
        \colon 
        \testCurves\del{ [a,b]; X}
        \to 
        \testCurves\del{ [a,b]; X}
    $
    is  Borel.
    We will prove that, in fact, this function is a $\metricAlone_{TC}$-isometry. Every element of $\testCurves\del{ [a,b]; X}$ is continuous on a set of full measure. Therefore, for $\gamma, \gamma' \in \testCurves\del{ [a,b]; X}$ we have that
        \begin{equation*}
        \overleftarrow{\gamma}(t)
        =
        \gamma 
            \del{
                 a + b - t             
            },
            \overleftarrow{\gamma'}(t)
        =
        \gamma'
            \del{
                 a + b - t             
            }
        \text{ for almost all }
        t \in [a,b].
    \end{equation*}
    Thus, 
    \begin{align*}
        \metric_{\,TC}\del{
            \overleftarrow{\gamma}, \overleftarrow{\gamma'}
        }
        &=
        \integral{ [a,b] }{
            \min \del{
                1,
                \metric\del{
                    \overleftarrow{\gamma}(t), 
                    \overleftarrow{\gamma'}(t)
                }
            }
        }{t}
                \\
        &=
        \integral{ [a,b] }{
            \min \del{
                1,
                \metric\del{
                    \gamma( s), 
                    \gamma'(s )
                }
            }
        }{s}
        \\ 
        &=
        \metric_{\,TC}\del{
            \gamma, 
            \gamma '
        },
    \end{align*}
    so 
    $
        \metric_{\,TC}\del{
            \overleftarrow{\gamma}, \overleftarrow{\gamma'}
        }
        =
        \metric_{\,TC}\del{
            \gamma, 
            \gamma '
        }
    $.
\end{proof}
\begin{proposition} \label{ewaluacje}
    Let $(X, d)$ be a metric space, then for every $t \in [0,1]$ the evaluation map $e_t :  \testCurves\del{[0,1];X} \rightarrow X$, $e_t (\gamma):=\gamma(t)$ is Borel.
\end{proposition}
\begin{proof}
Let $\phi \colon X \to \continuousAndBounded(X)$ be the Kuratowski isometric embedding. Let us note that if $\gamma \in \testCurves\del{[0,1];X}$,
then $ \phi \circ \gamma \in \testCurves\del{[0,1]; \continuousAndBounded(X) }$. For a fixed $N \in \mathbb{N}$ we will consider the space
$$
    Y_N = 
    \set{
        \gamma \in \testCurves\del{ [0,1]; X}
        \given
        V( \gamma ) \le N
    }.
$$
By Corollary \ref{cor:lower_semicontinuity_of_variation_function_on_test_curves} the map $V$ is lower semicontinuous, thus each of $Y_N$ is a Borel subset of $TC([0,1]; X)$. 

Next, let us fix $s \in [0,1)$, and let 
$I = [s, s+\delta]$ for some $\delta > 0$ such that $s + \delta \le 1$.
We define the following map 
\begin{equation*}
    Y_N
    \ni
    \gamma 
    \mapsto 
    f_I( \gamma) 
    \coloneqq 
    \integral{I}{ \phi \circ \gamma(t) }{t},
\end{equation*}
where the integral is understood in the Bochner sense.
First, let us note that for every $\gamma \in \testCurves\del{[0,1]; X}$ integral
$
    \integral{I}{ \phi \circ \gamma(t) }{t}
$
is well-defined in the Bochner sense.
Indeed, the image of $ \phi \circ\gamma $
is separable, since it is totally bounded. Moreover, for every $\Lambda \in (\continuousAndBounded(X))^*$ we have 
$\Lambda \circ \phi \circ\gamma
 \in \testCurves\del{ [0,1]; \mathbb{R} }$,
 so $\Lambda \circ \phi \circ\gamma$ is Lebesgue measurable.
Thus, $\phi \circ\gamma$ has a separable image and is weakly measurable. Therefore by the Pettis Theorem,  $\phi \circ\gamma$ is Bochner measurable. 
Furthermore, since the image of 
$\phi \circ\gamma$ is totally bounded, it is bounded, hence the image of 
$\| \phi \circ\gamma \|_{\continuousAndBounded(X)}$
is bounded. Thus, 
$\| \phi \circ\gamma \|_{\continuousAndBounded(X)}$
is Lebesgue integrable, hence indeed
$ \phi \circ\gamma $ is Bochner integrable.

Next, we will show that 
$
    f_I \colon Y_N \to \continuousAndBounded(X)
$
is continuous. For this purpose we fix $\gamma \in Y_N$ and $\eps \in \intoo{ 0, \abs{ I} }$, where $\abs{ I}$ denotes the  measure of $I$. Therefore, for every 
$ \gamma' \in Y_N$ such that
$
    \metricAlone_{TC}( \gamma, \gamma' ) \le \eps
$
we see that the measure of the set
$
    \set{
        t \in [0,1]
        \given
        \metricAlone( \gamma(t), \gamma'(t) )
        \ge 1
    }
$
cannot be greater than $\eps$.
Since $\abs{I} > \eps,$ there exists
$\tau \in I$ such that
$
    \metricAlone( \gamma(\tau), \gamma'(\tau) ) < 1.
$
Now, if $t \in I$ is such that $t < \tau$, then $\Delta = (0,t, \tau, 1)$ is a partition of $[0,1]$, 
hence for all $\widetilde{\gamma} \in Y_N$ we have
\begin{equation*}
    \metricAlone( 
        \widetilde{\gamma}(t),
        \widetilde{\gamma}( \tau )
    )
    \le 
    V^{ \Delta }( \widetilde{\gamma} )
    \le 
    V(\wave{\gamma})
    \le 
    N.
\end{equation*}
In consequence, for all $t \in I$ such that $t < \tau$ we have
\begin{equation*}
    \metricAlone( \gamma(t), \gamma'(t) )
    \le 
    \metricAlone( \gamma(t), \gamma( \tau) )
    + \metricAlone( \gamma( \tau), \gamma'( \tau) )
    + \metricAlone( \gamma'( \tau), \gamma'(t) )
    \le 
    2N+1.
\end{equation*}
In a similar way we prove that 
$
    \metricAlone( \gamma(t), \gamma'(t) ) \le 2N+1
$
for $t \in I$ with $t > \tau$.
Thus, for all $t \in I$ we have
 $$
    \metricAlone( \gamma(t), \gamma'(t) ) \le 2N+1.
$$
Therefore, 
\begin{align*}
    \norm{
        f_I( \gamma) -
        f_I( \gamma')
    }_{ \continuousAndBounded(X) }
    &=
    \norm{
        \integral{I}{ 
        \phi \circ \gamma(t)
        }{t}
        -
        \integral{I}{ 
        \phi \circ \gamma'(t)
        }{t}
    }_{ \continuousAndBounded(X)}
    \\
    &\le 
    \integral{I}{ 
        \norm{
            \phi \circ \gamma(t)
            - \phi \circ \gamma'(t)
        }_{ \continuousAndBounded(X)}
    }{t}
    \\
    &=
    \integral{I}{ 
        \metricAlone( \gamma(t) , \gamma'(t) )
    }{t}
    \\
    &= 
    \integral{I}{ 
        \min\del{ 
            2N+1, 
            \metricAlone( \gamma(t) , \gamma'(t) )
        }
    }{t}
    \\
    &\le 
    (2N+1)
    \integral{I}{ 
        \min\del{ 
            1, 
            \metricAlone( \gamma(t) , \gamma'(t) )
        }
    }{t}
    \\
    &\le
    (2N+1) \metricAlone_{TC}( \gamma, \gamma' ).
\end{align*}
Hence, we conclude that 
$f_I$ is continuous, and consequently Borel.

Since each element of
$
    \testCurves\del{ [0,1]; \continuousAndBounded(X) }
$
is right-continuous on $[0,1)$, 
for a sequence of intervals 
$I_n = [s, s+1/n]$ and every 
$
    \gamma \in Y_N \subseteq \testCurves\del{ [0,1]; X}
$
we have
\begin{equation*}
    \dfrac{ 1 }{ \abs{ I_N } }
    f_{ I_n} (\gamma) 
    =
    \fint_{I_n} \phi \circ\gamma(t) \ \mathrm{d}t
    \xrightarrow{ n \to \infty }
    \phi \circ\gamma(s).
\end{equation*}
Thus, for every $s \in [0,1)$ the function
$
    \gamma 
    \mapsto 
    \phi \circ\gamma(s)
$
is a Borel map from
$ Y_N$
to
$\continuousAndBounded(X)$. This is also true for $s = 1$, since we can choose a sequence $(s_n)_n$ of elements of $[0,1)$ that converges to $1$ and we have
$
    \phi \circ\gamma(s_n)
    \to 
    \phi \circ\gamma(1)
$
pointwise.

Next, we claim that for every $s\in[0,1]$ function $\gamma \mapsto \gamma(s)$ is a Borel map from $Y_N$ to $X$. Indeed, since the inverse function $\phi^{-1} : \phi [X] \rightarrow X$ is continuous, we get $\gamma \mapsto \gamma(s)=\phi^{-1} \circ \phi \circ \gamma (s)$ is a Borel map from $Y_N$ to $X$.

Finally, we will show that for all $s \in [0,1]$ the function
$
    e_s
$
is Borel as a map from the entire $\testCurves\del{ [0,1]; X }$ to $X$. 
Indeed, suppose that $B$ is Borel in $X$. 
Then 
$
    \set{
        \gamma \in Y_N
        \given
         \gamma(s) \in B
    }
$
is Borel in $Y_N$ for all $N \in \mathbb{N}$. 
However, since sets $Y_N$ are Borel in 
$\testCurves\del{ [0,1]; X  }$, these sets are also Borel in 
$\testCurves\del{ [0,1]; X  }$.
In consequence, set
\begin{equation*}
    \set{
        \gamma \in \testCurves\del{ [0,1]; X }
        \given 
         \gamma(s) \in B
    }
    =
    \bigcup_{ N=1 }^{\infty}
    \set{
        \gamma \in Y_N
        \given 
         \gamma(s) \in B
    }
\end{equation*}
is Borel and the proof follows.   
\end{proof}
\section{Function spaces}\label{sec::function_spaces}

\subsection{$\testCurves$-Newtonian Space}\label{subsec::TC-Newtonian}

Within this subsection we introduce the $\testCurves$-Newtonian spaces which are a modification of the Newtonian spaces introduced by Shanmugalingam \cite{newtonian}. 
To perform this modification,
we replace the family of rectifiable curves with the family $\testCurves\del{[0,1];X}$ and the integral along the curve $\int_\gamma$ with the symmetrized integral $\sint{\gamma}$,
to define the $\testCurves$-Newtonian space $\newtonianTC\del{X}$.
The theory of the $p$-modulus can be recreated with almost no changes.
However, the results for the $p$-weak upper gradients require greater care due to the possible discontinuity of the considered curves.

\begin{definition}[p-modulus of the family of test curves]
\label{def:p-modulus_of_the_family_of_test_curves}
    Let $(X, \metricAlone)$ be a metric space and let $\measure$ be a Borel measure on $X$. We denote the family of \emph{non-trivial} test curves (i.e. ones with non-zero variation), by
    \begin{equation*}
        \testCurvesNonTrivial
        \coloneqq 
        \set{
            \gamma \in \testCurves\del{ [0,1]; X }
            \given
            V(\gamma) > 0
        }.
    \end{equation*}
Note that by Corollary \ref{cor:lower_semicontinuity_of_variation_function_on_test_curves}
    we have that 
    $
        V \colon 
        \testCurves\del{ [a,b]; X}
        \to \intco{0, \infty}
    $
    is lower semicontinuous, hence Borel. 
    Thus, $\testCurvesNonTrivial$ is a Borel subset
    of $\testCurves\del{ [a,b]; X}$.
    For $\Gamma \subseteq \testCurvesNonTrivial$ we define
    \begin{equation*}
        F(\Gamma)
        \coloneqq 
        \set{
            \rho \colon X \to \intcc{ 0, \infty }
            \given 
            \rho \text{ is Borel and }
            \sint{ \gamma} { \rho } 
            \ge 1
            \text{ for all }
            \gamma \in \Gamma
        }.
    \end{equation*}
    Then, for $p \in \intco{ 0, \infty }$ we define the $p$-modulus of families of test curves
    $   
        \modulus^p 
        \colon 
        2^{ \testCurvesNonTrivial }
        \to 
        \intcc{0, \infty}
    $
    by the formula
    \begin{equation*}
        \forall 
        \Gamma \subseteq \testCurvesNonTrivial
        \qquad 
        \modulus^p( \Gamma)
        \coloneqq 
        \inf_{ \rho \in F(\Gamma) }
        \norm{ \rho }^p_{L^p(\measure)}.
    \end{equation*}

    We shall say that property $P$ holds for $\modulus^p$ almost every $\gamma \in \testCurvesNonTrivial$, if the set of $\gamma \in \testCurvesNonTrivial$ for which $P$ does not hold has a $p$-modulus of $0$.
\end{definition}
\begin{proposition} \label{mod0}
    $\modulus^p$ is an outer measure on $\testCurvesNonTrivial$. Moreover, for $\Gamma \subseteq \testCurvesNonTrivial$ we have $\modulus^p( \Gamma ) = 0$ if and only if there exists 
    $\rho \in F( \Gamma)$ such that $\rho \in L^p( \measure )$ and $\sint{ \gamma }{ \rho } = \infty 
            \text{ for all }
            \gamma \in \Gamma.$
\end{proposition}
\begin{proof}
    First of all we shall prove that $\modulus^p$ is an outer measure. Let us note that $\modulus^p(\emptyset) = 0$ as $0 \in F( \emptyset)$. If $\Gamma_1 \subseteq \Gamma_2 \subseteq \testCurvesNonTrivial$, then $F(\Gamma_2) \subseteq F(\Gamma_1)$. Hence, $\modulus^p(\Gamma_1) \le \modulus^p(\Gamma_2)$.  For $n \in \bN$ let $\Gamma_n \subseteq \testCurvesNonTrivial$. Then, for a fixed $\eps > 0$ and for all $n \in \bN$ there is $\rho_n \in F( \Gamma_n)$ such that 
    $
        \norm{ \rho_n }^p_{L^p(\measure) }
        \le 
        \modulus^p( \Gamma_n ) + \eps 2^{-n}
    $.
    Let us define 
    $
        \rho 
        \coloneqq 
        \del{ 
            \sum_{n=1}^\infty
                \rho_n^p
            }^{1/p}
    $.
    Then, for all $n \in \bN$ we have 
    $\rho \ge \rho_n$, so $\rho \in F( \bigcup_{ n \in \bN} \Gamma_n )$.
    Therefore,
    \begin{equation*}
        \modulus^p \! \! \del{ 
            \bigcup_{ n =1}^\infty \Gamma_n
        }
        \le 
            \norm{ \rho }^p_{L^p(\measure)}
        \!
        =
            \integral{ X }{ 
                \sum_{n=1}^\infty 
                    \rho_n^p
            }{ \measure }
        =
            \sum_{n=1}^\infty 
                \integral{ X }{ 
                        \rho_n^p
                }{ \measure }
        \le 
            \sum_{n=1}^\infty 
                \modulus^p( \Gamma_n ) + \eps 2^{-n}
        =
            \sum_{n=1}^\infty 
                \modulus^p( \Gamma_n )
            +
            \eps.
    \end{equation*}
    Since $\eps > 0$ is arbitrary, we have 
    $$
        \modulus^p \! \! \del{ 
            \bigcup_{ n =1}^\infty \Gamma_n
        }
        \le
        \sum_{n=1}^\infty 
                \modulus^p( \Gamma_n )
    $$
    as needed.
    
    Now, we shall prove the characterization of sets of $p$-modulus $0$. Let $\rho \in F(\Gamma)$ be such that $\rho \in L^p( \measure )$ and $\sint{ \gamma }{ \rho } = \infty$ for all $\gamma \in \Gamma$. 
        Then $\rho /n \in F(\Gamma) $ for all $n \in \bN$. 
        Hence,
        $$
            \modulus^p(\Gamma)
            \le 
            \norm{ \rho / n }^p_{L^p(\measure)}
            =
            \frac{1}{n^p} \norm{ \rho }^p_{L^p(\measure)}
            \to
            0
        $$
        as $n \to \infty$.

        In order to prove the other direction, let $\modulus^p(\Gamma) = 0$.
        Then for all $n \in \bN$ 
        there is $\rho_n \in F(\Gamma)$
        such that $ \norm{ \rho_n }_{L^p(\measure)} \le 2^{-n}$.
        Then 
        $
            \rho
            \coloneqq 
            \sum_{n=1}^\infty 
                \rho_n
            \in 
            F(\Gamma)
        $,
        we have 
        $
            \norm{ \rho }_{L^p(\measure)}
            \le 
            \sum_{n=1}^\infty 
                \norm{ \rho_n }_{L^p(\measure)}
            \le 
            1
        $,
        and 
        $
            \sint{ \gamma }{ \rho } 
            =
            \sum_{n=1}^\infty 
                \int_\gamma^S \rho_n
            =
            \infty
        $
        for all $\gamma \in \Gamma $, as claimed.
\end{proof}
\begin{corollary}
    If $\rho \in L^p(\measure)$ is Borel, 
    then $\int_\gamma^S \rho < \infty$ 
    for $\modulus^p$ almost every $\gamma \in \testCurvesNonTrivial$.
\end{corollary}
\begin{lemma}[Fuglede Lemma]
    \label{lem:Fuglede_lemma}
    Let $(X, \metricAlone)$ be a metric space,  $\measure$ be a Borel measure on $X$ 
    and $f_n, f \colon X \to \hat{\bR}$ 
    be Borel and such that $f_n \to f$ in $L^p(\measure)$, where $p \in \intco{1, \infty}$. Then, there exists a subsequence $(f_{n_k})_k$ such that 
    \begin{equation*}
                \modulus^p-\text{a.e.}
            \sint{\gamma}{ \abs{ f_{n_k} - f} }
            \to 0.
    \end{equation*}
\end{lemma}
\begin{proof}
    Let $(f_{n_k})_k$ be a subsequence of $(f_n)_n$ such that
    \begin{equation*}
        \forall k \in \bN
        \qquad 
        \norm{ 
                f_{n_k} - f
        }_{L^p(\measure)}^p
        \le 
        2^{-k(p+1)}.
    \end{equation*}
    For $k \in \bN$ let us define
    $
        \rho_k \coloneqq \abs{ f_{n_k} - f }
    $
    and 
    $
        \Gamma_k 
        \coloneqq 
        \set{
            \gamma \in \testCurvesNonTrivial
            \given 
            \sint{\gamma}{ \rho_k }
                \ge 
            2^{-k}
        }
    $.
    Then, for all $k \in \bN$ we have $2^k \rho_k \in F( \Gamma_k )$ and
    \begin{equation*}
        \modulus^p( \Gamma_k )
        \le 
        \norm{ 2^k \rho_k }_{L^p(\measure)}^p
        =
        2^{pk}
        \norm{ 
            \rho_k
        }_{L^p(\measure)}^p
        \le 
        2^{pk} \cdot 2^{-k(p+1)}
        =
        2^{-k}.
    \end{equation*}

    Since
    $
        \Gamma 
        \coloneqq 
        \set{
            \gamma \in \testCurvesNonTrivial
            \given 
            \sint{\gamma}{ \rho_k }
            \not \to 
            0
        }
    $
    satisfies
    $\Gamma \subseteq \bigcup_{k=m}^\infty \Gamma_k$
    for all $m \in \bN$,
    we have 
    \begin{equation*}
        \modulus^p( \Gamma )
        \le 
        \sum_{k=m}^\infty 
            \modulus^p( \Gamma_k )
        \le 
        \sum_{k=m}^\infty 
            2^{-k}
        =
        2^{-m+1}
        \xrightarrow{ m \to \infty }
        0.
    \end{equation*}
    Hence $\modulus^p( \Gamma ) = 0$ and $(f_{n_k})_k$ is the sought subsequence.
\end{proof}
\begin{definition}[Upper S-gradient, p-weak upper S-gradient]
    Let $(X, \metricAlone)$ be a metric space and $f \colon X \to \bRExtended$.
    We shall say that a Borel function 
    $g \colon X \to \intcc{0, \infty }$
    is an \emph{upper $S$-gradient} of $f$
    if
    \begin{equation*}
        \forall 
            \gamma \in \testCurvesNonTrivial
        \qquad 
            \abs{
                f \circ \gamma(1)
                - f \circ \gamma(0)
            }
            \le 
                \int_\gamma^S g
    \end{equation*}
    and a \emph{$p$-weak upper $S$-gradient} of $f$, if
    \begin{equation*}
        \modulus^p-\text{a.e.}
        \qquad 
            \abs{
                f \circ \gamma(1)
                - f \circ \gamma(0)
            }
            \le 
                \int_\gamma^S g.
    \end{equation*}
\end{definition}
We will denote the family of upper $S$-gradients of $f$ by $\upperSGradient f$ and the family of $p$-weak upper $S$-gradients of $f$ by $\pWeakUpperSGradient f$.
\begin{remark}
    Let $(X, \metricAlone)$ be a metric space, $\measure$ be a Borel measure on $X$ and $f \colon X \to \bRExtended$ be a measurable map. If $g, g' \colon X \to [0,\infty]$ are Borel functions such that $g$ is a $p$-weak upper $S$-gradient of $f$ and $g' = g$ $\measure$ almost everywhere, then $g'$ also is a $p$-weak upper $S$-gradient of $f$.
\end{remark}
\begin{proof}
    Let us consider a sequence $(\rho_n)_n$ of Borel functions defined by $\rho_n \coloneqq \abs{ g - g' }$.
    Then $\rho_n \to  0$ in $L^p(\measure)$.
    Hence, by the Fuglede Lemma (Lemma \ref{lem:Fuglede_lemma}) we have a subsequence
    $(\rho_{n_k})_k$ such that
    \begin{equation*}
        \modulus^p-\text{a.e.}
        \qquad 
        \sint{ \gamma }{ \abs{ g - g' }}
        =
        \sint{ \gamma }{ \rho_{n_k }}
        \xrightarrow{ k \to \infty }
        0.
    \end{equation*}
    Hence, 
    $
        \sint{ \gamma }{ \abs{ g - g' }}
        =
        0
    $
    for $\modulus^p$ almost every $\gamma \in \testCurvesNonTrivial$.
    Therefore,
    \begin{equation*}
        \abs{
            f \circ \gamma( b) 
            - f \circ \gamma(a)
        }
        \le 
        \sint{ \gamma }{ g }
        \le 
        \sint{ \gamma }{ g' }
        + \sint{ \gamma }{ \abs{ g- g' } }
        =
        \sint{ \gamma }{ g' }
    \end{equation*}
    for $\modulus^p$ almost every $\gamma \in \testCurvesNonTrivial$.
\end{proof}
\begin{example}

    Consider the space $(X, |.|, \lambda)$, where $X = [0,1],$ $|.|$ is the Euclidean metric, and $\lambda$ is the Lebesgue measure and let $f \colon X \to \bR$ be defined by $f \coloneqq \infty  \indicator{\set{1}}$.
    Then there is no Borel $g \colon X  \to \intcc{0,\infty}$ such that $g \in L^p(X)$ and $g$ is a $p$-weak upper $S$-gradients of $f$, even though $f = 0$ $\measure$-almost everywhere and every nonnegative Borel function $g$ is an upper $S$-gradient of the zero function.

\end{example}
\begin{proof}
    First, let us note that if $g \colon X \to \intcc{0, \infty}$ is Borel, then 
    $\sint{\gamma} g \ge 0$ for any 
    $\gamma \in \testCurvesNonTrivial$,
    hence the latter claim follows.

    Now, let us suppose that $g \colon X \to \intcc{0,\infty}$ is a $p$-weak upper $S$-gradient of $f$ such that $g \in L^p(X)$. Consider $\gamma \colon [0,1] \to X$
    given by $\gamma(t) = t$ for $t \in [0,1]$. 
    Clearly, $\gamma \in \testCurvesNonTrivial$. 
    Furthermore 
    $
        \abs{ f \circ \gamma(1) - f \circ \gamma (0)} = \infty,
    $
    but
    $
        \sint{\gamma} g
        =
        \int_0^1 g(t) \ \mathrm{d}t
        \le 
        \norm{ g }_{L^p(X)}
        < 
        \infty.
    $
    Hence, 
    $
        \abs{ f \circ \gamma(1) - f \circ \gamma (0)}
        > 
        \sint{\gamma} g.
    $
    Therefore, $\modulus^p(\set{\gamma}) = 0$ as $g$ is a $p$-weak upper $S$-gradient of $f$. 
    By Proposition \ref{mod0} there is $\rho \in L^p(X)$ such that 
    $\sint{\gamma} \rho = \infty$. 
    However, 
    $
        \sint{\gamma} \rho
        =
        \int_0^1 \rho(t) \ \mathrm{d}t
        \le 
        \norm{ \rho }_{L^p(X)}
        < 
        \infty,
    $
    so such $\rho$ cannot exist. In consequence, 
    $g$ is not a $p$-weak upper $S$-gradient of $f$. 
\end{proof}
\begin{definition}[TC-Newtonian Space]
    Let $(X, \metricAlone)$ be a metric space and $\measure$ be a Borel measure on $X$. Let $p \in \intco{1, \infty}$. We define a space
    \begin{equation*}
        \newtonianTCPre^{1,p}(X)
        \coloneqq 
        \set{
            f: X \rightarrow \mathbb{R}
            \given
            \int_X |f|^p \,\mathrm{d} \measure< \infty \,\text{and} \, L^p(\measure) \cap \pWeakUpperSGradient f \neq \emptyset 
            }
        .
    \end{equation*}
    On space $\newtonianTCPre^{1,p}(X)$ we define a seminorm
    \begin{equation*}
        \forall 
            f \in \newtonianTCPre^{1,p}(X)
        \qquad 
            \norm{ f }_{\newtonianTC^{1,p}(X)} 
            \coloneqq 
            \norm{ f }_{L^p(\measure)}
            + 
            \inf_{ g }
                \norm{ g }_{ L^p( \measure) },
    \end{equation*}
    where the infimum is taken over all $p$-weak
    upper $S$-gradients $g$ of $f$. Also, we define an equivalence relation $\sim$ by
    \begin{equation*}
        \forall 
            f, f'\in \newtonianTCPre^{1,p}(X)
        \qquad 
            f \sim f'
            \iff
            \norm{ f - f' } _{\newtonianTC^{1,p}(X)}
            =
            0.
    \end{equation*}
    Finally, we define the TC-Newtonian space, as the quotient space  $\newtonianTC^{1,p}(X) \coloneqq \newtonianTCPre^{1,p}(X) / {\sim}$.
    On this space $\norm{ \cdot } _{\newtonianTC^{1,p}(X)}$ is a norm.
\end{definition}
\begin{lemma}\label{lem::almost_acc_p}

    Let $(X, \metricAlone)$ be a metric space with Borel measure $\measure$. For a function $f:X \rightarrow \hat{\bR}$ and $g \in  L^p(\measure) \cap \pWeakUpperSGradient f $ let us define 
    \begin{align*}
        \Gamma_1
        &\coloneqq 
        \set{
            \gamma \in \testCurvesNonTrivial
            \given 
            \sint{ \gamma  }
            g 
            =
            \infty
        }, 
        \\
        \Gamma_2
        &\coloneqq 
        \set{
            \gamma \in \testCurvesNonTrivial
            \given 
            \exists
             s, t \in [0,1] 
             \text{ such that }
             s<t,
            \gamma(s) \ne \gamma(t^-)
            \text{ and }
            \abs{
                f (\gamma(s))
                - 
                f( \gamma(t^-))
            }
            > 
            \sint{ \gamma \rvert_{[s,t^-]} }
            g 
        },      
        \\
        \Gamma_3
        &\coloneqq 
        \set{
            \gamma \in \testCurvesNonTrivial
            \given 
            \exists
             t \in [0,1] 
             \text{ such that }
            \gamma(t^-) \ne \gamma(t)
            \text{ and }
            \abs{
                f (\gamma(t^-))
                - 
                f( \gamma(t))
            }
            > 
            \tfrac{g (\gamma(t^-))
                +
                g( \gamma(t))}{2}
            \metric\del{ \gamma(t^-), \gamma(t)}
               }.      
    \end{align*}
    Then $\modulus^p(\Gamma_1 \cup \Gamma_2 \cup \Gamma_3) = 0$.
\end{lemma}
\begin{proof}
    First, let us notice that since 
    $
        g \in L^p(\measure),
    $
    then  $\modulus^p( \Gamma_1)=0 $
    by Proposition \ref{mod0}.
    Since $g$ is a $p$-weak upper $S$-gradient of $f$, 
    the family
    \begin{equation*}
        \Gamma'
        \coloneqq 
        \set{
            \gamma \in \testCurvesNonTrivial
            \given 
            \abs{
                f (\gamma(1))
                - 
                f( \gamma(0))
            }
            > 
            \sint{ \gamma  } g
        }
    \end{equation*} 
    has a $p$-modulus of $0$. 
    In consequence, by Proposition \ref{mod0} 
    there is $\rho \in L^p(\measure)$, $\rho \ge 0$ such that 
    $
        \sint{ \gamma  } \rho = \infty 
    $
    for all $\gamma \in \Gamma'$. 
    
    Fix $\gamma \in \Gamma_2 \setminus \Gamma_1$. 
    There are $s, t \in [0,1]$ such that 
    $\gamma(s) \ne \gamma(t^-)$ 
    and
    $
        \abs{
                f (\gamma(s))
                - 
                f( \gamma(t^-))
            }
            > 
            \sint{ \gamma \rvert_{[s,t^-]} }
            g.
    $
    Let us define $\Phi :[0,1] \rightarrow [s,t]$ as follows $\Phi(x)=s+(t-s)x$. Then by Proposition \ref{rem:test_Curves_changing_intervals} the map $\tilde{\gamma} =  \Phi_{\#}   ( {\gamma} \rvert_{ [ s, t^-]})$ belongs to $\testCurves\del{[0,1]; X}$ and we have
    \[
         \abs{
                f (\tilde{\gamma}(0))
                - 
                f(\tilde{ \gamma}(1))
            }=
        \abs{
                f (\gamma(s))
                - 
                f( \gamma(t^-))
            }
            > 
            \sint{ \gamma \rvert_{[s,t^-]} }
            g=
            \sint{ \tilde{\gamma} }
            g.
    \]    
    Thus, $\tilde{\gamma} \in \Gamma'$
    and 
    $
        \sint{ \tilde{\gamma} } \rho = \infty.
    $
    Therefore, 
    $
        \sint{ \gamma } \rho 
        \ge 
         \sint{ \gamma \rvert_{[s,t^-]} } \rho = \sint{ \tilde{\gamma} } \rho
         = \infty,
    $
    hence $\modulus^p(\Gamma_2 \setminus \Gamma_1) = 0$
    by Proposition \ref{mod0}.

    Now, fix $\gamma \in \Gamma_3 \setminus \Gamma_1$.
    There is $t \in [0,1]$ such that 
    $\gamma(t^-) \ne \gamma(t)$
    and
    $
            \abs{
                f (\gamma(t^-))
                - 
                f( \gamma(t))
            }
            > 
            \tfrac{g (\gamma(t^-))
                +
                g( \gamma(t))}{2}
            \metric\del{ \gamma(t^-), \gamma(t)}.
    $
    Let $\gamma' \colon [0,1] \to X$ be defined as follows 
    $\gamma'(s) = \gamma(t^-)$ for $s \in \intco{0,1/2}$ and 
    $\gamma'(s) = \gamma(t)$ for $s \in \intcc{1/2,1}$.
    Clearly, $\gamma' \in \testCurvesNonTrivial$ and by Example \ref{prz325} we have
    \begin{multline*}
         \abs{
                f (\gamma'(0))
                - 
                f( \gamma'(1))
            }
        =
         \abs{
                f (\gamma(t^-))
                - 
                f( \gamma(t))
            }
        \\
            > 
            \frac{g (\gamma(t^-))
                +
                g( \gamma(t))}{2}
            \metric\del{ \gamma(t^-), \gamma(t)}
        =
        \frac{g (\gamma'(0))
                +
                g( \gamma'(1))}{2}
            \metric\del{ \gamma'(0), \gamma'(1)}
        =
        \sint{\gamma'} g,
    \end{multline*}
    hence $\gamma' \in \Gamma'$.
    Thus, $\sint{\gamma'} \rho = \infty$ 
    and 
    \begin{equation*}
        \sint{\gamma} \rho
        \ge 
        \frac{\rho (\gamma(t^-))
                +
                \rho( \gamma(t))}{2}
            \metric\del{ \gamma(t^-), \gamma(t)}
        =
        \frac{\rho (\gamma'(0))
                +
                \rho( \gamma'(1))}{2}
            \metric\del{ \gamma'(0), \gamma'(1)}
        =
        \sint{\gamma'} \rho
        =
        \infty.
    \end{equation*}
    In consequence, $\modulus^p(\Gamma_3 \setminus \Gamma_1) = 0$ 
    by Proposition \ref{mod0}.
        
    Finally, since $p$-modulus is an outer measure, 
    $$
        \modulus^p(\Gamma_1 \cup \Gamma_2 \cup \Gamma_3)
        \le 
        \modulus^p(\Gamma_1)
        +\modulus^p(\Gamma_2 \setminus \Gamma_1)
        +\modulus^p(\Gamma_3 \setminus \Gamma_1)
        =
        0.
    $$
\end{proof}
\begin{proposition}\label{prop::when_ae_0_gives_0_in_Newtonian}

    Let $(X, \metricAlone)$ be a metric space and $\measure$ be a Borel measure on $X$. Let 
    $f\colon X \rightarrow \hat{\bR}$ be a measurabe map  such that $f = 0$ $\measure$-almost everywhere and $ L^p(\measure) \cap \pWeakUpperSGradient f \neq \emptyset$.
    If $f$ is Borel or $\measure$ is Borel regular, then 
     $0$ is a $p$-weak upper $S$-gradient of $f$. 

\end{proposition}
\begin{proof}
    Consider the set 
    $
        E'
        \coloneqq 
        \set{
            x \in X 
            \given 
            f(x) \ne 0
        }.
    $
    Since $f = 0$ $\measure$-almost everywhere, 
    we have $\measure(E') = 0$. 
    If $f$ is Borel, then $E'$ is Borel.
    If $\measure$ is Borel regular, then 
    there is a Borel set $E$ of measure $0$
    such that $E' \subseteq E$. 
    In both cases, we have a Borel set $E$ such that
    $\measure(E) = 0$ and $E' \subseteq E$.

    Let us consider the families
    \begin{equation*}
        \Gamma_{E'}
        \coloneqq 
        \set{
            \gamma \in \testCurvesNonTrivial
            \given 
            E' \cap \closure{ \image{\gamma}}
            \ne 
            \emptyset 
        }
        \quad \text{and} \quad
        \Gamma_E^+
        \coloneqq 
        \set{
            \gamma \in \testCurvesNonTrivial
            \given 
            \mu_\gamma^S\del{
                E \cap \closure{ \image{\gamma}}
            }
            >
            0
        }.
    \end{equation*}
    Our next goal is to show that 
    $\modulus^p(\Gamma_{E'}) = 0$. 
    
    First, we will show that 
    $
        \modulus^p\del{ \Gamma_E^+ } = 0.
    $
    Indeed, let us consider a function 
    $
        \rho \coloneqq \infty \indicator{E}.
    $
    Clearly, $\rho$ is a Borel function such that 
    $ \rho = 0 $ $\measure$-almost everywhere, and 
    hence $\norm{ \rho }_{L^p( \measure )} = 0$.
    If 
    $
        \gamma \in \Gamma_E^+,
    $
    then we have 
    $
        \sint{\gamma} \rho 
        =
        \int_{X} \rho \ \mathrm{d}\mu_\gamma^S 
        =
        \infty \mu_\gamma^S\del{
                E 
            } \geq  \infty \mu_\gamma^S\del{
                E \cap  \closure{ \image{\gamma}}
            }
        =
        \infty.
    $
    Therefore, 
    $
        \modulus^p\del{ \Gamma_E^+ } = 0
    $
    by Proposition \ref{mod0}.

    Since $ L^p(\measure) \cap \pWeakUpperSGradient f \neq \emptyset $, there is 
    $g \in L^p(\measure)$ that is a $p$-weak upper $S$-gradient of $f.$ Let 
    $\Gamma \coloneqq \Gamma_1 \cup \Gamma_2\cup \Gamma_3$, 
    where $\Gamma_1, \Gamma_2, \Gamma_3$ are as in Lemma \ref{lem::almost_acc_p}.
    We will show that 
    $
        \Gamma_{E'} \setminus \Gamma_E^+
        \subseteq 
        \Gamma.
    $
    Fix $\gamma \in \Gamma_{E'} \setminus \Gamma_E^+$.
    If
    $
        \sint{\gamma} g = \infty,
    $
    then $\gamma \in \Gamma$. 
    Suppose that 
    $
        \sint{\gamma} g < \infty
    $
    and 
    $
        \gamma \notin \Gamma.
    $

    Then for all $s, t \in [0,1]$ such that $s<t$ and $\gamma(s) \ne \gamma(t^-)$ we have $
        \abs{
            f(\gamma(s)) - f(\gamma(t^-))
        }
        \le 
        \sint{ \gamma \rvert_{ [s, t^-]} } g.
    $ 
       Also, for all $t \in [0,1]$ such that $\gamma(t^-) \ne \gamma(t)$ 
    we have 
    $
        \abs{
                f (\gamma(t^-))
                - 
                f( \gamma(t))
            }
            \le 
            \frac{g (\gamma(t^-))
                +
                g( \gamma(t))}{2}
            \metric\del{ \gamma(t^-), \gamma(t)}.
    $
    We will show that $f \circ \gamma$ has finite values.

    First, we will show that $f\circ \gamma(0)$ is finite.
    Notice that since $\gamma \in \Gamma_{E'} \subseteq \testCurvesNonTrivial$, 
    we have $V_\gamma > 0$ and there is $t \in \intoc{0,1}$ such that 
    $\gamma(0) \ne \gamma(t)$. 
    If $\gamma(0)=\gamma(t^-)$, then $\gamma(t^-) \ne \gamma(t)$, and both $f(\gamma(t^-))=f(\gamma(0))$ 
    and $f(\gamma(t))$ are finite since 
    \begin{equation*}
        \abs{
                f (\gamma(t^-))
                - 
                f( \gamma(t))
            }
            \le 
            \frac{g (\gamma(t^-))
                +
                g( \gamma(t))}{2}
            \metric\del{ \gamma(t^-), \gamma(t)}
            \le 
            \sint{\gamma } g
            < 
            \infty.
    \end{equation*}
    Otherwise, if $\gamma(0) \neq \gamma(t^-)$, then  $f(\gamma(0))$ and $f(\gamma(t^-))$ are finite because
    \begin{equation*}\label{eq::ineq_for_f_[0,t-]}
        \abs{
                f (\gamma(0))
                - 
                f( \gamma(t^-))
            }
            \le 
            \sint{\gamma \rvert_{[0,t^-]} } g
            \le 
            \sint{\gamma } g
            < 
            \infty.
    \end{equation*}

    Now, let us observe that $f(\gamma(s^-))$ is finite for all $s \in \intoc{0,1}$.
    Indeed, if $\gamma(0) = \gamma(s^-)$, then this fact is clear. 
    If $\gamma(0) \ne \gamma(s^-)$ then $f(\gamma(s^-))$ is finite because
    \begin{equation*}
        \abs{
                f (\gamma(0))
                - 
                f( \gamma(s^-))
            }
            \le 
            \sint{\gamma \rvert_{[0,s^-]} } g
            \le 
            \sint{\gamma } g
            < 
            \infty.
    \end{equation*}
    Finally, we have that $f(\gamma(s))$ is finite for all $s \in \intoc{0,1}$.
    Indeed, if $\gamma(s^-) = \gamma(s)$, then it is clear.
    If $\gamma(s^-) \ne \gamma(s)$, then we have
    \begin{equation*}
        \abs{
                f (\gamma(s^-))
                - 
                f( \gamma(s))
            }
            \le 
            \frac{g (\gamma(s^-))
                +
                g( \gamma(s))}{2}
            \metric\del{ \gamma(s^-), \gamma(s)}
            \le 
            \sint{\gamma } g
            < 
            \infty.
    \end{equation*}

     We will show next that for all $t \in [0,1]$ we have
    \begin{equation}\label{eq::continuity_properties_of_f_circ_gamma}
        f( \lim_{r \to t^-} \gamma(r))
        =
        \lim_{r \to t^-} f( \gamma(r))
        \quad \text{and}\quad 
        f( \lim_{r \to t^+} \gamma(r))
        =
        \lim_{r \to t^+} f( \gamma(r)).
    \end{equation}
    First, let $(r_n)_n$ be a sequence of elements of $\intoo{0,t}$
    such that $r_n \to t^-$.
    Then $\sint{\gamma \rvert_{[r_n, t^-]}} g \to 0$ as $n \to \infty$ by Lemma \ref{lem::descending_family_of_subsets_to_0}.
    For all $n \in \bN$ we either have $\gamma(r_n) =\gamma(t^-)$ 
    or $\gamma(r_n) \ne\gamma(t^-)$. Since in the former case we have 
    $\abs{f(\gamma(r_n)) -f(\gamma(t^-))} = 0$ due to the finiteness of 
    $f\circ \gamma$, in both cases we have
    \begin{equation*}
        \abs{
                f (\gamma(r_n))
                - 
                f( \gamma(t^-))
            }
            \le 
            \sint{\gamma \rvert_{[r_n,t^-]} } g
            \xrightarrow{n \to \infty}
            0,
    \end{equation*}
    proving the first equality in \eqref{eq::continuity_properties_of_f_circ_gamma}.
    Next, let $(r_n)_n$ be a sequence of elements of $\intoc{t,1}$
    such that $r_n \to t^+$.
    By Lemma \ref{lem::descending_family_of_subsets_to_0} we have
    \begin{equation*}
        \sint{\gamma \rvert_{[t, r_n^-]}} g 
        +
        \frac{g (\gamma(r_n^-))
                +
                g( \gamma(r_n))}{2}
            \metric\del{ \gamma(r_n^-), \gamma(r_n)}
        \xrightarrow{n \to \infty} 0.
    \end{equation*}
    For all $n \in \bN$ we either have $\gamma(t) = \gamma(r_n)$, 
    $\gamma(r_n) \ne\gamma(t)$ and $\gamma(r_n^-) = \gamma(r_n)$, 
    or  $\gamma(r_n) \ne\gamma(t)$ and $\gamma(r_n^-) \ne \gamma(r_n)$. In the first case we have 
    $\abs{f(\gamma(t)) - f(\gamma(r_n)) } = 0$ due to the finiteness of 
    $f\circ \gamma$. In the second case, we have
    \begin{equation*}
        \abs{
                f (\gamma(t))
                - 
                f( \gamma(r_n))
            }=
        \abs{
                f (\gamma(t))
                - 
                f( \gamma(r_n^-))
            }
            \le 
            \sint{\gamma \rvert_{[t,r_n^-]} } g
            \xrightarrow{n \to \infty}
            0.
    \end{equation*}
    In the last case we have 
    \begin{multline*}
        \abs{
                f (\gamma(t))
                - 
                f( \gamma(r_n))
            }
        \le 
        \abs{
                f (\gamma(t))
                - 
                f( \gamma(r_n^-))
            }
            +
            \abs{
                f (\gamma(r_n^-))
                - 
                f( \gamma(r_n))
            }
            \\
            \le 
            \sint{\gamma \rvert_{[t, r_n^-]}} g 
            +
            \frac{g (\gamma(r_n^-))
                    +
                    g( \gamma(r_n))}{2}
                \metric\del{ \gamma(r_n^-), \gamma(r_n)}
            \xrightarrow{n \to \infty} 0.
    \end{multline*}
    Therefore, in all cases we have $\abs{f(\gamma(t)) - f(\gamma(r_n))} \to 0$ as $n \to \infty$,
    proving the second equality in \eqref{eq::continuity_properties_of_f_circ_gamma}.

    Now, we shall show that $f \circ \gamma$ is continuous on entire $[0,1]$.    
    By \eqref{eq::continuity_properties_of_f_circ_gamma} we have $f \circ \gamma$ is continuous at points of continuity of $\gamma$. 
     On the other hand, notice that if $t$ is a point of discontinuity of $\gamma$, 
    then since $\gamma(t^-)= \overleftarrow{\gamma}(1-t)$,  
    by Lemma \ref{lem:reversing_test_curve} we have 
    \begin{eqnarray*}
        \mu_\gamma^S\del{ \{\gamma (t^-) \} }
        &\ge& 
        \frac{1}{2} \overleftarrow{\gamma}_\# \mu_{\overleftarrow{\gamma}} \del{\{\overleftarrow{\gamma}(1-t) \} } 
        \ge
        \frac{1}{2}  \mu_{\overleftarrow{\gamma}} \del{\{1-t \} } \\
        &=&
         \frac{1}{2} \left(V_{\overleftarrow{\gamma}}(1-t) - V_{\overleftarrow{\gamma}}((1-t)^-) \right)
        =
        \frac{1}{2} \left(\overleftarrow{V_{\gamma}}((1-t)^-) - \overleftarrow{V_{\gamma}}(1-t) \right)\\
        &=& \frac{1}{2}\left(V_{\gamma}(t) - V_{\gamma}(t^-)\right)= \frac{1}{2} \metric\del{
            \gamma(t^-), \gamma(t)
        }
        >
        0
    \end{eqnarray*}
       and, similarly,
    $$
        \mu_\gamma^S(\set{\gamma(t)}) 
        \ge 
        \frac{1}{2} \metric\del{ \gamma(t^-), \gamma(t) }
        >
        0.
    $$
    Hence $\set{ \gamma(t^-), \gamma(t) } \cap E = \emptyset$
    as $\gamma \notin \Gamma_E^+$. 
    In consequence, $f(\gamma(t^-)) = f(\gamma(t)) = 0$
    and by (\ref{eq::continuity_properties_of_f_circ_gamma}) we have
    \begin{equation*}
        \lim_{r \to t^-} f( \gamma(r))
        =
        f( \lim_{r \to t^-} \gamma(r))
        =
        0
        =
        f(\gamma(t))
        =
        f( \lim_{r \to t^+} \gamma(r))
        =
        \lim_{r \to t^+} f( \gamma(r)).
    \end{equation*}
    Thus, $f \circ \gamma$ is also continuous at points of discontinuity of $\gamma$. 
    
    Since $\gamma \in \Gamma_{E'}$, there exists 
    $ x \in \closure{ \image{ \gamma }}$ such that 
    $f(x) \ne 0$. 
    Therefore by Remark \ref{rem:closure_of_the_image_of_a_test_curve_is_compact} there exists $t \in [0,1]$ such that $x = \gamma(t)$
    or $x = \gamma(t^-)$.
    However, since $f \circ \gamma$ is continuous on $[0,1]$ and 
    $ f( \lim_{r \to t^-} \gamma(r))     =
        \lim_{r \to t^-} f( \gamma(r))$, 
    in both cases we have $f(x) = f(\gamma(t))$.
    Recall that 
    $\gamma \in \testCurvesNonTrivial \setminus \Gamma_E^+$,
    so $\mu_\gamma^S(X) > 0$ and $\mu_\gamma^S(E) = 0$. 
    In consequence, there is $s' \in [0,1]$ such that 
    $\gamma(s') \notin E$ or $\gamma(s'^-) \notin E$,
    hence, as $f \circ \gamma$ is continuous and $ f( \lim_{r \to s'^-} \gamma(r))
        =
        \lim_{r \to s'^-} f( \gamma(r))$, we have
    $f(\gamma(s')) = 0$. 
    First, suppose that $s' < t$.
    Let 
    $ 
        s 
        \coloneqq 
        \sup 
        \set{ 
            s' \in [0,1] 
            \given 
            f(\gamma(s')) = 0
            \text{ and }
            s' < t
        }.
    $
    Clearly, 
    $s < t$, 
    $\gamma(s) \ne \gamma(t)$, 
    and
    $\gamma \sbr{  \intoc{ s, t}} \subseteq E$.
    In consequence, recalling the definitions from subsection \ref{subsec::integration_along_a_curve},
    \begin{equation*}
        2\mu_\gamma^S(E)
        \ge 
        2\mu_\gamma^S\del{
            \gamma \sbr{  \intoc{ s, t}}
        }
        \ge 
        \gamma_\# \mu_\gamma \del{
            \gamma \sbr{  \intoc{ s, t}}
        }
        \geq
        \mu_\gamma \del{ \intoc{ s, t} }
        =
        V_\gamma(t) - V_\gamma(s)
        \ge 
        \metric\del{ \gamma(t), \gamma(s) }
        >
        0,
    \end{equation*}
    which contradicts that $\gamma \notin \Gamma_E^+$. 
    Now, suppose that $s' > t$. 
    Let
    $
        s 
        \coloneqq 
        \inf 
        \set{ 
            s' \in [0,1] 
            \given 
            f(\gamma(s')) = 0
            \text{ and }
            s' > t
        }.
    $
    Clearly, 
    $f(\gamma(s)) = 0$,
    so also $f(\gamma(s^-)) = 0$.
    Thus,
    $s > t$, 
    $\gamma(s^-) \ne \gamma(t)$, 
    and
    $\gamma \sbr{  \intoo{ t, s }} \subseteq E$.
    In consequence, 
    \begin{equation*}
        2\mu_\gamma^S(E)
        \ge 
        2\mu_\gamma^S\del{
            \gamma \sbr{  \intoo{ t, s}}
        }
        \ge 
        \gamma_\# \mu_\gamma \del{
            \gamma \sbr{  \intoo{ t, s}}
        }
        \geq
        \mu_\gamma \del{ \intoo{ t, s} }
        =
        V_\gamma(s^-) - V_\gamma(t)
        \ge 
        \metric\del{ \gamma(s^-), \gamma(t) }
        >
        0,
    \end{equation*}
    which also contradicts that $\gamma \notin \Gamma_E^+$.

    The obtained contradiction shows that there is no 
    $\gamma \in \Gamma_{E'} \setminus \Gamma_E^+$
    such that $\gamma \notin \Gamma$, 
    hence 
    $\Gamma_{E'} \setminus \Gamma_E^+ \subseteq \Gamma$. 
    From Lemma \ref{lem::almost_acc_p} we have 
    $\modulus^p(\Gamma) = 0$.
    Thus, recalling that 
    $\modulus^p(\Gamma_E^+) = 0$,
    we have
    \begin{equation*}
        \modulus^p(\Gamma_{E'})
        \le 
        \modulus^p(\Gamma_{E'} \setminus \Gamma_E^+)
        + \modulus^p(\Gamma_E^+)
        \le 
        \modulus^p( \Gamma )
        + \modulus^p(\Gamma_E^+)
        =
        0.
    \end{equation*}
    This implies that $0$ is a $p$-weak upper $S$-gradient of
    $f$.
    Indeed, 
    we have $\modulus^p(\Gamma_{E'}) = 0$
    and if 
    $
        \gamma 
        \in 
        \testCurvesNonTrivial \setminus \Gamma_{E'},
    $
    then 
    $
        f(\gamma(0)) = f(\gamma(1)) = 0
    $
    and 
    $
        \abs{ f(\gamma(0)) - f(\gamma(1)) }
        = 
        0
        =
        \sint{\gamma} 0,
    $
        which ends the proof.

\end{proof}

\subsection{$\hat{\testCurves}$-Newtonian Space} \label{subsec::TC-hat_Newtonian_space}
In this subsection we introduce the $\hat{\testCurves}$-Newtonian spaces. 
On the one hand, in the definition of $\hat{\testCurves}$-Newtonian spaces we will require that a certain condition is satisfied for sufficiently many Borel measures defined on $\testCurves\del{[0,1]; X}$. 
On the other hand, the sense in which we understand the ``sufficiently many'' part mimics the similar requirement defined for $\newtonianTC$-Newtonian spaces using the $p$-modulus.

As we will see, the main difference between the $\hat{\testCurves}$-Newtonian spaces and the $\testCurves$-Newtonian spaces arises due to the fact that we require that not just the upper gradients but also the elements of the $\hat{\testCurves}$-Newtonian space are equivalence classes of Borel functions.

\begin{definition}[Integral on space of Test Curves]
    Let $(X, \metricAlone)$ be a metric space.
    Let $\mu$ be a Borel measure on the space
    $
        \del{
            \testCurves\del{ [0,1]; X},
            \metricAlone_\measurable
        }
    $.
    For $f \colon X \to \bRExtended$ Borel, bounded from below or above,
    we define its integral along $\mu$ by    
    \begin{equation*}
        \int_\mu f
        \coloneqq 
        \integral{
            \testCurves\del{ [0,1]; X}
        }{
            \quad \int_\gamma f \quad
        }{
            \mu(\gamma)
        }
    \end{equation*}
    and the symmetrized integral along $\mu$ by
    \begin{equation*}
       \sint{ \mu }{ f }
        \coloneqq 
        \integral{
            \testCurves\del{ [0,1]; X}
        }{
            \quad \sint{ \gamma }{ f } \quad
        }{
            \mu(\gamma)
        }.
    \end{equation*}
\end{definition}
Let us remark that by Theorem \ref{lem::hyper_envaluations_are_Borel} the maps 
$ \testCurves\del{ [0,1]; X}
              \ni \gamma
        \mapsto 
        \int_\gamma f
 $ and
 $
              \testCurves\del{ [0,1]; X}
        \ni 
        \gamma
        \mapsto 
        \sint{\gamma}{ f}
$ are Borel.
\begin{definition}[Generalized p-modulus]
    Let $(X, \metricAlone)$ be a metric space, then we will denote the family of \say{test measures} by
    \begin{equation*}
        \testMeasuresNonTrivial
        \coloneqq 
        \set{
            \mu 
            \text{ Borel measure on }
            \testCurves\del{ [0,1]; X }
            \given 
            \mu\del{
                \testCurves\del{ [0,1]; X }
                \setminus 
                \testCurvesNonTrivial
            }
            =
            0
        }.
    \end{equation*}
     For $\Gamma \subseteq \testMeasuresNonTrivial$ we define
    \begin{equation*}
        \tilde{F}(\Gamma)
        \coloneqq 
        \set{
            \rho \colon X \to \intcc{ 0, \infty }
            \given 
            \rho \text{ is Borel and }
            \sint{ \mu} { \rho } 
            \ge 1
            \text{ for all }
            \mu \in \Gamma
        }.
    \end{equation*}
    Let $(X, \metricAlone)$ be a metric space and let $\measure$ be a Borel measure on $X$.  Then, for $p \in \intco{ 0, \infty }$ we define the generalized $p$-modulus of families of test measures
    $   
        \tilde{\modulus}^p 
        \colon 
        2^{ \testMeasuresNonTrivial }
        \to 
        \intcc{0, \infty}
    $
    by the formula
    \begin{equation*}
        \forall 
        \Gamma \subseteq \testMeasuresNonTrivial
        \qquad 
        \tilde{\modulus}^p( \Gamma)
        \coloneqq 
        \inf_{ \rho \in \tilde{F}(\Gamma) }
        \norm{ \rho }^p_{L^p(\measure)}.
    \end{equation*}

    We shall say that property $P$ holds for $\tilde{\modulus}^p$ almost every $\mu \in \testMeasuresNonTrivial$, if the set of $\mu \in \testMeasuresNonTrivial$ for which $P$ does not hold has a generalized $p$-modulus of $0$. 
\end{definition}
\begin{lemma}[Generalized $\bm{p}$-modulus agrees with $\bm{p}$-modulus]
\label{lem:generalized_p-modulus_agrees_with_p-modulus}
    Let $(X, \metricAlone)$ be a metric space and let $\measure$ be
    a Borel measure on $X$.
    Fix $\Gamma \subseteq \testCurvesNonTrivial$ and let
    $\tilde{\Gamma} \subseteq \testMeasuresNonTrivial$ 
    be defined by
    \begin{equation*}
        \tilde{\Gamma}
        \coloneqq 
        \set{
            \delta_{\gamma}
            \given 
            \gamma \in \Gamma
        }.
    \end{equation*}
    Then $\tilde{\modulus}^p(\tilde{\Gamma}) = \modulus^p(\Gamma)$.
\end{lemma}
\begin{proof}
    Let $\gamma \in \Gamma$, then for all Borel $\rho \colon X \to \intcc{0, \infty}$ we have
    \begin{equation*}
        \sint{\gamma}{\rho}
        =
        \integral{ \testCurves\del{ [0,1]; X} }
        { \quad \sint{\gamma'}{\rho} \quad }
        { \delta_{\gamma}(\gamma')}
        =
        \sint{\delta_{\gamma}}{\rho}.
    \end{equation*}
    Therefore, $ F(\Gamma)=\tilde{F}(\tilde{\Gamma})$.
    In consequence
    \begin{equation*}
        \tilde{\modulus}^p\del{ \tilde{\Gamma} }
        =
        \inf_{ \rho \in \tilde{F}\del{ \tilde{\Gamma}}}
        \norm{ \rho }_{L^p(\measure)}^p
        =
        \inf_{ \rho \in F\del{ \Gamma}}
        \norm{ \rho }_{L^p(\measure)}^p
        =
        \modulus^p\del{ \Gamma} 
    \end{equation*}
    and the claim is proved.
\end{proof}
\begin{lemma}[Absolute continuity of the generalized p-modulus with respect to the p-modulus]
\label{lem:absolute_continuity_for_two_moduli}
    Let $(X, \metricAlone)$ be a metric space and let $\measure$ be a Borel measure on $X$.
    Let $\Gamma \subseteq \testCurvesNonTrivial$ be a Borel set such 
    $\modulus^p(\Gamma) = 0$.
    Then for the family $\hat{\Gamma} \subseteq \testMeasuresNonTrivial$ 
    defined by
    \begin{equation*}
        \hat{\Gamma}
        \coloneqq 
        \set{
            \mu \in \testMeasuresNonTrivial
            \given 
            \mu\del{ \Gamma }
            > 
            0
        }
    \end{equation*}
    we have $\tilde{\modulus}^p\del{ \hat{\Gamma}} = 0$.
\end{lemma}
\begin{proof}
    Since $\modulus^p(\Gamma) = 0$, by Proposition \ref{mod0} there exists a Borel map $\rho \colon X \to \intcc{0, \infty}$ such that $\rho \in L^p(\measure)$ 
    and for all $\gamma \in \Gamma$ we have 
    $
        \sint{ \gamma}{ \rho } = \infty
    $.
    Let $\mu \in \hat{\Gamma}$. Then, since  $\mu\del{ \Gamma } > 0$, we have
    \begin{equation*}
        \sint{ \mu}{ \rho }
        =
        \integral{ \testCurves\del{ [0,1]; X} }{ 
            \quad \sint{ \gamma}{ \rho }\quad 
        }{ \mu(\gamma)}
        \ge 
        \integral{ \Gamma }{ 
            \quad \sint{ \gamma}{ \rho }\quad 
        }{ \mu(\gamma)}
        =
        \integral{ \Gamma }{ 
            \infty 
        }{ \mu(\gamma)}
        =
        \infty \mu\del{ \Gamma }
        =
        \infty
    \end{equation*}
    that is, $\sint{ \mu}{ \rho } = \infty$. 
    By the analogue for the generalized $p$-modulus to Proposition \ref{mod0}
    we have $\tilde{\modulus}^p\del{ \hat{\Gamma}} = 0$.
\end{proof}
\begin{definition}[Generalized $\bm{p}$-weak upper $\bm{S}$-gradient]
    Let $(X, \metricAlone)$ be a metric space and let $\measure$ be a Borel measure on $X$.
    For a Borel function $f \colon X \to \bRExtended$
    we will say that a Borel function $g \colon X \to \intcc{0, \infty}$
    is its generalized $p$-weak upper $S$-gradient, if
    \begin{equation*}
        \tilde{\modulus}^p-\text{a.e.}
        \qquad
            \integral{
                \testCurves\del{ [0,1]; X}
            }{
                \abs{ f \circ \gamma(1) - f \circ \gamma(0) }
            }{
                \mu(\gamma)
            }
            \le \sint{ \mu}{ g }.
    \end{equation*}
\end{definition}
Let us remark that by Theorem \ref{lem::hyper_envaluations_are_Borel} the map 
$ \testCurves\del{ [0,1]; X}
              \ni \gamma
        \mapsto |f \circ \gamma(1) - f \circ \gamma(0)|
 $ is Borel.\footnote{It is worth to remember about our notation.}
 We will denote the family of generalized $p$-weak upper $S$-gradients of $f$ by $\generalisedPWeakUpperSGradient f$
\begin{definition}[$\hat{TC}$-Newtonian Space]
    Let $(X, \metricAlone)$ be a metric space and $\measure$ be a Borel measure on $X$. Let $p \in \intco{1, \infty}$. We define a space
    \begin{equation*}
        \tilde{\mathcal{N}}_{\hat{TC}}^{1,p}(X)
        \coloneqq 
       \set{
            f: X \rightarrow \mathbb{R}
            \given  f\, \text{is Borel}, 
            \int_X |f|^p \, \mathrm{d} \measure< \infty \,\text{and} \, L^p(\measure) \cap \generalisedPWeakUpperSGradient f \neq \emptyset 
            }
        .
    \end{equation*}
    On space $\tilde{\mathcal{N}}_{\hat{TC}}^{1,p}(X)$ we define a seminorm
    \begin{equation*}
        \forall 
            f \in \tilde{\mathcal{N}}_{\hat{TC}}^{1,p}(X)
        \qquad 
            \norm{ f }_{\mathcal{N}_{\hat{TC}}^{1,p}(X)} 
            \coloneqq 
            \norm{ f }_{L^p(\measure)}
            + 
            \inf_{ g }
                \norm{ g }_{ L^p( \measure) },
    \end{equation*}
    where the infimum is taken over all generalized $p$-weak
    upper $S$-gradients $g$ of $f$. Also, we define an equivalence relation $\sim$ by
    \begin{equation*}
        \forall 
            f, f'\in \tilde{\mathcal{N}}_{\hat{TC}}^{1,p}(X)
        \qquad 
            f \sim f'
            \iff
            \norm{ f - f' } _{\mathcal{N}_{\hat{TC}}^{1,p}(X)}
            =
            0.
    \end{equation*}
    Finally, we define the $\hat{TC}$-Newtonian space, as the quotient space  $\mathcal{N}_{\hat{TC}}^{1,p}(X) \coloneqq \tilde{\mathcal{N}}_{\hat{TC}}^{1,p}(X) / {\sim}$.
    On this space $\norm{ \cdot } _{\mathcal{N}_{\hat{TC}}^{1,p}(X)}$ is a norm.
\end{definition}
\begin{theorem}
    \label{lem:comparison_between_various_upper_gradients}
    Let $(X, \metricAlone)$ be a metric space, $\measure$ be a Borel measure on $X$ and let $f \colon X \to \bRExtended$ and  $g \colon X \to \intcc{0,\infty}$
    be Borel maps. Then the following statements are equivalent
    \begin{enumerate}[label=(\roman*)]
        \item $g$ is a generalized $p$-weak upper $S$-gradient of $f$,
        \item $g$ is a $p$-weak upper $S$-gradient of $f$.
    \end{enumerate}
\end{theorem}
\begin{proof}
    $(i) \Rightarrow (ii)$
       If $g \colon X \to \intcc{0,\infty}$ is a generalized $p$-weak upper $S$-gradient of $f$, then there exists $D \subset \testMeasuresNonTrivial$ such that 
    $\tilde{\modulus}^p (\testMeasuresNonTrivial \setminus D)=0$ and for $\mu \in D$ we have
    \begin{equation*}
        \integral{
                \testCurves\del{ [0,1]; X}
            }{
                \abs{ f \circ \gamma(1) - f \circ \gamma(0) }
            }{
                \mu(\gamma)
            }  
            \le  \sint{ \mu}{ g }.
    \end{equation*}

    Let\footnote{By Theorem  \ref{lem::hyper_envaluations_are_Borel} $\Gamma$ is a Borel set.}
    \begin{equation*}
        \Gamma 
        \coloneqq 
        \set{
            \gamma \in \testCurvesNonTrivial
            \given 
            \abs{ f \circ \gamma(1)
            - f\circ \gamma(0) }
            > 
            \sint{\gamma}{g}
        }
        \quad \text{ and } \quad 
        \wave{\Gamma} 
        \coloneqq 
        \set{
            \delta_\gamma 
            \given 
            \gamma \in \Gamma 
        }.   
    \end{equation*}
    Let us observe that $D\cap \tilde{\Gamma}=\emptyset$. Indeed, suppose there exists $\gamma \in \Gamma$ such that $\delta_{\gamma} \in D\cap \tilde{\Gamma}$, then we have
    \begin{equation*}
        \abs{ f \circ \gamma(1)
            - f\circ \gamma(0) }
        =        
                \integral{
                    \testCurves\del{ [0,1]; X}
                }{\abs{
                     f \circ \gamma'(1) - f \circ \gamma'(0) 
                }}{
                    \delta_{\gamma}(\gamma')
                }            
        \le 
        \integral{
                \testCurves\del{ [0,1]; X}
            }{
                \quad \sint{\gamma'}{g}
            }{
                \delta_{\gamma}(\gamma')
            }
        =
        \sint{\gamma}{g}.
    \end{equation*}
    This leads us to the contradiction with $\gamma \in \Gamma$.
    
    Therefore, 
    \[
        \tilde{\modulus}^p (\tilde{\Gamma}) \leq \tilde{\modulus}^p ((\testMeasuresNonTrivial \setminus D)\cap \tilde{\Gamma})+\tilde{\modulus}^p ( D\cap \tilde{\Gamma}) =0.
    \]
    Hence, by Lemma \ref{lem:generalized_p-modulus_agrees_with_p-modulus} 
    we have $\modulus^p(\Gamma) = 0$
    and    $g$ is a $p$-weak upper $S$-gradient of $f$.

       $(ii) \Rightarrow (i)$   Now, let $g$ be a $p$-weak upper $S$-gradient of $f$.
    Then, keeping $\Gamma$ as it was introduced in the previous implication, we have $\modulus^p(\Gamma) = 0$. 
    Then, by Lemma \ref{lem:absolute_continuity_for_two_moduli}
    we have $\tilde{\modulus}^p\del{ \hat{\Gamma}} = 0$, where
    \begin{equation*}
        \hat{\Gamma}
        =
        \set{
            \mu \in \testMeasuresNonTrivial
            \given 
            \mu\del{ \Gamma } > 0
        }.
    \end{equation*}
    Hence, for $\mu \in \testMeasuresNonTrivial \setminus \hat{\Gamma}$ we have 
    
    \begin{equation*}
        \integral{
                \testCurves\del{ [0,1]; X}
            }{
                \abs{ f \circ \gamma(1) - f \circ \gamma(0) }
            }{
                \mu(\gamma)
            }
        =  \integral{
                \testCurvesNonTrivial \setminus \Gamma
            }{
                \abs{ f \circ \gamma(1) - f \circ \gamma(0) }
            }{
                \mu(\gamma)
            }                        \le 
            \integral{
                \testCurves\del{ [0,1]; X}
            }{
                \quad \sint{\gamma}{g}  
            }{
                \mu(\gamma)
            }.
    \end{equation*}
    This shows that $g$ is a generalized $p$-weak upper $S$-gradient
    of $f$.
\end{proof}
\begin{corollary} \label{cor::TC_hat_Newtonian_0_ae}
    Let $(X, \metricAlone)$ be a metric space and let $\measure$ be a Borel measure on $X$. If $f, f' \in \tilde{\mathcal{N}}_{\hat{TC}}^{1,p}(X)$
    are such that $f = f'$ $\measure$-almost everywhere, 
    then $\norm{ f- f'}_{\mathcal{N}_{\hat{TC}}^{1,p}(X)} = 0$.

\end{corollary}
\begin{proof}
        It is sufficient to show that if $f = 0$ $\measure$-almost everywhere, 
    then $\norm{ f }_{\mathcal{N}_{\hat{TC}}^{1,p}(X)} = 0$.
    Since $f \in \tilde{\mathcal{N}}_{\hat{TC}}^{1,p}(X)$, 
    there is a $g \in L^p(\measure)$ that is a generalized $p$-weak upper $S$-gradient of $f.$
    By Theorem \ref{lem:comparison_between_various_upper_gradients}
    it is also a $p$-weak upper $S$-gradient of $f$, 
    so $f \in \newtonianTCPre^{1,p}(X)$. 
    As $f$ is Borel, by Proposition \ref{prop::when_ae_0_gives_0_in_Newtonian}
    the zero function is a $p$-weak upper $S$-gradient of $f$,
    hence by Theorem \ref{lem:comparison_between_various_upper_gradients}
    the zero function is a generalized $p$-weak upper $S$-gradient of $f.$
    Since $f = 0$ $\measure$-almost everywhere, we have
    $\norm{ f }_{\mathcal{N}_{\hat{TC}}^{1,p}(X)} = 0$.
\end{proof}

\subsection{Ambrosio-Gigli-Savaré-like Space} \label{subsec::gigli-like_space}

Within this subsection we introduce the last of the considered function spaces --- a modification of the spaces considered by Ambrosio, Gigli and Savaré \cite{AGS1, AGS,  gigli}. 
This time we replace the family of absolutely continuous curves with $\testCurves\del{[0,1];X}$ and we again replace the integral along the curve with the symmetrized integral $\sint{\gamma}$.

In the next definition we define the $\measure$-admissibilty of Borel measures defined on the space $\testCurves\del{[0,1];X}$ which mimics the notion of test plans present in the papers of Ambrosio, Gigli and Savaré. 
Let us note that in the mentioned papers the topology on the space of curves is induced by the supremum metric, hence the evaluations maps $e_t \colon C\del{[0,1];X} \to X$, $e_t(\gamma) = \gamma(t)$ for $t \in [0,1]$, are Borel. In our setting, evaluation maps $e_t \colon \testCurves\del{[0,1];X} \to X$ are Borel by Proposition \ref{ewaluacje}.

\begin{definition}[$\measure$-admissible measure]
    Let $(X, \metricAlone)$ be a metric space and $\measure$ be a Borel measure on $X$.
    We will say that a Borel probability measure 
    $
        \mu 
    $
    on $\testCurves\del{ [0,1];X }$
    is $\measure$-admissible if
    \begin{itemize}
        \item 
        There exists a constant $C \ge 0$ such that
        for all $t \in [0,1]$ {and Borel $B \subseteq X$ we have
        $
            (e_t)_\# \mu(B) \le C \measure(B)
        $,} 
         where $e_t$ is
        the evaluation map 
        $
            e_t 
            \colon 
            \testCurves \del{ [0,1];X }
            \to 
            X
        $, $ e_t (\gamma) :=\gamma(t)$,        
        \item 
        $
            \integral{
                \testCurves\del{
                    [0,1]; 
                    X
                }
            }{
                 V(\gamma)
            }{
                \mu( \gamma)
            }
            < 
            \infty
        $.
    \end{itemize}
    We will denote the space of $\measure$-admissible probability measures by
    $
       \mathcal{P}^{(\measure)}(X )
    $ and  $\testMeasuresNonTrivial^{(\measure)}:= \testMeasuresNonTrivial \cap \mathcal{P}^{(\measure)}(X) $.
\end{definition}
\begin{definition}[$\measure$-upper S-gradient]
    Let $(X, \metricAlone)$ be a metric space and $\measure$ be a Borel measure on $X$.
    We will say that a Borel function $g \colon X \to [0,\infty]$ is an $\measure$-upper S-gradient of a Borel measurable
    function $f \colon X \to \bRExtended$  if 
    \begin{equation*}
        \forall 
        \mu \in \testMeasuresNonTrivial^{(\measure)} 
        \qquad 
        \integral{
            \testCurves\del{ [0,1]; X}
        }{
            \abs{
                f(\gamma(1)) - f(\gamma(0))
            }
        }{
            \mu(\gamma)
        }
        \le  \sint{ \mu }{ g }.
    \end{equation*}
\end{definition}
We will denote the family of all $\measure$-upper S-gradients of $f$ by $\mUpperSGradient f$.
\begin{remark}\label{r}
    Every upper S-gradient is an $\measure$-upper S-gradient.
\end{remark}

\begin{definition}[Ambrosio-Gigli-Savaré-like Space]
    Let $(X, \metricAlone)$ be a metric space and $\measure$ be a Borel measure on $X$. Let $p \in \intco{1, \infty}$. We define a space
    \begin{equation*}
        \tilde{G}^{1,p}(X)
        \coloneqq 
                \set{
            f: X \rightarrow \mathbb{R}
            \given  f\, \text{is Borel}, 
            \int_X |f|^p \mathrm{d} \measure< \infty \,\text{and} \, L^p(\measure) \cap \mUpperSGradient f \neq \emptyset 
            }.
    \end{equation*}
    On space $\tilde{G}^{1,p}(X)$ we define a seminorm
    \begin{equation*}
        \forall 
            f \in \tilde{G}^{1,p}(X)
        \qquad 
            \norm{ f }_{G^{1,p}(X)} 
            \coloneqq 
            \norm{ f }_{L^p(\measure)}
            + 
            \inf_{ g }
                \norm{ g }_{ L^p( \measure) },
    \end{equation*}
    where the infimum is taken over all $\measure$-upper $S$-gradients $g$ of $f$. Also, we define an equivalence relation $\sim$ by
    \begin{equation*}
        \forall 
            f, f'\in \tilde{G}^{1,p}(X)
        \qquad 
            f \sim f'
            \iff
            \norm{ f - f' } _{G^{1,p}(X)}
            =
            0.
    \end{equation*}
    Finally, we define the Ambrosio-Gigli-Savaré-like space, as the quotient space  $G^{1,p}(X) \coloneqq \tilde{G}^{1,p}(X) / {\sim}$.
    On this space $\norm{ \cdot } _{G^{1,p}(X)}$ is a norm.
\end{definition}
\begin{proposition}\label{prop::f=0_ae_and_in_Gigli_means_it_is_0_in_Gigli}

    Let $(X, \metricAlone)$ be a metric space and $\measure$ be a Borel measure on $X$.
    If $f \in \tilde{G}^{1,p}(X)$ is such that $f = 0$ $\measure$-almost everywhere, then $\norm{ f }_{G^{1,p}(X)} = 0$.
\end{proposition}    
\begin{proof}
        Let 
    $
        E 
        \coloneqq 
        \set{
            x \in X
            \given 
            f(x) \ne 0
        }
    $.
    Since $f \in \tilde{G}^{1,p}(X)$, $f$ is Borel, hence $E$ is a Borel set. 
    Furthermore, as $f = 0$ $\measure$-almost everywhere, $\measure(E) = 0$.
    We will show that $0$ is an $\measure$-upper S-gradient of $f$.
    
    Fix $\mu \in \testMeasuresNonTrivial^{(\measure)}$, there exists $C > 0$ such that 
    $
        (e_t)_\# \mu \le C \measure
    $
    for all $t \in [0,1]$.
    We will show that set
    \begin{equation*}
        E'
        \coloneqq 
        \set{
            \gamma \in \testCurves\del{ [0,1]; X}
            \given 
            f(\gamma(0) ) \ne 0 
            \text{ or } 
            f(\gamma(1) ) \ne 0 
        }
    \end{equation*}
    satisfies $\mu(E') = 0$. 
    Note that $E'$ is Borel as by Theorem \ref{lem::hyper_envaluations_are_Borel} $ \gamma \mapsto f(\gamma(t))$ is a Borel function for a fixed $t \in [0,1]$,
    hence $\mu(E')$ is well-defined.
    We have
    \begin{align*}
        \mu(E')
        &=
        \mu\del{
            \set{
                \gamma \in \testCurves\del{ [0,1]; X}
                \given 
                f(\gamma(0) ) \ne 0 
                \text{ or } 
                f(\gamma(1) ) \ne 0 
            }
        }
        \\
        &=
        \mu\del{
            e_0^{-1}\sbr{E}
            \cup e_1^{-1}\sbr{E}
        }
        \\
        &\le 
        \mu\del{
            e_0^{-1}\sbr{E}}
        + \mu\del{
            e_1^{-1}\sbr{E}}
        \\
        &=
        (e_0)_\#\mu(E) + (e_1)_\#\mu(E)
        \\
        &\le 
        2C \measure(E)
        \\
        &=0.
    \end{align*}    
    Therefore, 
    $
        f(\gamma(0)) = 0
    $
    and 
    $
        f(\gamma(1)) = 0
    $
    for $\mu$-almost every $\gamma \in \testCurves\del{ [0,1]; X}$, hence
    $
        \abs{f(\gamma(0)) - f(\gamma(1))} = 0
    $
    for $\mu$-almost every $\gamma \in \testCurvesNonTrivial$.
    In consequence,
    \begin{equation*}
        \int_{\testCurves\del{ [0,1]; X}}
            \abs{f(\gamma(0)) - f(\gamma(1))} 
        \ \mathrm{d}\mu(\gamma)
        =
        \int_{\testCurves\del{ [0,1]; X}}
            0
        \ \mathrm{d}\mu(\gamma)
        =
        \int_{\testCurves\del{ [0,1]; X}}
            \sint{\gamma} 0
        \ \mathrm{d}\mu(\gamma).
    \end{equation*}
    Hence, $0$ is an $\measure$-upper S-gradient of $f$.
    This, combined with the fact that $f = 0$ $\measure$-almost everywhere, gives us
    $\norm{ f }_{G^{1,p}(X)} = 0$.

\end{proof}

\section{Hajłasz--Sobolev spaces vs TC-Newtonian spaces} \label{sec::hajlasz_and_TC_Newtonian}
We devote this section to the comparison of the Hajłasz--Sobolev and the $\testCurves$-Newtonian spaces. 
The main result of this section, Theorem \ref{hn}, shows that  the $\testCurves$-Newtonian spaces are much more similar to the Hajłasz--Sobolev spaces than the usual Newtonian spaces.
Indeed, the theorems showing the equivalence between the Newtonian space $N^{1,p}$ and the Hajłasz--Sobolev space $M^{1,p}$ require that the measure $\mu$ on the space is doubling and supports some Poincar\'{e} inequality (for example, see \cite[Theorem 4.9]{newtonian} and \cite[Theorem 5.1]{comparisonNewtonianHajlasz}).

Let us note that the proof of the $(b)$ part of Theorem \ref{hn} is inspired by the proofs of \cite[Lemma 4.7]{newtonian} and \cite[Lemma 1.3]{hajlaszGradientsAreUpper}.
However, since we allow curves to be discontinuous, the proof becomes much more technical.

\begin{definition}[Hajłasz Gradient]
    Let $(X, \metricAlone)$ be a metric space and $\measure$ be a measure on $X$.
    Let $f \colon X \to \bRExtended$ be measurable and finite $\measure$- almost everywhere. 
    We will say that a measurable 
    function $g \colon X \to \intcc{0, \infty}$
    is a Hajłasz gradient of $f$, if there exists a measurable set $E \subset X$ of measure $0$ such that for every $x, y \in X \setminus E$ we have
    \begin{equation*}
                \abs{
            f(y) - f(x)
        }
        \le 
        \del{ g(x) + g(y) }
        \metric\del{
            x,y
        }.
    \end{equation*}
    We will denote the family of Hajłasz gradients of $f$
    by $\hajlaszGradient f$.
\end{definition}
\begin{definition}[Hajłasz-Sobolev space]
    Let $(X, \metricAlone)$ be a metric space and $\measure$ be a measure on $X$. 
    Let $p \in \intcc{1, \infty}$.
    We define the Hajłasz-Sobolev space $M^{1,p}(X)$ as the space
    \begin{equation*}
        M^{1,p}(X)
        \coloneqq 
        \set{
            f \in L^p(\measure)
            \given 
            \exists g \in L^p(\measure)
            \quad 
            g \in \hajlaszGradient f
        }
    \end{equation*}
    endowed with the norm
    \begin{equation*}
        \norm{ f }_{M^{1,p}(X)}
        \coloneqq 
        \norm{ f }_{L^p(\measure)}
        +
        \inf_{g \in \hajlaszGradient f}
        \norm{ g }_{L^p(\measure)}.
    \end{equation*}
\end{definition}
The following theorem is the main result of this section.
\begin{theorem} \label{hn}
    Let $(X, \metricAlone)$ be a metric space, $\measure$ be a Borel measure on $X$, and $p \in [1, \infty)$. Then, for any measurable functions $f \colon X \to \bRExtended$, $g \colon X \to \intcc{0,\infty}$ such that $f$ and $g$ are finite $\measure$-almost everywhere we have:
    \begin{enumerate}[label=(\alph*)]
        \item 
        If $g$ is a p-weak upper S-gradient of $f$, then $g/2$ is a Hajłasz gradient of $f$,
        \item 
        If $\measure$ is $\sigma$-finite and Borel regular, $g$ is a Hajłasz gradient of $f$, then there exist Borel functions $\tilde{f}:X \rightarrow \mathbb{R}$ and $\tilde{g}: X \rightarrow [0,\infty]$, equal $\measure$-a.e. to $f,g$, respectively, such that $76\tilde{g}$ is an upper S-gradient of $\tilde{f}$.
    \end{enumerate}
\end{theorem}
    Having in mind Proposition \ref{prop::when_ae_0_gives_0_in_Newtonian}, the folowing result is a corollary to Theorem \ref{hn}.
\begin{theorem} \label{ro}
    Let $(X, \metricAlone)$ be a metric space, $\measure$ be $\sigma$-finite and Borel regular measure, and $p \in [1, \infty)$. Then, 
    $$
        M^{1,p}(X)\cong \newtonianTC^{1,p}(X).
    $$
    \end{theorem}
The proof of Theorem \ref{hn} will follow from Theorem \ref{uno} and Theorem \ref{dos}.

\begin{theorem}\label{uno}
    Let $(X, \metricAlone)$ be a metric space and $\measure$ be a Borel measure on $X$.
    Let $f \colon X \to \bRExtended$ be finite $\measure$-almost everywhere.
    Then if $g \colon X \to \intcc{0, \infty}$
    is a $p$-weak upper S-gradient of $f$, then $g/2$ is a Hajłasz gradient of $f$.
\end{theorem}
\begin{proof}
    Let
    \begin{equation*}
        \Gamma 
        \coloneqq 
        \set{
            \gamma \in \testCurvesNonTrivial
            \given 
            \abs{
                f \circ \gamma(1)
                - f \circ \gamma(0)
            }
            >
            \sint{\gamma}{g}
        }.
    \end{equation*}
    Since $g$ is a $p$-weak upper S-gradient of $f$, then
   we have $\modulus^p\del{ \Gamma } = 0$.
    Therefore, there exists a sequence 
    $\rho_n \colon X \to \intcc{0, \infty}$
    of Borel functions such that 
    $
        \norm{ \rho_n}_{L^p(\measure)}^p \to 0
    $
    as $n \to \infty$ 
    and
    $\sint{\gamma}{\rho_n} \ge 1$
    for all $n \in \bN$ and all $\gamma \in \Gamma$. Hence, there exists a subsequence $\del{ \rho_{n_k}}_k$
    such that $\rho_{n_k} \to 0$ $\measure$-almost everywhere
    as $k \to \infty$. Let
    $E \subseteq X$ be such that $\measure \del{ X \setminus E } = 0$ and for $x \in E$ we have $\rho_{n_k}(x) \to 0$
    as $k \to \infty$.
    
    Let
    \begin{equation*}
        \Gamma_E
        \coloneqq 
        \set{
            \gamma \in \testCurvesNonTrivial
            \given 
            \exists x, y \in E
            \qquad
            \gamma
            =
            x \indicator{ \intco{0, 1/2} }
            + y \indicator{ \intcc{1/2, 1 } }
        }.
    \end{equation*}
    We observe that $\Gamma_E \subseteq \testCurvesNonTrivial \setminus \Gamma$. Indeed, if $\gamma \in \Gamma_E$, then, since $\gamma(1), \gamma(0) \in E$, we have
    \begin{equation*}
        \sint{ \gamma }{ \rho_{n_k} }
        =
        \frac{ \rho_{n_k}( \gamma(1)) + \rho_{n_k}( \gamma( 0 ) ) }{2}
        \metric\del{
            \gamma(1), \gamma(0)
        }
        \xrightarrow{ k \to \infty }
        0.
    \end{equation*}
    In particular, there exists $k \in \bN$ such that
    $
        \sint{ \gamma }{ \rho_{n_k} } < 1
    $
    and hence $\gamma \notin \Gamma$.
    
    Let $x, y \in E$ be such that $x \ne y$.
    Then we define $\gamma \in \Gamma_E$ as follows 
    $\gamma
            =
            x \indicator{ \intco{0, 1/2 } }
            + y \indicator{ \intcc{1/2, 1 } }
    $.
    For such $\gamma$ we have
    \begin{equation*}
        \abs{ f(y) - f(x) }
        =
        \abs{
                f \circ \gamma(1)
                - f \circ \gamma(0)
            }
            \le 
            \sint{\gamma}{g}
        =
        \frac{ g(\gamma(1)) + g(\gamma(0)) }{2}
        \metric\del{
            \gamma(1), \gamma(0)
        }
        =
        \frac{ g(y) + g(x) }{2}
        \metric\del{ y, x}.
    \end{equation*}
    Since $\measure\del{ X \setminus E }= 0$,
    this shows that $g/2$ is a Hajłasz gradient of $f$.
\end{proof}
\begin{theorem}\label{dos}
\label{prop:almost_everywhere_finite_Hajlasz_gradients_are_upper_gradietnts}
    Let $(X, \metricAlone)$ be a metric space and $\measure$
    be $\sigma$-finite Borel regular measure on $X$.
    Let $f \colon X \to \bRExtended$ and $g \colon X \to [0,\infty]$ be measurable and
    finite $\measure$-almost everywhere. 
    If $g$ is a Hajłasz gradient of $f$, 
    then there exist Borel functions $\tilde{f}\colon X \rightarrow \mathbb{R}$ and $\tilde{g}\colon  X \rightarrow [0,\infty]$, equal $\measure$-a.e. to $f,g$, respectively, such that $76\tilde{g}$ is an upper S-gradient of $\tilde{f}$.
\end{theorem}
\begin{proof}
The proof will be divided into some lemmata and steps.
\begin{lemma}
\label{lem:there_are_representatives_that_satisfy_Hajlasz_condition_everywhere}
    Let $(X, \metricAlone)$ be a metric space and $\measure$
    be a Borel regular and $\sigma$-finite measure on $X$.
    Let $f \colon X \to \bRExtended$ be measurable, finite $\measure$-almost everywhere and 
    $g \colon X \to \intcc{ 0, \infty}$ be a Hajłasz gradient of $f$.
    Then there exist Borel functions $f': X \rightarrow \mathbb{R}$
    and $g' : X \rightarrow [0, \infty]$ such that $f=f'$ and $g=g'$ $\measure$-almost everywhere
    such that 
    \begin{equation*}
        \forall x, y \in X
        \qquad 
        \abs{
            f'(y) - f'(x)
        }
        \le 
        \del{ g'(x) + g'(y) }
        \metric\del{
            x,y
        }.
    \end{equation*}
\end{lemma}
\begin{proof}
    Since $\measure$ is Borel regular and $\sigma$-finite, 
    there are Borel functions $\wave{f}$ and $\wave{g}$ (with $\wave{g}$ being non-negative), such that $f = \wave{f}$ and $g = \wave{g}$ $\measure$-almost everywhere.
    Since $\measure$ is Borel regular, there exists Borel set $E \subseteq X$ such that $\measure\del{ E} = 0$, $\wave{f}$ is finite on $X \setminus E$  and 
    \begin{equation*}
        \forall x, y \in X \setminus E
        \qquad 
        \abs{
            \wave{f}(y) - \wave{f}(x)
        }
        \le 
        \del{ \wave{g}(x) + \wave{g}(y) }
        \metric\del{
            y, x
        }.
    \end{equation*}

    Now, let us define the following Borel maps:
    $
        f' \coloneqq \wave{f} \indicator{X \setminus E} 
    $
    and 
    $
        g' \coloneqq \wave{g} \indicator{X \setminus E} + \infty \indicator{E}
    $.
    Then we have $f' = f$ and $g' = g$ $\measure$-almost everywhere.
    Moreover, by a straightforward checking we have 
    $$
        \abs{ f'(y) - f'(x) }
        \le 
        \del{ g'(y) + g'(x) } \metric\del{ x, y}
    $$
    for all $x, y \in X$. 
\end{proof}

\begin{lemma}
\label{lem::bounding_difference_in_f_by_integral_with_endpoints}
    Let $(X, \metricAlone)$ be a metric space.
    Let $f \colon X \to \bR$ be measurable and  $g\colon X \to \intcc{0, \infty}$ be Borel and such that 
    \begin{equation*}
        \forall x, y \in X
        \qquad 
        \abs{
            f(x) - f(y)
        }
        \le 
        \del{ g(x) + g(y) } \metric\del{ x, y}.
    \end{equation*}
    Let $\gamma \in \testCurves\del{ [0,1]; X}$. 
    Suppose that for some $a, b \in [0,1]$, $a < b$ there exists 
    $M > 0$ such that \footnote{ See Definition \ref{def:left-jump_right-jump}.}
    \begin{equation*}
        \forall s \in \intoo{a,b}
        \qquad 
        \phi_\gamma(s)
        \le 
        M.
    \end{equation*}
       Then
    \begin{equation*}
        \abs{
            f(\gamma(a)) - f(\gamma(b^-))
        }
        \le 
        8 M g(\gamma(a)) 
        + 16\sint{\gamma \rvert_{\intcc[0]{a, b^-}}}{g}
        + 8 M g(\gamma(b^-)).
    \end{equation*}
\end{lemma}
\begin{proof}
    First of all let us observe that if  
    $
        V_\gamma(b^-) - V_\gamma(a) \le 4M,
    $
    then the claim follows immediately.
    Indeed, 
    \begin{multline*}
        \abs{
            f(\gamma(a)) - f(\gamma(b^-))
        }
        \le 
        \del{
            g(\gamma(a)) + g(\gamma(b^-))
        }
        \metric\del[0]{ \gamma(a),\gamma(b^-)}
        \\
        \le 
        \del[1]{
            g(\gamma(a)) + g(\gamma(b^-))
        }
        \del{
            V_\gamma(b^-) - V_\gamma(a)
        }
        \le 
        4Mg(\gamma(a))
        +4Mg(\gamma(b^-)).
    \end{multline*}
    In the rest of the proof we assume that 
    $
        V_\gamma(b^-) - V_\gamma(a) > 4M.
    $
    
    \textbf{Step 1}
    We will iteratively construct a tuple
    $
        \del{ t_n }_{n=0}^m
    $
    of elements of $[a,b]$, where $m\geq 2$ with the following properties:
    \begin{itemize}
        \item 
        $ t_0 = a $ and $t_m = b$,
        \item 
        $
            \del{ t_n }_{n=0}^m
        $   
        is strictly increasing,
        \item 
        For all $i \in [m-1]$ $t_i$ 
        is a point of continuity of $\gamma$,
        \item 
        For all $i \in [m-1]$ we have
        $
            V_\gamma(t_i) - V_\gamma(t_{i-1} )
            \in 
            \intcc{M, 3M}
        $
        and 
        $
            V_\gamma(t_m^-) - V_\gamma(t_{m-1} )
            \in 
            \intcc{M, 4M}.
        $
    \end{itemize}

    Let $t_0 \coloneqq a$. 
    Suppose that we have constructed $t_i$ for some $i \in \bN_0$. 
    If $V_\gamma(b^-) - V_\gamma(t_i) \le 4M$, 
    then we define $m\coloneqq i+1$ and $t_{m} \coloneqq b$. 
    Otherwise, we have $V_\gamma(b^-) - V_\gamma(t_i) > 4M$.
    We will show that in this case
    there exists $t \in \intoo{t_i, b}$ such that
    $\gamma$ is continuous at $t$ and 
    $
        V_\gamma(t^-) - V_\gamma(t_i) 
        \in 
        \intcc{ M , 3M}.
    $
    We will then define $t_{i+1} \coloneqq t$.

    To prove the existence of such a point $t$, 
    we will first consider a point 
    \begin{equation*}
        s \coloneqq \min S,
        \text{ where }
        S \coloneqq 
        \set{
            \tau \in \intcc{t_i, b}
            \given 
            V_\gamma(\tau) - V_\gamma( t_i ) \ge M
        }.
    \end{equation*} 
    Clearly, $S \ne \emptyset$ as $b \in S$. 
    The right-continuity of $V_\gamma$ implies that
    $\inf S \in S$, 
    hence $s \coloneqq \min S$ is well-defined.
    Since $t_i \notin S$ and 
    $
        V_\gamma(b^-) - V_\gamma( t_i ) > 4M,
    $
    we have $s \in \intoo{ t_i, b}$.
    By the definition of $s$ and $S$ we have 
    $
        V_\gamma(s^-) - V_\gamma(t_i) \le M.
    $   
    By the assumptions of the current lemma and by Proposition \ref{rem:properties_of_variation_function} we have
    $
        V_\gamma(s) - V_\gamma(s^-) = \phi_\gamma(s) \le M,
    $
    which implies that 
    \begin{equation*}
        V_\gamma(s) - V_\gamma(t_i)
        =
        \del{
            V_\gamma(s) - V_\gamma(s^-)
        }
        + \del{
            V_\gamma(s^-) - V_\gamma(t_i)
        }
        \le 
        M + M
        = 2M.
    \end{equation*}
    This, combined with the definition of $s$ gives us
    $
        V_\gamma(s) - V_\gamma(t_i) \in \intcc{M, 2M}.
    $

    Now, since $V_\gamma$ is right-continuous at $s$,
    there is $\delta \in (0, b-s)$ such that 
    if $\tau \in \intcc{s, s+ \delta}$,
    then $0 \le V_\gamma(\tau) - V_\gamma(s) \le M$. 
    Since $\gamma$ can have at most countable points of discontinuity, there is $t \in \intcc{s, s+ \delta}$
    such that $\gamma$ is continuous at $t$. 
    Note that for such a $t$ we have
    \begin{equation*}
        V_\gamma(t) - V_\gamma(t_i)
        = 
        \del{V_\gamma(t) - V_\gamma(s) }
        + \del{ V_\gamma(s) - V_\gamma(t_i) }
        \in 
        \sbr{M, 3M}.
    \end{equation*}
    Hence, $t \in (t_i, b)$ with the desired properties exists, and, 
    as previously mentioned, we define $t_{i+1} \coloneqq t$.

    We will now show that the tuple $(t_i)_{i=0}^m$
    has all the desired properties.
    Firstly, we have to show that the construction must end, that is, 
    that for some $i \in \bN$ we have 
    $
        V_\gamma(b^-) - V_\gamma(t_i) \le 4M.
    $
    Let us suppose that for all $i \in \bN$ we have
    $
        V_\gamma(b^-) - V_\gamma(t_i) > 4M, 
    $
    then
    \begin{equation*}
        V_{\gamma} \geq V_\gamma(b^-) - V_\gamma(a)
        =
        V_\gamma(b^-) - V_\gamma(t_i)
        +
        \sum_{j=1}^i (
            V_\gamma(t_j) - V_\gamma(t_{j-1}))
        \ge 
        V_\gamma(b^-) - V_\gamma(t_i) + Mi 
        > 
        M(i+4).
    \end{equation*} 
    Since the right-hand side diverges to $\infty$ as $i \to \infty$, we get a contradiction. 
    Hence, there exists $i \in \bN$ such that 
    $
        V_\gamma(b^-) - V_\gamma(t_i) \le 4M.
    $

    The only desired property of the tuple $(t_i)_{i=0}^m$ that does not follow
    directly from the construction is that 
    $
        V_\gamma(t_m^-) - V_\gamma(t_{m-1} ) \ge M.
    $
       Recall that we assume 
    $
        V_\gamma(b^-) - V_\gamma(a) > 4M,
    $
    which implies that $m \ge 2$. 
    By the definition of $m$ we have 
    $
        V_\gamma(b^-) - V_\gamma(t_{m-2}) > 4M,
    $
    hence 
    \begin{equation*}
        V_\gamma(t_m^-) - V_\gamma(t_{m-1} )
        =
        \del{ V_\gamma(b^-) - V_\gamma(t_{m-2}) }
        - \del{ V_\gamma(t_{m-1})- V_\gamma(t_{m-2}) }
        \ge 
        4M - 3M
        = 
        M,
    \end{equation*}
    as needed.
    
    \textbf{Step 2}
    For $i \in [m]$ denote $A_i \coloneqq \sbr[0]{ t_{i-1}, t_i^- }$ and $A_0 \coloneqq A_1$, $A_{m+1} \coloneqq A_m$.
    We will show the existence of a tuple $(x_i)_{i=0}^{m+1}$ of points in $X$
    that satisfy the following properties:
    \begin{itemize}
        \item 
        $x_0 = \gamma(a)$, $x_{m+1} = \gamma(b^-)$,
        \item 
        For all $i \in \set{0, \ldots, m+1}$ we have  
        $
            x_i \in \closure{ \image\del{\gamma \rvert_{A_i}}},
        $
        \item 
        For all $i \in [m]$ we have
        $
            g(x_i)
            \le 
            \frac{1}{M}
            \sint{ 
                \gamma \rvert_{ A_{i} }
            }{g},
        $
        \item 
        For all $i \in [m+1]$ we have 
        $
            \metric\del{
                x_i, x_{i-1}
            }
            \le 
            8M.
        $
    \end{itemize}
    First of all let us observe that for all $i \in \set{0, \ldots, m+1}$ we have
    $
        \mu_{\gamma\rvert_{ A_{i} }}^S\del{
            \closure{ \image\del{\gamma \rvert_{A_i}}}
        }
        \ge 
        M.
    $
       Indeed, by Remark \ref{remM}, Remark \ref{rem} and Corollary \ref{cor:variation_is_additive}, for $i \in [m]$ we have
    \begin{eqnarray*}
        \mu_{\gamma\rvert_{ A_{i} }}^S\del{
            \closure{ \image\del{\gamma \rvert_{A_i}}}
        } = \mu_{\gamma\rvert_{ A_{i} }}^S\del{X} = V(\gamma\rvert_{ A_{i} }) = V_{\gamma\rvert_{ [t_{i-1},1] } }(t_i^-) =
        V_\gamma(t_i^-) - V_\gamma(t_{i-1}) 
        \ge M.
    \end{eqnarray*}
    Moreover, for $i \in \set{0, m+1}$ the claim follows from the fact that
    $A_0 = A_1$ and $A_m = A_{m+1}$.
    We conclude that\footnote{$\sintAveraged{ 
                \gamma \rvert_{ A_{n,i} }
            }{g} = \frac{1}{ \mu_{\gamma\rvert_{ A_{i} }}^S\del{
            X}} \int_X g d \mu_{\gamma\rvert_{ A_{i} }}^S = \frac{1}{ \mu_{\gamma\rvert_{ A_{i} }}^S\del{
            \closure{ \image\del{\gamma \rvert_{A_i}}}}} \int_{\closure{ \image\del{\gamma \rvert_{A_i}}}} g d \mu_{\gamma\rvert_{ A_{i} }}^S$.} 
    $
        \sintAveraged{ 
                \gamma \rvert_{ A_{n,i} }
            }{g}
    $
    is well-defined for all $i \in \set{0, \ldots, m+1}$.

    Let us now construct the tuple $(x_i)_{i=0}^{m+1}$. 
    We define 
    $x_0 \coloneqq \gamma(a)$ 
    and 
    $x_{m+1} \coloneqq \gamma(b^-)$.
    Next, we will show that there exists 
    $
        x \in \closure{ \image\del{\gamma \rvert_{A_i}}}
    $
    such that 
    $
        g(x)
        \le 
        \sintAveraged{ 
            \gamma \rvert_{ A_{n,i} }
        }{g}.
    $
    Indeed, suppose such $x$ does not exists. We can assume that $\sintAveraged{\gamma \rvert_{ A_{n,i} }}{g} < \infty$. 
    Then for all 
    $
        x \in \closure{ \image\del{\gamma \rvert_{A_i}}}
    $
    we have
    $
        g(x)
        >
        \sintAveraged{ 
            \gamma \rvert_{ A_{i} }
        }{g},
    $
    hence 
    \begin{equation*}
        0
        <
        \sintAveraged{ 
            \gamma \rvert_{ A_{i} }
        }{
            \del{
                g(x)
                -
                \sintAveraged{ 
                    \gamma \rvert_{ A_{i} }
                }{g}
            }
        }
        =
        \sintAveraged{ 
          \gamma \rvert_{ A_{i} }
        }{g}
        -
        \sintAveraged{ 
            \gamma \rvert_{ A_{i} }
        }{g}
        =
        0,
    \end{equation*}
    where the first inequality follows from the fact that the 
    integrand is a strictly positive function and we integrate it over a set 
    with positive measure.
    We obtain a contradiction, hence $x$ with the desired property exists --- 
    we then define $x_i \coloneqq x$.

    The fact that tuple $(x_i)_{i=0}^{m+1}$ satisfies the first two of the desired properties follows immediately from the construction.
    Let us prove the third property. 
    Fix $i \in [m]$. We then have 
    \begin{equation*}
        g(x_i)
        \leq
        \sintAveraged{ 
            \gamma \rvert_{ A_{i} }
        }{g}
        =
        \frac{1}{
             \mu_{\gamma\rvert_{ A_{i} }}^S\del{
            X}}
        \sint{ 
            \gamma \rvert_{ A_{i} }
        }{g}
        \le 
        \frac{1}{M}
        \sint{ 
            \gamma \rvert_{ A_{i} }
        }{g},
    \end{equation*}
    as needed.
    It remains to show the last property. 
    First, note that for all $i \in [m+1]$ there is some 
    $
        y_i 
        \in 
        \closure{ \image\del{\gamma \rvert_{A_i}}} 
        \cap 
        \closure{ \image\del{\gamma \rvert_{A_{i-1}}}}.
    $
    Indeed, when $i = 1$ or $i = m+1$ it follows from the fact that $A_i = A_{i-1}$,
    and when $i \in [m] \setminus \set{1}$ it follows from the fact that 
    $t_i$ is a point of continuity of $\gamma$, so 
    $
        y_i \coloneqq \gamma(t_i)
    $
    has the desired property.
    Next, note that for all $i \in [m]$ we have
    \begin{equation*}
        \diam\del{
            \closure{
                \image\del{ \gamma \rvert_{A_i}}
            }
        }
        =
        \diam\del{
            \image\del{ \gamma \rvert_{A_i}}
        }
        \leq
        V_\gamma(t_i^-) - V_\gamma(t_{i-1})
        \le 
        4M,
    \end{equation*}
    where in the first inequality we have used
    Corollary \ref{cor:functions_of_bounded_variation_image_totally_bounded} and Remark \ref{rem}.
    Hence, since $A_0 = A_1$ and $A_{m+1} = A_m$, 
    for all $i \in \set{0, \ldots, m+1}$, 
    we have 
    $
        \diam\del{
            \closure{
                \image\del{ \gamma \rvert_{A_i}}
            }
        }
        \le 
        4M.
    $    
    Therefore, 
    for all $i \in [m+1]$ we have
    \begin{equation*}
        \metric\del{
            x_i, x_{i-1}
        }
        \le 
        \metric\del{
            x_i, y_i
        }
        +
        \metric\del{
            y_i, x_{i-1}
        }
        \le 
        \diam\del{
            \closure{
                \image\del{ \gamma \rvert_{A_i}}
            }
        }
        +
        \diam\del{
            \closure{
                \image\del{ \gamma \rvert_{A_{i-1}}}
            }
        }
        \le 
        8M
    \end{equation*}
    and 
    $
        \metric\del{ x_i, x_{i-1}} \le 8M
    $
    as claimed.
     
    \textbf{Step 3} We will show the desired inequality.
    \vspace{10pt}
    Using the properties of tuple $(x_i)_{i=0}^{m+1}$, we have
    \begin{align*}
        \abs{ f(x_0) - f(x_{m+1})}
        &\le 
        \sum_{i=1}^{m+1}
            \abs{ f(x_{i-1}) - f(x_{i})}
        \\ 
        &\le 
        \sum_{i=1}^{m+1}
            \del{ g(x_{i-1}) + g(x_{i})}
            \metric\del{
                x_{i-1}, x_i
            }
        \\ 
        &\le 
        \sum_{i=1}^{m-1}
            \del{ g(x_{i-1}) + g(x_{i})}
            (8M)
        \\ 
        &=
        8Mg(x_0) + \sum_{i=1}^m 16M g(x_i) + 8Mg(x_{m+1})
        \\
        &\le 
        8Mg(x_0) + 
        \sum_{i=1}^m 
            \sint{ 
                \gamma \rvert_{ A_{i} }
            }{16 g} 
        + 8Mg(x_{m+1})
        \\
        &=
        8Mg(x_0) + 
        \sint{ 
                \gamma \rvert_{ [a, b^-] }
            }{16g} 
        + 8Mg(x_{m+1}),
    \end{align*}
    where in the last equality we have used Lemma \ref{lem:adding_symmetrized_integrals_over_intervals}
    and the fact that for $i \in [m-1]$ $t_i$ is a point of continuity of $\gamma$.
\end{proof}
\begin{lemma}
\label{lem::inequality_for_curves_for_lipshitz_case}
    Let $(X, \metricAlone)$ be a metric space. 
    Let $f \colon X \to \bR$ be measurable and  $g \colon X \to [0,\infty]$ be bounded Borel map such that 
    \begin{equation*}
        \forall 
        x, y \in X 
        \qquad 
        \abs{ f(x) - f(y) } 
        \le 
        \del{ g(x) + g(y) } \metric\del{ x, y}.
    \end{equation*}
        Then 
    \begin{equation*}
        \forall
            \gamma \in
            \testCurves\del{ [0,1]; X}
        \qquad 
        \abs{
            f(\gamma(0)) - f(\gamma(1))
        }
        \le 
        \sint{\gamma}{18g}.
    \end{equation*}
\end{lemma}
\begin{proof}
    Fix 
    $
        \gamma 
        \in
        \testCurves\del{ [0,1]; X}.
    $
   If $\gamma$ is continuous, then for all $s \in (0,1)$ and all $M > 0$
    we have $\phi_\gamma(s) \le M$. 
    Hence, by Lemma \ref{lem::bounding_difference_in_f_by_integral_with_endpoints}, for all $M > 0$ we have 
    \begin{equation*}
        \abs{ f(\gamma(0)) - f(\gamma(1)) }
        =
        \abs{ f(\gamma(0)) - f(\gamma(1^-)) }
        \le 
        8M g(\gamma(0)) 
        + \sint{\gamma}{16g}
        + 8M g(\gamma(1^-)).
    \end{equation*}
    Therefore, since $g$ is bounded,
    by passing to the limit $M \to 0^+$
    we obtain
    $
        \abs{ f(\gamma(0)) - f(\gamma(1)) }
        \le 
        \sint{\gamma}{16g}.
    $

    Suppose that $\gamma$ has points of discontinuity.
    By Lemma \ref{cor:variation_bounds_sum_of_jumps_from_above} we have that  $\sup_{\tau \in [0,1]} \phi_\gamma(\tau)$ is finite and $\set{
            t \in [0,1] 
            \given 
            \phi_\gamma(t) 
            =
            \sup_{\tau \in [0,1]} \phi_\gamma(\tau)
        }$ is a nonempty finite set. Therefore, the quantity 
    \begin{equation*}
        t_0
        \coloneqq
        \inf 
        \set{
            t \in [0,1] 
            \given 
            \phi_\gamma(t) 
            =
            \sup_{\tau \in [0,1]} \phi_\gamma(\tau)
        }
    \end{equation*}
    is well defined, $t_0$ is a point of discontinuity of $\gamma$ and $t_0 \in (0,1)$. 
    Next, for $n \in \bN$ let us define
    \begin{align*}
        t_{-n}
        &\coloneqq
    \begin{cases}
            \inf 
        \set{
            t \in \intcc{0,t_{-(n-1)}}
            \given 
            \phi_\gamma(t) 
            =
            \sup_{
                \tau \in \intco{0,t_{-(n-1)}}
            } \phi_\gamma(\tau)
        },
            &\text{ if } t_{-(n-1)} >0, \\
            0,
            &\text{ if } t_{-(n-1)} =0,
        \end{cases}
            \\
        t_{n}
        &\coloneqq
        \begin{cases}
           \sup
        \set{
            t \in \intcc{t_{n-1}, 1} 
            \given 
            \phi_\gamma(t) 
            =
            \sup_{
                \tau \in \intoc{t_{n-1}, 1}
            } \phi_\gamma(\tau)
        } ,
            &\text{ if } t_{(n-1)} <1, \\
            1,
            &\text{ if } t_{(n-1)} =1.
        \end{cases}
    \end{align*}
    We shall prove that
    \begin{equation}\label{o1}
        \abs{f(\gamma(t_0)) - f(\gamma(1))}
        \le 
        8 \phi_\gamma(t_0) g(\gamma(t_0))
            + 18\sint{\gamma \rvert_{\intcc[0]{t_0, 1}}}{g}
    \end{equation}
    and
    \begin{equation}\label{o2}
        \abs{f(\gamma(0)) - f(\gamma(t_0^-))}
        \le 
        8 \phi_\gamma(t_0) g(\gamma(t_0^-))+18\sint{\gamma \rvert_{\intcc[0]{0, t_0^-}}}{g}.
    \end{equation}
    The proof of (\ref{o2}) goes in the same manner as the proof of (\ref{o1}). Therefore, we give the detailed proof of (\ref{o1}).
    
First of all, let us observe that for any $m \in \mathbb{N}$ such that $t_{m-1}<1$ we have
    \begin{eqnarray} \label{op1}
       \sum_{i=1}^{m-1}
            \abs{f(\gamma(t_{i})) - f(\gamma(t_{i}^-))}
            &\le 2  \sint{\gamma \rvert_{\intcc[0]{t_{0}, 1^-}}}{g}.
    \end{eqnarray}
    Indeed, by our assumption and by Lemma \ref{lem:adding_symmetrized_integrals_over_intervals} we have
    \begin{align*}
       \sum_{i=1}^{m-1}
            \abs{f(\gamma(t_{i})) - f(\gamma(t_{i}^-))}
            &\le \sum_{i=1}^{m-1}
            (g(\gamma(t_{i})) + g(\gamma(t_{i}^-)))\metric\del{\gamma(t_{i}), \gamma(t_{i}^-)}\\
            &= 2\sum_{i=1}^{m-1} {\frac{g(\gamma(t_i^-)) + g(\gamma(t_i)) }{2} 
                    \del{
                        V_\gamma(t_i) - V_\gamma(t_i^-)
                                        }}\\
            &\le 2\sum_{i=1}^{m-1} \left(\sint{\gamma \rvert_{\intcc[0]{t_{i-1}, t_{i}^-}}}{g}
                    + \frac{g(\gamma(t_i^-)) + g(\gamma(t_i)) }{2} 
                    \del{
                        V_\gamma(t_i) - V_\gamma(t_i^-)
                                        }\right)\\
            &\qquad + 2  \sint{\gamma \rvert_{\intcc[0]{t_{m-1}, 1^-}}}{g}
            \\&=2  \sint{\gamma \rvert_{\intcc[0]{t_{0}, 1^-}}}{g}.
    \end{align*}
   
    \textbf{Step 1} 
    Proof of (\ref{o1}) in the case when there is $n \in \bN$ such that $t_n = 1$. 
    \vspace{10pt}

    Let $m \in \bN$ be the smallest such $n$. 
    Since $\gamma$ is left-continuous at $1$,
    the definition of $m$ implies that $\gamma$ has no points of discontinuity within $(t_{m-1}, t_m)$. 
    Let 
    $
        A_i \coloneqq \intoo{t_{i-1}, t_i}.
    $
    for $i \in [m]$. 
    Notice that for all $s \in A_i$ we have 
    $
        \phi_\gamma(s) 
        \le 
        \phi_\gamma(t_{i}).
    $
    In particular, for all $s \in A_m$ we have
    $\phi_\gamma(s) = 0$, as $\phi_\gamma(1) = 0$. 
    Fix $\eps \in \intoo{0, \phi_\gamma(t_{m-1})}$.

    By Lemma \ref{lem:adding_symmetrized_integrals_over_intervals}, Lemma \ref{lem::bounding_difference_in_f_by_integral_with_endpoints} and since $\phi_\gamma(t_i) \le \phi_\gamma(t_{i-1})$
    we have
    \begin{align*}
                \sum_{i=1}^m
            \abs{f(\gamma(t_{i-1})) - f(\gamma(t_{i}^-))}
        &\le 
        \sum_{i=1}^{m-1} \left(
            8 \phi_\gamma(t_i) g(\gamma(t_{i-1})) 
            + 16\sint{\gamma \rvert_{\intcc[0]{t_{i-1}, t_{i}^-}}}{g}
            + 8 \phi_\gamma(t_i) g(\gamma(t_i^-)) \right)
        \\
        &\qquad 
        +
            8 \eps g(\gamma(t_{m-1})) 
            + 16\sint{\gamma \rvert_{\intcc[0]{t_{m-1}, 1^-}}}{g}
            + 8 \eps g(\gamma(1^-)) 
        \\
        &\le
            8 \phi_\gamma(t_0) g(\gamma(t_0))
            + \sum_{i=1}^{m-1}
            \del{
                16\sint{\gamma \rvert_{\intcc[0]{t_{i-1}, t_{i}^-}}}{g}
                + 8 \phi_\gamma(t_i) g(\gamma(t_i^-))
                + 8 \phi_\gamma(t_i) g(\gamma(t_i))
            }
        \\
        &\qquad 
        +  16\sint{\gamma \rvert_{\intcc[0]{t_{m-1}, 1^-}}}{g}
            + 8 \eps g(\gamma(1^-))   
        \\
        &=
            8 \phi_\gamma(t_0) g(\gamma(t_0))
            + 16\left(
                \sum_{i=1}^{m-1}
                    \del{
                    \sint{\gamma \rvert_{\intcc[0]{t_{i-1}, t_{i}^-}}}{g}
                    + \frac{g(\gamma(t_i^-)) + g(\gamma(t_i)) }{2} 
                    \del{
                        V_\gamma(t_i) - V_\gamma(t_i^-)
                                        }}\right.
        \\
        &\qquad 
        +  \left. \sint{\gamma \rvert_{\intcc[0]{t_{m-1}, 1^-}}}{g} \right)
            + 8 \eps g(\gamma(1^-))  
        \\
        &=
            8 \phi_\gamma(t_0) g(\gamma(t_0))
            + 16\sint{\gamma \rvert_{\intcc[0]{t_0, 1^-}}}{g}
            + 8 \eps g(\gamma(1^-)).  
    \end{align*}
    Hence, by the triangle inequality and by (\ref{op1}) we get 
        \begin{align*}
        \abs{f(\gamma(t_0)) - f(\gamma(1^-))}
        &\le 
        \sum_{i=1}^m
            \abs{f(\gamma(t_{i-1})) - f(\gamma(t_{i}^-))} +
           \sum_{i=1}^{m-1}
            \abs{f(\gamma(t_{i})) - f(\gamma(t_{i}^-))}\\
        &\le 
        8 \phi_\gamma(t_0) g(\gamma(t_0))
            + 18\sint{\gamma \rvert_{\intcc[0]{t_0, 1^-}}}{g}
            + 8 \eps g(\gamma(1^-)).            
    \end{align*}
    By passing to the limit $\eps \to 0^+$ and using the left-continuity of $\gamma$ at $1$, we obtain
    \begin{equation*}
        \abs{f(\gamma(t_0)) - f(\gamma(1))}
        \le 
        8 \phi_\gamma(t_0) g(\gamma(t_0))
            + 18\sint{\gamma \rvert_{\intcc[0]{t_0, 1}}}{g}.
    \end{equation*}

    \textbf{Step 2}
    Proof of (\ref{o1}) in the case when there is no $n \in \bN$ such that $t_n=1$.
    \vspace{10pt}

    In this subcase, $(\phi_\gamma(t_n))_n$ is a strictly decreasing sequence. 
    By Lemma \ref{cor:variation_bounds_sum_of_jumps_from_above}
    we have $\phi_\gamma(t_n) \to 0$ as $n \to \infty$.
    For $i \in \bN$ let 
    $
        A_i \coloneqq \intoo{t_{i-1}, t_i}
    $
    and 
    $
        B_i \coloneqq \intoo{ t_i, 1}.
    $
    Notice that for all $s \in A_i$ we have 
    $
        \phi_\gamma(s) 
        \le 
        \phi_\gamma(t_{i})
    $
    and for all $s \in B_i$ we have
    $
        \phi_\gamma(s) 
        \le 
        \phi_\gamma(t_{i}).
    $
    By Lemma \ref{lem:adding_symmetrized_integrals_over_intervals}, Lemma \ref{lem::bounding_difference_in_f_by_integral_with_endpoints} and (\ref{op1})
    for all $n \in \bN$ we have
    \begin{align*}
        \abs{f(\gamma(t_0)) - f(\gamma(1^-))}
        &\le 
        \sum_{i=1}^{n-1}
            \abs{f(\gamma(t_{i-1})) - f(\gamma(t_{i}^-))}
        +
            \abs{
                f(\gamma(t_{n-1})) - f(\gamma(1^-))
            } + \sum_{i=1}^{n-1}
            \abs{f(\gamma(t_{i})) - f(\gamma(t_{i}^-))}
        \\
        &\le 
        \sum_{i=1}^{n-1}
            8 \phi_\gamma(t_i) g(\gamma(t_{i-1})) 
            + 16\sint{\gamma \rvert_{\intcc[0]{t_{i-1}, t_{i}^-}}}{g}
            + 8 \phi_\gamma(t_i) g(\gamma(t_i^-)) 
        \\
        &\qquad 
        +
            8 \phi_\gamma(t_n) g(\gamma(t_{n-1})) 
            + 16\sint{\gamma \rvert_{\intcc[0]{t_{n}, 1^-}}}{g}
            + 8 \phi_\gamma(t_{n-1}) g(\gamma(1^-)) + 2\sint{\gamma \rvert_{\intcc[0]{t_0, 1}}}{g}
        \\
        &\le
            8 \phi_\gamma(t_0) g(\gamma(t_0))
            + \sum_{i=1}^{n-1}
                16\sint{\gamma \rvert_{\intcc[0]{t_{i-1}, t_{i}^-}}}{g}
                + 8 \phi_\gamma(t_i) g(\gamma(t_i^-)) 
                + 8 \phi_\gamma(t_i) g(\gamma(t_i)) 
        \\
        &\qquad 
        +  16\sint{\gamma \rvert_{\intcc[0]{t_{n-1}, 1^-}}}{g}
            + 8 \phi_\gamma(t_{n-1}) g(\gamma(1^-)) +2\sint{\gamma \rvert_{\intcc[0]{t_0, 1}}}{g}
        \\
        &=
            8 \phi_\gamma(t_0) g(\gamma(t_0))
            + 16\left( \sum_{i=1}^{n-1}\del{
                \sint{\gamma \rvert_{\intcc[0]{t_{i-1}, t_{i}^-}}}{g}
                + \dfrac{g(\gamma(t_i^-)) + g(\gamma(t_i)) }{2} \del{V_\gamma(t_i)-V_\gamma(t_i^-)}        
                }\right.
        \\
        &\qquad 
        +  \left.\sint{\gamma \rvert_{\intcc[0]{t_{n-1}, 1^-}}}{g} \right)
            + 8 \phi_\gamma(t_{n-1}) g(\gamma(1^-))  +2\sint{\gamma \rvert_{\intcc[0]{t_0, 1}}}{g}
        \\
        &=
            8 \phi_\gamma(t_0) g(\gamma(t_0))
            + 18\sint{\gamma \rvert_{\intcc[0]{t_0, 1^-}}}{g}
            + 8 \phi_\gamma(t_{n-1}) g(\gamma(1^-)). 
    \end{align*}
    By passing to the limit $n \to \infty$ and using the left-continuity of $\gamma$ at $1$, we obtain
    \begin{equation*}
        \abs{f(\gamma(t_0)) - f(\gamma(1))}
        \le 
        8 \phi_\gamma(t_0) g(\gamma(t_0))
            + 18\sint{\gamma \rvert_{\intcc[0]{t_0, 1}}}{g},
    \end{equation*}
 and the proof of (\ref{o1}) follows.

 Now, we can finish the proof. Indeed, by (\ref{o1}) and (\ref{o2}) we have
    \begin{align*}
        \abs{f(\gamma(0)) - f(\gamma(1))}
        &\le 
        \abs{f(\gamma(0)) - f(\gamma(t_0^-))}
        + \abs{ f(\gamma(t_0^-)) - f(\gamma(t_0)) }
        + \abs{ f(\gamma(t_0)) - f(\gamma(1)) }
        \\
        &\le 
        18\sint{\gamma \rvert_{\intcc[0]{0, t_0^-}}}{g}
        + 8 \phi_\gamma(t_0) g(\gamma(t_0^-))
        + \frac{\phi_\gamma(t_0)}{2} 
        \del{
            g(\gamma(t_0^-)) + g(\gamma(t_0))
        }
        \\
        &\qquad 
        +8 \phi_\gamma(t_0) g(\gamma(t_0))
            + 18\sint{\gamma \rvert_{\intcc[0]{t_0, 1}}}{g}
        \\ 
        &\le 
        \sint{\gamma }{ 18g},
    \end{align*}
    where we have used 
     Lemma \ref{lem:adding_symmetrized_integrals_over_intervals}
    and the fact that 
    $
        \phi_\gamma(t_0) = V_\gamma(t_0) - V_\gamma(t_0^-).
    $
\end{proof}
Now we are in a position to continue the proof of Theorem \ref{dos}. Let $f'$ and $g'$ be like in Lemma \ref{lem:there_are_representatives_that_satisfy_Hajlasz_condition_everywhere} and put $\tilde{g}=g'$. It means that functions $f': X \rightarrow \mathbb{R}$
    and $\tilde{g} : X \rightarrow [0,\infty]$ are Borel such that $f=f'$ and $g=\tilde{g}$ $\measure$-almost everywhere
    such that 
    \begin{equation}\label{hg}
        \forall x, y \in X
        \qquad 
        \abs{
            f'(y) - f'(x)
        }
        \le 
        \del{\tilde{g}(x) + \tilde{g}(y) }
        \metric\del{
            x,y
        }.
    \end{equation}
    
    Let us consider the sets
    \begin{equation*}
        E_k 
        \coloneqq 
        \set{
            x \in X 
            \given 
            \tilde{g}(x) \le 2^k
        }.
        \end{equation*}
        
        Since $g$ is finite $\measure$-almost everywhere, then so is $\tilde{g}$. From (\ref{hg}) and the definition of $E_k$
        we have that $f' \rvert_{E_k}$ is $2^{k+1}$-Lipschitz.
        Therefore by the McShane's Lemma there exists function $f_k'$ which is a $2^{k+1}$-Lipschitz extension of $f' \rvert_{E_k}$. 
        This function has the following form     %
    \begin{equation*}
        f_k' (\cdot )
        \coloneqq 
        \inf_{ y \in E_k} 
            f'(y) +
            2^{k+1} \metric\del{ \cdot, y}.
    \end{equation*}
    Moreover, let us define 
    \begin{equation*}
            g_k'
        \coloneqq 
        \tilde{g} \indicator{ E_k}
        +
            2^{k+1}
        \indicator{ X \setminus E_k }.
    \end{equation*}
    Thus, for all $k \in \bN$ we have $g_k' \le 2\tilde{g}$.
    Also, since $\tilde{g}$ is finite $\measure$-almost everywhere and $$
        X \setminus \set{ x \in X \given \tilde{g}(x) = \infty }
        =
        \bigcup_{k=1}^\infty E_k,
    $$ we have  $f_k' \to f'$ $\measure$-almost everywhere.

    Moreover
    \begin{equation}\label{dwie}
        \abs{
            f_k'(x) - f_k'(z)
        }
        \le 
        \del{ g'_k(x) + g'_k(z) }
        \metric\del{ x,z}
    \end{equation}
    for all $x,z \in X$ and all $k \in \bN$.

    Indeed, if $x, z \in E_k$ then 
    \begin{equation*}
        \abs{
            f_k'(x) - f_k'(z)
        }=  \abs{
            f' (x) - f' (z) }
        \le 
        \del{ \tilde{g}(x) + \tilde{g}(z) }
        \metric\del{ x,z}
        =
        \del{ g_k'(x) + g_k'(z) }\metric\del{ x,z}.
    \end{equation*}
    If $x \notin E_k$ or $z \notin E_k$, 
    then 
    $
        g'_k(x) + g'_k(z) \ge 2^{k+1}
    $
    and since $f_k'$ is $2^{k+1}$-Lipschitz we have
    \begin{equation*}
        \abs{
            f_k'(x) - f_k'(z)
        }
        \le 
        2^{k+1}\metric\del{ x,z}
        \le 
        \del{ g'_k(x) + g'_k(z) }
        \metric\del{ x,z}.
    \end{equation*}
    Now, by (\ref{dwie}) we have that functions $f'_k$ and $g'_k$ satisfy the conditions of Lemma \ref{lem::inequality_for_curves_for_lipshitz_case}, so for all 
    $
        \gamma \in \testCurves\del{ [0,1]; X}
    $
    we have 
    \begin{equation} \label{jedna}
        \abs{
            f'_k(\gamma(0)) - f'_k(\gamma(1))
        }
        \le 
        \sint{\gamma}{ 18 g'_k}
        \le 
        \sint{\gamma}{ 36\tilde{g}}.
    \end{equation}
Next, we define  
    $
        \tilde{f} 
        \coloneqq
        \liminf_{k \to \infty}
        f'_k.
    $
    Clearly, $\tilde{f} = f$ and $\tilde{g} = g$ $\measure$-almost everywhere.
    
    Fix 
    $
        \gamma \in  \testMeasuresNonTrivial,
    $
    that is,
    $
        \gamma \in \testCurves\del{ [0,1]; X}
    $
    such that 
    $
        V(\gamma)> 0.
    $
    If $\sint{\gamma}{\tilde{g}} = \infty$,
    then 
    $$
        \abs{
            \tilde{f}(\gamma(0)) - \tilde{f}(\gamma(1))
        }
        \le 
        \sint{\gamma}{76\tilde{g}}.
    $$
    Therefore, we assume that 
    $\sint{\gamma}{\tilde{g}} < \infty$.

    {\bf Step 1} There exists $t \in (0,1)$ such that $ \tilde{g}(\gamma(t))$ and $ \tilde{g}(\gamma(t^-))$ are finite. 
    
    Since $\sint{\gamma}{\tilde{g}} < \infty$, we have
    \begin{eqnarray}\label{jeden}
        \mu_\gamma^S\del{
            X \setminus 
            \set{
                x \in X 
                \given 
                \tilde{g}(x) < \infty
            }
        }
        =
        0.
    \end{eqnarray}
    Moreover,
    \begin{equation}\label{dwa}        %
        \mu_\gamma^S\del{ \gamma\sbr{ (0,1) } }  > 0.
    \end{equation}
    Indeed, since $\gamma$ is left-continuous at $1$ we can write
    \begin{eqnarray*}
        \mu_\gamma^S\del{ \gamma\sbr{ (0,1) } }
        &\ge& 
        \frac{1}{2} \gamma_\# \mu_\gamma \del{
            \gamma\sbr{ (0,1) } 
        } 
        \ge
        \frac{1}{2} \mu_\gamma\del{ (0,1)  }\\
        &=&
        \frac{1}{2}\del{V_\gamma(1) - V_\gamma(0)}
        =
        \frac{1}{2}V_\gamma(1)
        =
        \frac{1}{2}V(\gamma) > 0.
    \end{eqnarray*}
    Therefore, gathering (\ref{jeden}) with (\ref{dwa}) we have 
    $$
        \mu_\gamma^S\del{
            \gamma\sbr{ (0,1) } 
            \cap 
            \set{
                    x \in X 
                    \given 
                    \tilde{g}(x) < \infty
                }
        }
        > 0,
    $$
    which implies the existence of $t \in (0,1)$ such that $\tilde{g}(\gamma(t)) < \infty$.
    
    Next, we shall prove that $\tilde{g}(\gamma(t^-)) < \infty$. If $\gamma(t) = \gamma(t^-)$, then 
    $
        \tilde{g}(\gamma(t^-)) = \tilde{g}(\gamma(t)) < \infty.
    $
    Therefore, we shall assume that 
    $\gamma(t) \ne \gamma(t^-)$. In this case we have\footnote{See the proof of Proposition \ref{prop::when_ae_0_gives_0_in_Newtonian}.}
     \begin{equation}\label{trzy}        %
        \mu_\gamma^S\del{ \{\gamma (t^-) \} }  > 0.
    \end{equation}
       Hence, by (\ref{jeden}) and (\ref{trzy}) we have that $\tilde{g}(\gamma(t^-)) < \infty$ and the proof of Step 1 is completed. 

    Since $\tilde{g}(\gamma(t)) < \infty $ and $\tilde{g}(\gamma(t^-)) < \infty $,
    there is $N \in \bN$ such that $\gamma(t), \gamma(t^-) \in E_N$.   Hence, for $k \ge N$ we have $\tilde{f}(\gamma(t)) = f'_k(\gamma(t))$ and $\tilde{f}(\gamma(t^-)) = f'_k(\gamma(t^-))$. In particular, $\tilde{f}(\gamma(t))$ and $\tilde{f}(\gamma(t^-))$ are finite.

    {\bf Step 2} For $t$ from Step 1 we have
    \begin{eqnarray}\label{cztery1}
                \abs{
            \tilde{f}(\gamma(t))-\tilde{f}(\gamma(1))
        }
        \le 
        \sint{ \gamma }{36\tilde{g}}
        \end{eqnarray}
        and
    \begin{eqnarray}\label{cztery2}
             \abs{
            \tilde{f}(\gamma(0))-\tilde{f}(\gamma(t^-))
        }
        \le 
        \sint{ \gamma }{36\tilde{g}}.
        \end{eqnarray}
    In order to prove (\ref{cztery1}) we define $\gamma' \colon [0,1] \to X$ 
    given by the formula
    $
        \gamma'(s) 
        \coloneqq 
        \gamma\del{
            t+ (1-t)s
        }.
    $
    We claim that 
    $
        \gamma' \in \testCurves\del{ [0,1]; X}
    $
    and 
    \begin{eqnarray}\label{piec1}
        \sint{ \gamma' }{\tilde{g}}
        \le 
        \sint{ \gamma }{\tilde{g}}.
    \end{eqnarray}
    Indeed, let us define $\Phi \colon [0,1] \to [t,1]$ as follows $\Phi(s) = t+ (1-t)s$, then
    $
    \gamma' = \Phi_{\#}(   \gamma \rvert_{ [t,1] }).
    $
    Since $\gamma \rvert_{ [t,1] } \in \testCurves\del{ [t,1]; X} $, by Proposition \ref{rem:test_Curves_changing_intervals}
    we have $
        \gamma' \in \testCurves\del{ [0,1]; X}
    $ and
    \begin{equation*}
        \sint{ \gamma' }{\tilde{g}}
        =
        \sint{ \Phi_{\#}( \gamma \rvert_{ [t,1] }) }{\tilde{g}}
        =
        \sint{\gamma \rvert_{ [t,1] }}{\tilde{g}}
        \le 
        \sint{ \gamma }{\tilde{g}},
    \end{equation*}
    where the last inequality follows from the fact that $\tilde{g} \ge 0$.
    
    Now, by (\ref{jedna}) and (\ref{piec1}) 
    for all $k \in \bN$ we have
    \begin{equation*}
        \abs{
            f'_k(\gamma(t)) - f'_k(\gamma(1))
        }
        =
        \abs{
            f'_k(\gamma'(0)) - f'_k(\gamma'(1))
        }
        \le 
        \sint{\gamma'}{36\tilde{g}}
        \le 
        \sint{ \gamma }{36\tilde{g}}
        .
    \end{equation*}
    Hence, for all $k \geq N$ we have
   $$
        f'_k(\gamma(1)) 
        \le 
        \sint{ \gamma }{36\tilde{g}}
        + \tilde{f}(\gamma(t)).
    $$
    Thus, 
    \begin{equation*}
        \tilde{f}(\gamma(1))
        =
        \liminf_{k \to \infty}
        f'_k(\gamma(1)) 
        \le 
        \sint{ \gamma }{36\tilde{g}}
        + \tilde{f}(\gamma(t))
        < 
        \infty.
    \end{equation*}
    On the other hand,
    a similar calculation shows 
    $$   
      -\infty <  \tilde{f}(\gamma(t))
        \le 
        \sint{ \gamma }{36\tilde{g}}
        + \tilde{f}(\gamma(1)).
    $$
    Note that those inequalities imply that 
    $
        \tilde{f}(\gamma(1))
    $
    is finite.
    As both $\tilde{f}(\gamma(t))$ and $\tilde{f}(\gamma(1))$
    are finite, we conclude  
    $$
        \abs{
            \tilde{f}(\gamma(t))-\tilde{f}(\gamma(1))
        }
        \le 
        \sint{ \gamma }{36\tilde{g}}
    $$
    and (\ref{cztery1}) follows. 

    Now, to prove (\ref{cztery2}) we define     
    $
        \gamma'' \colon [0,1] \to X
    $
    given by the formula 
   \begin{align*}
    \gamma''(s)
    \coloneqq
        \begin{cases}
           \gamma(st),
            &\text{ if } s \in [0,1), \\
            \gamma(t^-),
            &\text{ if } s=1.
        \end{cases}
    \end{align*}
        Then
    $
        \gamma'' \in \testCurves\del{ [0,1]; X}
    $
    and 
    \begin{eqnarray}\label{piec2}
        \sint{ \gamma'' }{ \tilde{g}} \le \sint{ \gamma }{ \tilde{g} }.
    \end{eqnarray}
    Indeed, let us define $\Psi \colon [0,1] \to [0,t]$ as follows $\Psi(s) = st$, then
    $
    \gamma'' = \Psi_{\#}(   \gamma \rvert_{ [0,t^-] }).
    $
    Since $\gamma \rvert_{ [0,t^-] } \in \testCurves\del{ [0,t]; X} $, by Proposition \ref{rem:test_Curves_changing_intervals}
    we have $
        \gamma'' \in \testCurves\del{ [0,1]; X}
    $ and
    \begin{equation*}
        \sint{ \gamma'' }{\tilde{g}}
        =
        \sint{ \Psi_{\#}( \gamma \rvert_{ [0,t^-] }) }{\tilde{g}}
        =
        \sint{\gamma \rvert_{ [0,t^-] }}{\tilde{g}}
        \le 
        \sint{ \gamma }{\tilde{g}}.
    \end{equation*}
    Now, by (\ref{jedna}) and by (\ref{piec2}) 
    for all $k \in \bN$ we have
    \begin{equation*}
        \abs{
            f'_k(\gamma(0)) - f'_k(\gamma(t^-))
        }
        =
        \abs{
            f'_k(\gamma''(0)) - f'_k(\gamma''(1))
        }
        \le 
        \sint{\gamma''}{36\tilde{g}}
        \le 
        \sint{ \gamma }{36\tilde{g}}
        .
    \end{equation*}
    Hence, for all $k \geq N$ we have
   $$
        f'_k(\gamma(0)) 
        \le 
        \sint{ \gamma }{36\tilde{g}}
        + \tilde{f}(\gamma(t^-)).
    $$
    Thus, 
    \begin{equation*}
        \tilde{f}(\gamma(0))
        =
        \liminf_{k \to \infty}
        f'_k(\gamma(0)) 
        \le 
        \sint{ \gamma }{36\tilde{g}}
        + \tilde{f}(\gamma(t^-))
        < 
        \infty.
    \end{equation*}
    On the other hand,
    a similar calculation shows that
    $$   
      -\infty <  \tilde{f}(\gamma(t^-))
        \le 
        \sint{ \gamma }{36\tilde{g}}
        + \tilde{f}(\gamma(0)).
    $$
    Since $\tilde{f}(\gamma(t^-))$ and $\tilde{f}(\gamma(0))$
    are finite, we have 
    $$
        \abs{
            \tilde{f}(\gamma(t^-))-\tilde{f}(\gamma(0))
        }
        \le 
        \sint{ \gamma }{36\tilde{g}}
    $$
    and (\ref{cztery2}) follows.

    Now we are in position to finish the whole proof. By Step 2, inequality (\ref{dwie}) and by Lemma \ref{lem:adding_symmetrized_integrals_over_intervals} for $k \geq N$
    we have
    \begin{align*}
        \abs{
            \tilde{f}(\gamma(0)) -\tilde{f}(\gamma(1))
        }
        &\le 
        \abs{
            \tilde{f}(\gamma(0)) - \tilde{f}(\gamma(t^-))
        }
        + \abs{
            \tilde{f}(\gamma(t^- )) - \tilde{f}(\gamma(t))
        }
        +\abs{
            \tilde{f}(\gamma(t)) - \tilde{f}(\gamma(1))
        }
         \\
        &\le 
        \sint{ \gamma }{36\tilde{g}}
        + \abs{ f'_k(\gamma(t^- )) - f'_k (\gamma(t)) }
        + \sint{ \gamma }{36\tilde{g}}
        \\
        &\le 
        \sint{ \gamma }{72\tilde{g}}
        + \del{ g'_k(\gamma(t^- )) + g'_k(\gamma(t)) }
        \metric\del{ \gamma(t^- ), \gamma(t) }
        \\
        &\le 
        \sint{ \gamma }{72\tilde{g}}+ 4\del{ \tilde{g}(\gamma(t^- )) + \tilde{g}(\gamma(t)) }\frac{\metric\del{ \gamma(t^- ),\gamma(t) }}{2}\\
        &\le 
        \sint{ \gamma }{76\tilde{g}}.
    \end{align*}
    Finally, let us observe that by replacing $\tilde{f}$ by $\tilde{f} \indicator{\{|\tilde{f}|<\infty\}}$ we ensure that our function has finite values and $\tilde{g}$ is an upper $S$-gradient of   $\tilde{f} \indicator{\{|\tilde{f}|<\infty\}}$. Indeed, it is sufficient to note that for all $x, y \in X$ we have
    \begin{equation*}
        \abs{
            \tilde{f} \indicator{\set[0]{|\tilde{f}|<\infty}}(x) - \tilde{f} \indicator{\set[0]{|\tilde{f}|<\infty}}(y)
        }
        \leq 
        \abs{
            \tilde{f}(x)-\tilde{f}(y)
        }.
    \end{equation*}
   If $x,y \in \set[0]{|\tilde{f}|<\infty},$ then the inequality is clear. Otherwise, the right hand side equals $\infty$, hence the inequality is also true.
\end{proof}
\section{Comparison of spaces} \label{sec::final_comparison}
%

\begin{theorem}
\label{thm::comparison_with_Gigli}
    Let $(X, \metricAlone)$ be a metric space 
    and let $\measure$ be a Borel regular measure on $X$.
    Then for any measurable functions $f \colon X \to \bRExtended$ and $g \colon X \to [0,\infty]$ such that $f, g$ are 
    finite $\measure$-almost everywhere we have:
    \begin{enumerate}[label=(\alph*)]
        \item 
        If $\measure$ is $\sigma$-finite and $g$ is a 
        is a Hajłasz gradient 
        of $f$, then there are Borel functions 
        $f': X \rightarrow \mathbb{R}$ and $g':X \rightarrow [0,\infty]$ such that $f = f'$ and $g = g'$ $\measure$-almost everywhere and $76 g'$ is an $\measure$-upper S-gradient of $f'$.
        \item 
        If $\measure$ is doubling and $f, g \in L^1_{\mathrm{loc}}(\measure)$ are Borel maps such that 
        $g$ is an $\measure$-upper S-gradient of $f$,
        then $g/2$ is a Hajłasz gradient of $f$.
    \end{enumerate}
\end{theorem}
\begin{proof}
    \vspace{-1em}
    \begin{enumerate}[label=(\alph*),itemindent=2em,leftmargin=0em, itemsep=0.5em]
        \item 
        By Theorem \ref{prop:almost_everywhere_finite_Hajlasz_gradients_are_upper_gradietnts} there are Borel maps 
         $f'\colon X \rightarrow \mathbb{R}$ and $g'\colon X \rightarrow [0,\infty]$ such that $f = f'$ and $g=g'$ $\measure$-almost everywhere and $76 g'$ is an upper S-gradient of $f'$.
        Therefore, by Remark \ref{r} we have $76 g'$ is an $\measure$-upper S-gradient of $f'$.
        \item 
        Under these conditions the Lebesgue differentiation theorem is satisfied \parencite[Theorem 1.8]{lectures-analysis-metric}. Fix $x, y \in X$ such that $x \neq y$ and $r \in (0, \metric\del{x, y}/2)$. We then define Borel measures $\mu_{x,r}$ and $\mu_{y,r}$ as follows
        \[
            \mu_{x,r}(A) := \fint_{ B(x, r) } \indicator{ A } \ \mathrm{d}\measure
            \quad \text{ and } \quad 
            \mu_{y,r}(A) := \fint_{ B(y, r) } \indicator{ A } \ \mathrm{d}\measure.
        \]
        Next, let $\mu'$ be a Borel probability measure on $X \times X$ defined by
        \[
            \mu' =\mu_{x,r} \otimes \mu_{y,r}.
        \]
        In particular, for Borel sets $A, C \subseteq X$ we have
        \begin{equation*}
          \mu'(A \times C)
            = \mu_{x,r}(A) \mu_{y,r} (C).
        \end{equation*}
        Next, for $w, z \in X$ we define function $\gamma_w^z \colon [0,1] \to X$ by
        the formula
        \begin{equation*}
            \forall t \in [0,1]
            \qquad 
            \gamma_w^z (t) 
            \coloneqq 
            w \indicator{ [0,1/2 ) }(t)
            + z \indicator{ \intcc[0]{ 1/2, 1} }(t).
        \end{equation*}
        By Example  \ref{e1}, $\gamma_w^z \in \testCurves\del{ [0,1] ; X}$. Therefore, $h \colon X^2 \to \testCurves\del{
            [0,1]; X
        }$ defined by $h(w, z) \coloneqq \gamma_w^z$
        for all $w, z \in X$ is well-defined. Moreover, it is continuous, hence Borel.
        Now, let $\mu \coloneqq h_\# \mu'$. 
        Then $\mu$ is a Borel probability measure on 
        $\testCurves\del{
            [0,1]; X
        }$.

        Let us note that for all $t \in [0,1]$ and every Borel set $A \subset X$ we have
        \begin{equation*}
           (e_t)_\# \mu (A)
            =
            \mu_{x,r}(A)
            \indicator{ [0, 1/2 ) }(t)
            +
            \mu_{y,r}(A)
            \indicator{ \intcc[0]{ 1/2, 1} }(t)
            \le 
            C \measure(A),
        \end{equation*} 
        where
        $
            C 
            =  
            \max\del{ 
                \measure\del{ B(x, r) }^{-1} ,
                \measure\del{ B(y, r) }^{-1}
            }
        $.
        Moreover,
        \begin{multline*}
            \integral{
                \testCurves\del{ [0,1]; X}
            }{
                V(\gamma)
            }{
                \mu(\gamma)
            }
            =
            \integral{
                \testCurves\del{ [0,1]; X}
            }{
                V(\gamma)
            }{
                h_\#(\mu')(\gamma)
            }
            \\
            =
            \integral{X^2}{
                \metric\del{
                    w, z
                }
            }{
                \mu'(w, z)
            }
            \le 
            \integral{X^2}{
                r + \metric\del{x, y} + r
            }{
                \mu'(w, z)
            }
            \\
            =
            2r + \metric\del{x, y}             < \infty.
        \end{multline*}
        Therefore, $\mu \in  \mathcal{P}^{(\measure)}(X )$. Furthermore, we have
        \[
          \mu(\testCurves\del{ [0,1]; X } \setminus \testCurvesNonTrivial )= \mu'(\Delta),
        \]
        where $\Delta =\set{(z,z) \in X^2 \given z \in X}$. 
        Now, recall that $r < \metric\del{x, y}/2$, so
        $\Delta \subseteq X^2 \setminus B(x,r) \times B(y,r)$, hence
        \begin{align*}
            \mu'( \Delta )
            &\le 
            \mu'\del{
                X^2 \setminus B(x,r) \times B(y,r)
            }            \\
            &=\mu_{x,r}(X)\mu_{y,r}(X) - \mu_{x,r}(B(x,r))\mu_{y,r}(B(y,r))
                        \\
            &=
            1-1 = 0.
        \end{align*} 
         In this way we have $\mu \in \testMeasuresNonTrivial^{(\measure)}$.     
        
        Now, since $g$ is an $\measure$-upper S-gradient of $f$, 
        we have
        \begin{align*}
            \abs{
                \fint_{ B(x, r) } f \ \mathrm{d}\measure 
                - \fint_{ B(y, r) } f \ \mathrm{d}\measure 
            }
            &=
            \abs{
                \integral{X^2}{ f(w) }{ \mu'(w,z)} 
                - \integral{X^2}{ f(z) }{ \mu'(w,z)} 
            }
            \\
            &\le 
            \integral{X^2}{
                \abs{
                 f(w) - f(z) }
                 } { \mu'(w,z)} 
            \\
            &=
            \integral{
               \testCurves\del{ [0,1]; X}
            }{
                \abs{
                    f(\gamma(1)) - f(\gamma(0))
                }
            }{
                h_\#(\mu')(\gamma)
            }
            \\
            &=
            \integral{
                \testCurves\del{ [0,1]; X}
            }{
                \abs{
                    f(\gamma(1)) - f(\gamma(0))
                }
            }{
                \mu(\gamma)
            }
            \\
            &\le 
            \integral{
                \testCurves\del{ [0,1]; X}
            }{
                \enskip \sint{ \gamma }{ g } \enskip 
            }{
                \mu( \gamma)
            }
            \\
            &=
            \integral{
                \testCurves\del{ [0,1]; X}
            }{
                \enskip \sint{ \gamma }{ g } \enskip 
            }{
                h_\# \mu'( \gamma)
            }
            \\
            &=
            \integral{
                 X^2
            }{
                \frac{
                    g(w) + g(z)                
                }{2}
                \metric\del{
                    w, z
                }
            }{
                \mu'(w, z)
            }
            \\
            &=
            \integral{
                 X^2
            }{
                \frac{ g(w) }{ 2}
                \metric\del{
                    w, z
                }
            }{
                \mu'(w, z)
            }
            +
            \integral{
                X^2
            }{
                \frac{ g(z) }{ 2}
                \metric\del{
                    w, z
                }
            }{
                \mu'(w, z)
            }
            \\
            &\le 
            \integral{
                X^2
            }{
                \frac{ g(w) }{ 2}
                \del{
                    r + \metric\del{x, y} + r
                }
            }{
                \mu'(w, z)
            }
            +
            \integral{
                 X^2
            }{
                \frac{ g(z) }{ 2}
                \del{ r + \metric\del{x, y} + r}
            }{
                \mu'(w, z)
            } 
            \\
            &=
            \frac{ 2r + \metric\del{x, y} }{2}
            \del{
               \fint_{ B(x, r) } g \ \mathrm{d}\measure
               + \fint_{ B(y, r) } g \ \mathrm{d}\measure
            }.
        \end{align*}
        Therefore,
        \begin{equation*}
            \abs{
                \fint_{ B(y, r) } f \ \mathrm{d}\measure 
                - \fint_{ B(x, r) } f \ \mathrm{d}\measure 
            }
            \le 
            \frac{\del{ 2r + \metric\del{x, y} }}{2}
            \del{
               \fint_{ B(x, r) } g \ \mathrm{d}\measure
               + \fint_{ B(y, r) } g \ \mathrm{d}\measure
            }.
        \end{equation*}
        The above inequality is true for any $x, y \in X$ and $r \in (0, \metric\del{x, y}/2)$.
        Since $\mu$ is Borel regular and doubling, 
        and $f, g \in L^1_{\mathrm{loc}}(\measure)$,
        for $\measure$-almost all $z \in X$ we have
        \begin{equation*}
            \fint_{ B(z, \eps) } f \ \mathrm{d}\measure
            \xrightarrow{ \eps \to 0^+}
            f(z)
            \quad \text{ and } \quad 
            \fint_{ B(z, \eps) } g \ \mathrm{d}\measure
            \xrightarrow{ \eps \to 0^+}
            g(z).
        \end{equation*}
        Therefore, by passing to the limit $r \to 0^+$ for $\measure$-almost all $x, y \in X$ we have
        \begin{equation*}
            \abs{
               f(x) - f(y)
            }
            \le 
            \frac{\metric\del{x, y}}{2}
            \del{ g(x) + g(y) }
        \end{equation*}
        and $g/2$ is a Hajłasz gradient of $f$.
    \end{enumerate}
\end{proof}
Finally, let us make a comparison between the various definitions of First Order Sobolev spaces.%
\begin{theorem}\label{thm::final_comparison}
    Let $(X, \metricAlone)$ be a metric space $\measure$ be a Borel measure and $p\in [1, \infty)$, then:
    \begin{enumerate}
        \item 
        $
           \mathcal{N}_{\hat{TC}}^{1,p}(X, \metricAlone, \measure) \hookrightarrow  N^{1,p}_{TC}(X, \metricAlone, \measure)
            \stackrel{{\sim}}{\hookrightarrow}
            M^{1,p}(X, \metricAlone, \measure)
        $\footnote{Here, symbol $\stackrel{\sim}{\hookrightarrow}$ is used to indicatate that the ``embedding'' is  not necessarily injective; 
        however, its kernel consists of equivalence classes in which functions are equal to $0$ $\measure$-almost everywhere}     
        and
        $
           \mathcal{N}_{\hat{TC}}^{1,p}(X, \metricAlone, \measure) \hookrightarrow  
            M^{1,p}(X, \metricAlone, \measure),
        $

        \item 
       
        If $\measure$ is Borel regular, then
        \begin{equation*}
           \mathcal{N}_{\hat{TC}}^{1,p}(X, \metricAlone, \measure) \hookrightarrow  N^{1,p}_{TC}(X, \metricAlone, \measure)
            \hookrightarrow
            M^{1,p}(X, \metricAlone, \measure),
         \end{equation*}
                \item 
        If $\measure$ is $\sigma$-finite and Borel regular, then
        \begin{equation*}
            \mathcal{N}_{\hat{TC}}^{1,p}(X), \metricAlone, \measure) \cong  N^{1,p}_{TC}(X, \metricAlone, \measure) \cong M^{1,p}(X, \metricAlone, \measure)
             \hookrightarrow G^{1,p}(X, \metricAlone, \measure),
         \end{equation*}
        \item 
        If $\measure$ is doubling and Borel regular, then
        \[
            \mathcal{N}_{\hat{TC}}^{1,p}(X), \metricAlone, \measure) \cong  N^{1,p}_{TC}(X, \metricAlone, \measure) \cong M^{1,p}(X, \metricAlone, \measure)
             \cong G^{1,p}(X, \metricAlone, \measure).
        \]
    \end{enumerate}
\end{theorem}
\begin{proof}
    \vspace{-1em}
    \begin{enumerate}[itemindent=2em,leftmargin=0em, itemsep=0.5em]
        \item 
            By Theorem \ref{lem:comparison_between_various_upper_gradients} and Corollary \ref{cor::TC_hat_Newtonian_0_ae} we have $\mathcal{N}_{\hat{TC}}^{1,p}(X, \metricAlone, \measure) \hookrightarrow  N^{1,p}_{TC}(X, \metricAlone, \measure)$ and by Theorem \ref{hn} we have $N^{1,p}_{TC}(X, \metricAlone, \measure)
                   \stackrel{{\sim}}{\hookrightarrow}
                    M^{1,p}(X, \metricAlone, \measure)$.
            We have $\mathcal{N}_{\hat{TC}}^{1,p}(X, \metricAlone, \measure) \hookrightarrow M^{1,p}(X, \metricAlone, \measure)$ since by Corollary \ref{cor::TC_hat_Newtonian_0_ae} the composition of the previous two embeddings has a trivial kernel, and hence is injective.

        \item 
            follows from $(1)$ and Proposition \ref{prop::when_ae_0_gives_0_in_Newtonian}.

        \item 
            By Theorem \ref{ro} we have $ N^{1,p}_{TC}(X, \metricAlone, \measure) \cong M^{1,p}(X, \metricAlone, \measure)$. Theorem \ref{hn}, Theorem \ref{lem:comparison_between_various_upper_gradients} and Corollary \ref{cor::TC_hat_Newtonian_0_ae} give 
            $M^{1,p}(X, \metricAlone, \measure) \hookrightarrow \mathcal{N}_{\hat{TC}}^{1,p}(X, \metricAlone, \measure)$, and by Theorem \ref{thm::comparison_with_Gigli} 
            and Proposition \ref{prop::f=0_ae_and_in_Gigli_means_it_is_0_in_Gigli}
            we obtain $ M^{1,p}(X, \metricAlone, \measure)
                     \hookrightarrow G^{1,p}(X, \metricAlone, \measure)$.

        \item  
            It follows from $(3)$, Theorem \ref{thm::comparison_with_Gigli} and Proposition \ref{prop::f=0_ae_and_in_Gigli_means_it_is_0_in_Gigli}.
    \end{enumerate}
\end{proof}

{\bf Acknowledgement} We wish to thank Timo Schultz for an insightful question that led to an improvement in the paper.
\medskip

%
%
\printbibliography
\bigskip
{\small Przemys{\l}aw  G\'orka}\\
\small{Faculty of Mathematics and Information Science,}\\
\small{Warsaw University of Technology,}\\
\small{Pl. Politechniki 1, 00-661 Warsaw, Poland} \\
{\tt przemyslaw.gorka@pw.edu.pl}\\
\\

{\small Kacper Kurowski}\\
\small{Faculty Mathematics and Information Science,}\\
\small{Warsaw University of Technology,}\\
\small{Pl. Politechniki 1, 00-661 Warsaw, Poland} \\
{\tt kacper.kurowski.dokt@pw.edu.pl}\\
\end{document}